\documentclass[12pt,reqno]{amsart} 

\usepackage[T1]{fontenc}
\usepackage[utf8]{inputenc}

\usepackage{amsmath,amssymb,amsthm}
\usepackage{mathrsfs,bbm,amsfonts}
\usepackage{mathtools,latexsym}

\usepackage{graphicx,xcolor}
\usepackage{float}

\usepackage{enumerate}
\usepackage{cite}
\usepackage[margin=3cm, a4paper]{geometry}

\usepackage{microtype} 
\allowdisplaybreaks[2]

\usepackage[colorlinks=true,backref]{hyperref}

\usepackage{cite}
\usepackage[margin=3cm, a4paper]{geometry}

\newcommand{\EQ}[1]{\begin{align*}\begin{split} #1 \end{split}\end{align*}}
\newcommand{\EQn}[1]{\begin{align}\begin{split} #1 \end{split}\end{align}}

\def\HH{{\mathcal H}}
\def\F{\mathscr{F} }
\def\N{\mathbb{N}}
\def\R{\mathbb{R}}
\def\Z{\mathbb{Z}}

\def\C{\mathbb{C}}

\def\mS{\mathcal{S}}
\def\mW{\mathcal{W}}
\def\H2{H^2(\R)}
\def\L2{L^2(\R)}
\def\to{\rightarrow}

\def\wh{\widehat}

\newcommand{\dx}{\,\mathrm{d}x}
\newcommand{\dy}{\,\mathrm{d}y}
\newcommand{\ds}{\,\mathrm{d}s}
\newcommand{\dt}{\,\mathrm{d}t}
\newcommand{\dta}{\,\mathrm{d}\tau}
\newcommand{\dxi}{\,\mathrm{d}\xi}
\newcommand{\dxio}{\,\mathrm{d}{\xi}_{1}}
\newcommand{\deta}{\,\mathrm{d}\eta}

\def\norm#1{\left\|#1\right\|}

\def\absb#1{\big|#1\big|}

\def\brk#1{\left(#1\right)}

\def\mbrkb#1{\big[#1\big]}

\def\fbrk#1{\left\lbrace#1\right\rbrace}

\def\jb#1{\langle#1\rangle} \def\norm#1{\|#1\|}

\def\H1{H^1(\R)}

\newcommand{\les}{\lesssim} \newcommand{\ges}{\gtrsim}
\newcommand{\wt}{\widetilde}

\newcommand{\al}{\alpha} 
\newcommand{\ga}{\gamma}

\newcommand{\pd}{\partial}
\newcommand{\De}{\Delta}

\newcommand{\p}{\partial} \newcommand{\na}{\nabla}
\newcommand{\re}{\mathop{\mathrm{Re}}}
\newcommand{\im}{\mathop{\mathrm{Im}}}

   \newcommand{\I}{\infty}

\newcommand{\ra}{\rightarrow}

 \newcommand{\Del}[1]{}

\numberwithin{equation}{section}

\newtheorem{thm}{Theorem}[section]
\newtheorem{cor}[thm]{Corollary}
\newtheorem{lem}[thm]{Lemma}

\newtheorem{prop}[thm]{Proposition}
\newtheorem{assu}[thm]{Assumption}
\newtheorem{definition}[thm]{Definition}
\theoremstyle{remark}
\newtheorem{remark}[thm]{Remark}
\newtheorem*{exam*}{Examples}

\title[Zakharov]{The subsonic limit of the 3D Zakharov system}

\author{Rui Jia}
\address{(R. Jia) Center for Applied Mathematics\\
	Tianjin University\\
	Tianjin 300072, China}
\email{jiarui0305@tju.edu.cn}

\author{Jia Shen}
\address{(J. Shen) School of Mathematical Sciences and LPMC\\
	Nankai University\\
	Tianjin 300071, China}
\email{shenjia@nankai.edu.cn}

\author{Yifei Wu}
\address{(Y. Wu)  School of Mathematical Sciences\\
	Nanjing Normal University\\
	Nanjing 210046, China}
\email{yerfmath@gmail.com}

\begin{document}
	
\begin{abstract}
We obtain the optimal convergence rates in the subsonic limit of the three-dimensional Zakharov system for initial data belonging to the low-regularity Sobolev space
$\HH^s=H^s\times H^{s-1}\times H^{s-1}$. For the Schr\"odinger component,  we prove first-order convergence in $L^2$ for initial data in $\HH^3$, and second-order convergence under the compatibility condition for data in $\HH^4$. For the wave component, we obtain first-order convergence in $L^2$ for data in $\HH^3$ and second-order convergence for data in $\HH^4$.

The obtained rates are optimal and coincide with those predicted by the formal asymptotic expansion. No localization assumptions, smallness or high-order regularity hypotheses are required. This improves all previous results on the subsonic limit of the Zakharov system and resolves the optimality issue at the Sobolev regularity level.

The proof relies on a uniform local well-posedness theory that remains valid in the subsonic limit. A key ingredient is a refined normal form analysis combined with bilinear Strichartz estimates in atomic function spaces, which allows us to fully exploit the dispersive structure of the Zakharov system at low regularity and to overcome the derivative losses arising from the singular coupling. 

%
%
%

	\end{abstract}
	
	\maketitle

	\tableofcontents	
	\section{Introduction}
	We consider the three-dimensional Zakharov system: 
	\begin{align}
		\label{equ-Z-al}\tag{$\text{Zak}$}
		\left\{
		\begin{aligned}
			&2i\partial_{t}u^{(\al)}-\Delta u^{(\al)}=-n^{(\al)}u^{(\al)}, \\
			&\al^{-2}\partial_{tt}n^{(\al)}-\Delta n^{(\al)}= \Delta|u^{(\al)}|^2, \\
			& u^{(\al)}(0,x)=u_0^{(\al)}, \quad n^{(\al)}(0,x)=n_0^{(\al)}, \quad \partial_tn^{(\al)}(0,x)=n_1^{(\al)},
		\end{aligned}
		\right.
	\end{align}
	where $u^{(\al)}:\R^{1+3}\rightarrow\C$ is a complex-valued function representing the varying envelope of a highly oscillatory electric field, $ n^{(\al)}:\R^{1+3}\rightarrow\R$ is a real-valued function describing the fluctuation of the plasma ion density from its equilibrium state, and $\alpha>0$ is the ion sound speed. It is well known that the system \eqref{equ-Z-al} preserves mass and energy:
	\begin{align}
		\label{mass}
		M(t)&\coloneqq \|u^{(\al)}(t,\cdot)\|_{L_x^2}=M(0),
		\\
		\label{energy}
		E(t)&\coloneqq  \int_{\R^3}|\nabla u^{(\al)}(t)|^2+\frac{1}{2}|n^{(\al)}(t)|^2+\frac{1}{2}||\al\na|^{-1}\partial_t n^{(\al)}(t)|^2+n^{(\al)}(t)|u^{(\al)}(t)|^2 \dx = E(0).
	\end{align} 
	
	The Zakharov system \eqref{equ-Z-al} was introduced by Zakharov \cite{zakharov1972} as a classical model for Langmuir turbulence in an unmagnetized, completely ionized plasma, based on a two-fluid description of plasma dynamics. It couples the ion acoustic wave equation with the evolution equation for the slowly varying envelope of the Langmuir electric field through nonlinear interaction terms. In suitably scaled variables, the system can be derived from Maxwell's equations together with linearized hydrodynamic equations. A derivation based on a Lagrangian formalism was later given by Gibbons, Thornhill, Wardrop, and ter Haar \cite{Gibbons_Thornhill_Wardrop_Haar_1977}.

	In the subsonic limit $\alpha \rightarrow \infty$, the system \eqref{equ-Z-al} formally reduces to the focusing cubic nonlinear Schr\"odinger equation:
	\begin{align}
	\label{equ-NLS}\tag{NLS}
		\left\{
		\begin{aligned}
			&2i\partial_{t}v-\Delta v=|v|^2v,  \\
			& v(0,x)\coloneqq v_0.
		\end{aligned}
		\right.
	\end{align}
Physically, this corresponds to the plasma responding almost instantaneously to variations in the electric field. Mathematically, the subsonic limit poses significant challenges, requiring precise control of nonlinear interactions and careful handling of the differing scales of dispersive and acoustic components.

Schochet and Weinstein \cite{Weinstein1986} first proved the convergence of the subsonic limit in 2D and 3D rigorously. For initial data $(u_0^{(\al)},n_0^{(\al)},|\nabla|^{-1}n_1^{(\al)})\in H^{m+1}\times H^m\times H^{m}(\R^d)$ with high regularity $m\ge[d/2]+3$, they prove that under the convergence condition
\EQ{
u_0^{(\al)}\ra v_0\text{ in }H^{m+1}; \quad \|\na(n_0^{(\al)}+|u_0^{(\al)}|^2)\|_{H^{m-1}}\les\al^{-1},
} then $(u^{(\al)},n^{(\al)}+|u^{(\al)}|^2)\ra(v,0)$ as $\al\ra\I$, where $v$ is the solution of \eqref{equ-NLS}.
	
Next, the explicit convergence rate on $\al$ was first identified by Added and Added \cite{Added1988} for dimensions $d=1,2,3$. They also made the initial-layer correction explicit. More precisely, when
\EQ{
n_0+|u_0|^2=0,
}
which is the so-called compatibility condition on the initial data, one expects
\EQ{
n^{(\al)}+|u^{(\al)}|^2\ra0.
}
In contrast, when the initial data is incompatible, namely 
\EQ{
n_0+|u_0|^2\ne0,
}
then one should instead expect 
\EQ{
n^{(\al)}+|u^{(\al)}|^2 - Q\ra 0,
}
where the fitting corrector $Q$ denotes the solution of some linear wave equation. Under high-regularity Sobolev assumptions, together with a smallness condition on the initial data in the two- and three-dimensional cases, they proved that for incompatible initial data, the convergence rate of both Schr\"odinger and wave components 
\EQ{
(u^{(\al)}-v,n^{(\al)}+|u^{(\al)}|^2-Q)
}
converge at rate $\al^{-\frac12}$ when $d=1,2$, and at rate $\al^{-1}\log(\al)$ when $d=3$. For compatible initial data, the convergence rate improves to $\al^{-1}$. However, their paper did not determine whether the rates  $\al^{-1}\log(\al)$ and $\al^{-1}$ are optimal.

A natural question arising from the above results concerns the optimality of the convergence rates. The first systematic investigation in this direction was carried out by Ozawa and Tsutsumi\cite{Ozawa1992}. One of their key observations is that the Schr\"odinger and wave components exhibit fundamentally different convergence behaviors. More precisely, they observed that the two convergences
$$
u^{(\al)}\to v \quad \mbox{and } \quad  n^{(\al)}+|u^{(\al)}|^2\to Q \quad (\mbox{for some  fitting corrector}),
$$
exhibit different behavior. For Schwartz initial data, without any smallness assumption, and under the additional condition $n_1^{(\al)}\in \dot H^{-1}$ in dimensions $d=1,2,3$, they proved that the convergence rate for the Schr\"odinger component $u^{(\al)}\to v$ is $\al^{-1}$ in the incompatible case and improves to $\al^{-2}$ in the compatible case. Moreover, both rates were shown to be optimal.  While for the wave component,  $n^{(\al)}+|u^{(\al)}|^2\to Q$, regardless of whether the data are compatible or incompatible, the convergence rate was proved as  $\al^{-1}$ in $L^2$ and $\al^{-2}$ in some local-in-spatial space  
$$
H^{m,-s}=\{f\in \mathcal S': \brk{1+|x|^2}^{-s/2}\brk{1-\De}^{m/2}f\in L^2 \}\quad \mbox{for some large } m, s>0.
$$
Moreover, the $\al^{-2}$-rate was proved to be optimal.  However, it remains an open problem whether the optimal $\al^{-2}$ convergence rate also holds in the global $L^2$ topology. Later Kenig, Ponce, and Vega \cite{Kening1995} removed the auxiliary assumption  $n_1^{(\al)}\in \dot H^{-1}$ and established the optimal convergence rates for the  Schr\"odinger  component under suitable high-regularity, spatial decay, and smallness assumptions on the initial data. 

In summary, previous optimal-rate results were proved either for Schwartz data (see \cite{Ozawa1992}), or for high-regularity weighted Sobolev data (see \cite{Kening1995}). However, without quantitative rates, the convergence in the subsonic limit for Zakharov system can be achieved at very low regularity. Masmoudi and Nakanishi \cite{Masmoudi2005} proved that
\EQ{
(u^{(\al)}-v, n^{(\al)} + |u^{(\al)}|^2-Q^{(1)}-Q^{(2)})\ra 0\quad\text{ in }H^s \times H^{s-1}\text{ with }s>\frac32,
}
on the maximal existence time of the limiting Schr\"odinger equation \eqref{eq:NLS}, provided the initial data satisfies $\|P_{\al}(n_0^{(\al)},|\al\na|^{-1}n_1^{(\al)})\|_{H^{s-1}}\rightarrow0$. Afterwards, they \cite{Masmoudi2008} further improved the regularity to energy-level $H^1 \times L^2$, provided that
\EQ{
	\lim_{R\ra\I}\lim\sup_{\al\rightarrow\infty}\|P_{>R}(n_0^{(\al)},|\al\na|^{-1}n_1^{(\al)})\|_{L^2}=0.
}
Notice that in these qualitative results, the convergence of the initial data and the convergence of the solutions are measured at essentially the same Sobolev regularity level. However, when studying the quantitative rate of convergence, see \cite{Added1988,Ozawa1992,Kening1995}, the initial data require more regularity than the solution.

Moreover, the limit for Zakharov-type systems and other related models has also been extensively investigated, see \cite{Masmoudi2005, Masmoudi2008, Masmoudi2010,masmoudi2003nonrelativistic,machihara2001nonrelativistic, machihara2002nonrelativistic, masmoudi2002nonlinear,Lu2025} for instance. 
	

Despite substantial progress, it remains unknown whether optimal convergence rates for the subsonic limit can be established at the purely Sobolev level. More specifically, existing optimal-rate results rely either on weighted spaces, local-in-space norms for the wave component, or additional decay and smallness assumptions. The main goal of this paper is to remove all of these restrictions in the three-dimensional setting.
	
	Before stating our main result, we give some basic settings. For simplicity, we omit the superscript \(\alpha\) from the corresponding solution and write
	\begin{align}
		\label{equ-Z}\tag{$\text{Zak}_\al$}
		\left\{
		\begin{aligned}
			&2i\partial_{t}u-\Delta u=-nu, \\
			&\al^{-2}\partial_{tt}n-\Delta n= \Delta|u|^2, \\
			& u(0,x)=u_0, \quad n(0,x)=n_0, \quad \partial_tn(0,x)=n_1.
		\end{aligned}
		\right.
	\end{align} 
	We also consider the relationship between the convergence rate and the incompatibility condition $n_0+|u_0|^2\neq 0$. A fitting corrector is given by the function $w(t,x)$, where $w$ is the solution of the wave equation
		\begin{align}
		\label{equ-w}\tag{Wave}
		\left\{
		\begin{aligned}
			&\al^{-2}\partial_{tt}w-\Delta w= 0,  \\
			& w(0) \coloneqq w_0= n_0+|u_0|^2, \quad \pd_tw(0) \coloneqq w_1= n_1+\partial_t|u|^2(0).
		\end{aligned}
		\right.
	\end{align}
	Throughout this paper, we use the  notation:
	\EQn{\label{defn:initial-data}
		\HH^s\coloneqq & H^s \times H^{s-1}\times H^{s-1}(\R^3)\text{, for every $s\ge1$};\\
		\vec u_0\coloneqq  &(u_0, n_0, |\al\nabla|^{-1}n_1);\\\vec u(t)\coloneqq  & (u(t), n(t), |\al\nabla|^{-1}n_t(t)).
	}

Now, we first state the uniform local theory for solutions to the Zakharov system.
\begin{prop}[Uniform local theory of \eqref{equ-Z}]
\label{prop:lwpZ}
Let $l \geq 2$ be an integer. For any initial data $\vec u_0 \in \HH^{l+1}$, there exists $T$ that only depends on $\left\| \vec u_0 \right\|_{\HH^3}$ such that the Zakharov system \eqref{equ-Z} admits a unique solution  $\vec u \in C\big([0,T];\HH^{l+1}\big)$. Moreover, the solution satisfies the following estimate:
\begin{align}
\label{lwpZ}
\| \vec u\|_{L_t^\I([0,T];\HH^{l+1})}   \leq C\|\vec u_0\|_{\HH^{l+1}},
\end{align}
where $C >0$ depends only  on $\|u_0\|_{H^{l-1}}$ and is independent of $\al.$
\end{prop}
Note that a crucial point of this proposition is that the local existence time $T$ and the local estimate \eqref{lwpZ} are independent of $\al$. A similar conclusion is necessary for previous results on the subsonic problem. Compared with \cite{Weinstein1986,Added1988,Ozawa1992,Kening1995}, the novelty of Proposition 1.1 is that we prove an $\al$-uniform Sobolev local theory for large and low-regularity data without spatial weights. Compared with the result of Masmoudi and Nakanishi \cite{Masmoudi2008}, they prove the uniformly bounded Zakharov energy, but their result is qualitative and does not provide the high-order $\al$-uniform Sobolev bounds needed to extract convergence rates.

The main result on the rate of convergence is as follows:
\begin{thm}
\label{mainthm}
Suppose that $\vec u_0\in\HH^{3}$. Let $\vec u,v,w$ be the solutions to \eqref{equ-Z}, \eqref{equ-NLS}, and \eqref{equ-w}, respectively, with $v_0=u_0$. Then, there exists $ T = T\big(\|\vec u_0\|_{\HH^{3}}\big)$
which is independent of $\alpha$, and $\al_0>0$ such that for any $\al>\al_0$, the following estimates hold:
\begin{enumerate}
\item 
It holds that
\begin{align}
\label{thm:rconvergencerate1}
\quad \sup_{t \in [0, T]}  \|u(t)-v(t)\|_{L^2_x(\R^3)} \leq C(\|\vec u_0\|_{\HH^{3}}) \al^{-1}.
\end{align}
If we further assume $\vec u_0\in \HH^{4}, n_0+|u_0|^2=0$, and  $|\na|^{-1}n_1\in H^{ 1},$ then
\begin{align} 
\label{thm:rconvergencerate2}
\sup_{t \in [0, T]} \| u(t)-v(t)\|_{L^2_x(\R^3)} \leq C(\|\vec u_0\|_{\HH^{4}}, \||\na|^{-1}n_1\|_{H^{ 1 }}) \al^{-2}.
\end{align}	
\item It holds that
\begin{align}
\label{thm:qconvergencerate1}
\quad \sup_{t \in [0, T]} \|n(t) + |u(t)|^2 - w(t)\|_{L^2_x(\R^3)}\leq C(\|\vec u_0\|_{\HH^{3}}) \al^{-1}.
\end{align}
If we further assume $\vec u_0\in \HH^{4}$, then
\begin{align} 
\label{thm:qconvergencerate2}
\sup_{t \in [0, T]} \|n(t) + |u(t)|^2 - w(t)\|_{L^2_x(\R^3)} \leq C(\|\vec u_0\|_{\HH^{4}}) \al^{-2}. 
\end{align}
\end{enumerate}
\end{thm}

	
\begin{remark} We  make several miscellaneous remarks regarding the main theorem.
\begin{enumerate} 
\item \textbf{Possible increase of $|\nabla|^{-1}n_1$.} In this paper, the convention $|\al\nabla|^{-1}n_1\in H^s,s>0$ means that
$\norm{|\al\nabla|^{-1}n_1}_{H^s}\le C,$
with a constant $C$ independent of $\al$. Therefore, the assumption $\vec u_0\in\HH^3$ implies that $\norm{|\nabla|^{-1}n_1}_{H^{2}}\le C\al$, allowing $|\nabla|^{-1}n_1 $ to grow as $\al\to+\I$. In this sense, our assumption is weaker than those in previous quantitative results, which require uniform-in-$\al$ control of either $n_1 \in H^s$ \cite{Kening1995} or $ |\nabla|^{-1}n_1 \in H^s$ \cite{Added1988,Ozawa1992}.

\item \textbf{Derivative loss.}  In the quantitative error estimates, the main derivative-loss terms arise from $|\al\na|^{-1}\p_{tt}|u|^2$ in the wave remainder. In the resonant region, the oscillation does not cancel the quadratic nonlinear effect. This is where the first-order $L^2$-estimate requires a three-derivative loss. The  resonant structure also shows that one more derivative gains one more power of $\al^{-1}$. In the non-resonant part, we use normal form reductions and transversal bilinear Strichartz estimates to obtain the desired convergence rate without further derivative loss. These two parts lead to the assumptions \(\vec u_0\in\HH^3\) for the first-order \(L^2\)-estimate and \(\vec u_0\in\HH^4\) for the second-order \(L^2\)-estimate.
\item \textbf{Initial data error.} We assume zero initial data error: \(v_0 = u_0\), \(w_0 = n_0 + |u_0|^2\), \(w_1 = n_1 + \partial_t |u|^2(0)\) for simplicity. The results remain valid if a suitable initial error estimate is imposed. For example, \eqref{thm:rconvergencerate1} and \eqref{thm:qconvergencerate1} hold under
$$
\|(u_0 - v_0,\, n_0 + |u_0|^2 - w_0,\, |\alpha\nabla|^{-1}(n_1 + \partial_t |u|^2(0) - w_1))\|_{H^{1/2} \times H^{1/2} \times H^{1/2}} \le C \alpha^{-1},
$$
where the initial data for \eqref{equ-Z} may implicitly depend on $\al$.	Moreover, the second-order convergence requires second-order initial errors.    

		\end{enumerate}

	\end{remark}

Now, we briefly explain the main ideas and innovations of our result. The proof has two main parts: the local theory uniformly in $\al$ and the error estimate.

\begin{enumerate}
\item \textbf{Local theory.}  
To obtain the convergence rate, we require uniform estimates for \eqref{equ-Z}, including both the existence time and spacetime norms independent of $\alpha$.

The main difficulty in establishing uniform local theory arises from the resonant frequency. In the related Klein-Gordon-Zakharov setting, Masmoudi and Nakanishi showed that it is impossible to obtain uniform bounds in any Sobolev space as $\alpha \to \infty$ via a simple fixed-point argument (Theorem 10.1 in \cite{Masmoudi2005}). In \cite{Masmoudi2008}, they derived local $H^1$ estimates using the energy structure on short time intervals. Moreover, Kenig-Ponce-Vega \cite{Kening1995} derived $\alpha$-independent local results using the standard iteration scheme, which requires high Sobolev regularity and decay assumptions on the initial data.

Inspired by the energy method in \cite{Masmoudi2008}, we use high-order energy identities to obtain $\mathit{a\ priori}$ control of the high-order Sobolev norms. However, the remainder cannot be closed using the energy structure alone, so we further invoke the mass structure with high-order derivatives to close the bootstrap argument.

\item \textbf{Error estimate: normal form.}  
In the error estimate, the first main challenge is to identify the number of derivatives required to achieve the desired first- or second-order convergence. In the wave component, the main term arises from the nonlinear term $\alpha^{-2}\partial_{tt}|u|^2$, particularly in the high-low frequency interaction:
\[
\alpha^{-1}\int_0^t e^{is \alpha |\nabla|} (|\nabla|^3 u_{high} u_{low}) \, ds.
\]
A direct application of Strichartz estimates yields a decay rate of $\alpha^{-1}$ for this term. However, the high-frequency part with the frequency support in $\fbrk{|\xi| \ll \alpha}$ poses an obstacle for higher-order decay. Fortunately, in this case, the non-resonant structure allows us to employ the normal form method to obtain additional decay.

\item \textbf{Error estimate: generating derivative terms.}  
The second main challenge is how to handle terms whose convergence order is insufficient. First, to identify the decay in $\al$ more clearly, we introduce a scaling transformation $S_\alpha f(t,x) = \alpha^{-1} f(\alpha^{-2} t, \alpha^{-1} x)$. Under this new variable, if we can generate an additional spatial derivative on the solution, then we can gain an additional $\al^{-1}$-decay. Then, we obtain additional decay on $\al$ using the following two strategies:

$\bullet$ Splitting normal form method.  
The main term of the Schr\"odinger component is
\begin{align}
\label{intro:Sterm}
\alpha^{-1} \int_0^t e^{\frac i2 s \Delta} S_\alpha w \, S_\alpha v \, ds.
\end{align}
Applying Strichartz estimates yields at most $\alpha^{-1/2+}$, which is insufficient to achieve the $\alpha^{-1}$ convergence. This necessitates us to exploit the oscillation integral property for this term.

After Fourier transformation, the term reads
\begin{align}
\label{intro:fourier}
\alpha^{-1} \int_0^t \int_{\R^3} e^{is\phi} \F\big(e^{is|\na|}S_\alpha w_{-}\big)(\xi_1)\, \F\big(e^{\frac i2s\De}S_\alpha v\big)(s, \xi-\xi_1)\dxio\ds,
\end{align} 
where $\phi=\frac{1}{2}|\xi|^2-|\xi_1|-\frac12 |\xi-\xi_1|^2$ and $S_\alpha w_-:=S_\al (w-i|\al\nabla|^{-1}w_t)$. The standard normal form transform is to write $e^{is\phi}=\frac{\partial_se^{is\phi}}{i\phi}$ and integrate by parts in $s$ when $\phi$ is non-degenerate. 

In this paper, we use a modification. We split the phase as $\phi = \phi_1 + \phi_2 + \phi_3$ with $\phi_1 = \frac12|\xi|^2$, $\phi_2 = -|\xi_1|$, $\phi_3 = -\frac12|\xi-\xi_1|^2$, and write $$e^{is\phi} = (i\phi_2)^{-1} \partial_s (e^{is\phi_2}) e^{is\phi_1} e^{is\phi_3}.$$
This splitting effectively gains an additional spatial derivative: In fact, from the viewpoint of Coifman-Meyer multilinear multiplier theorem, we have the heuristic observation
$$\phi_{2}^{-1}\p_s\big[e^{is\phi_1} \widehat{S_\al v}(s,\xi-\xi_1)\big]\sim |\xi_1|^{-1}|\xi|^{2}\widehat{S_\al v}(s,\xi-\xi_1)+|\xi_1|^{-1}\partial_s\widehat{S_\al v}(s,\xi-\xi_1),$$
then this gives $|\nabla|^{-1}$ for $S_\al w_{-}$ and $|\nabla|^2$ for $S_\al v$. Therefore, through the integration by parts in $s$, we can totally gain first-order derivative, thereby obtaining an extra factor of $\al^{-1}$. This idea is crucial for achieving the desired first-order convergence in the incompatible case and second-order convergence in the compatible case. We also note that the splitting normal form method is used in \cite{Bai2025rough} to control the derivative loss.

$\bullet$ Bilinear Strichartz estimates.  To obtain the optimal $\al^{-2}$-rate for wave component, the main term is
\[
\int_0^t e^{is|\nabla|} |\nabla| \big[(P_{\ll1} S_\alpha w)_{high} \, (P_{\ll1} S_\alpha u)_{low}(P_{\ll1} S_\alpha u)_{low}\big]   \ds.
\]
There are two difficulties: 
\begin{enumerate}
	\item Simply applying the Strichartz estimates for this integral provides at most $\alpha^{-2+}$ decay, with an arbitrarily small $\epsilon$-loss. Moreover, this estimate requires higher regularity $\vec u_0\in \HH^4$.
	\item 
	For this particular interaction frequency, the resonant structure prevents the use of the standard normal form, and the splitting normal form also fails.
\end{enumerate}
To overcome these difficulties, we exploit the transversality of the Schr\"odinger and wave components, namely the bilinear Strichartz estimate: for $N\ll1$,
\begin{align}
\label{intr:bil}
\| [ e^{\pm i t |\nabla|} f_M ] [ e^{\frac12 i t \Delta} g_N ] \|_{L^2_{t,x}}\les M \| f_M \|_{L^2}\| g_N \|_{L^2}.
\end{align}
Under the $U^p$-$V^p$ framework, it suffices to control it using duality
$$\int_{0}^{\al^2 T}\int_{\R^3} (P_{\ll1}|\nabla|S_\alpha w)_{high} \, |(P_{\ll1} S_\alpha u)_{low}|^2 \, \phi_{high}\dx\dt,$$
where $e^{it|\na|}\phi\in U^{2}(\R;L^2_x)$. This integral is bounded by a product of two $L^2_{t,x}$ integral: $$\| (P_{\ll1}|\nabla|S_\alpha w)_{high} \,(P_{\ll1} S_\alpha u)_{low}\|_{L^2_{t,x}} \|  (P_{\ll1} S_\alpha u)_{low}\, \phi_{high}\|_{L^2_{t,x}}.$$ 
Applying the transferred version of \eqref{intr:bil} twice to each of two $L^2_{t,x}$ integrals, we can effectively generate each $S_\al u$ an additional spatial derivative. This allows us to cover the loss of $\al^{\varepsilon}$.

\end{enumerate}

	\subsection{Organization of the paper}
	In Section \ref{sec:pre}, we give some notations, definitions, and useful results. In Section \ref{sec:local}, we prove the uniform $\mathit{a\ priori}$ estimate and local well-posedness for the solution of the system \eqref{equ-Z}. In Section \ref{sec:proofofmain}, we give the error estimate and finish the proof of Theorem \ref{mainthm}. 
	
		\vspace{0.5cm} 
	
	\section{Preliminaries}
	\label{sec:pre}
	
	\vspace{0.5cm}
	
\subsection{Notation}
	\label{sec:notations}
	For any $z\in\C$, we define $\re z$ and $\im z$ as the real and imaginary parts of $z$, respectively. Let $\varepsilon$ denote a constant satisfying $0<\varepsilon\ll 1$.
	Denote $\langle x\rangle=\sqrt{1+|x|^2}$.
	Let $C,C_i>0,\ i\in\N$ denote constants which may vary from line to line. If a constant $C$ depends on $a$, then we may write $C(a)$. If $f\leqslant C g$, we write $f\lesssim g$ or $f=O(g)$.
	Let $\chi\in C_0^\I(\R^d)$ be a radial, real-valued, and smooth cut-off function such that $\chi \geqslant  0$ satisfies
	\EQ{
		\chi(\xi) \coloneqq \left\{ \aligned
		&1\text{, when $|\xi|\le 1$,}\\
		&0\text{, when $|\xi|\ge 2$.}
		\endaligned
		\right.
	}
	For dyadic $N\in 2^{\Z}, $ let $\chi_N(\xi)\coloneqq \chi(N^{-1}\xi). $ Then we define
	$$ \varphi_N(\xi)\coloneqq \chi_N(\xi)-\chi_{\frac{N}{2}}(\xi). $$
	We use $\hat{f} $ or $\F f$ to denote the Fourier transform of $f$:
	\EQ{
		\hat{f} (\xi)=\F f(\xi)\coloneqq  \int_{\R^d} e^{-ix\cdot\xi}f(x)\rm dx,
	}
	we also define 
	\begin{align*}
		\F^{-1} g(x)\coloneqq \frac{1}{(2\pi)^d} \int_{\R^d} e^{ix\cdot\xi}g(\xi)\dxi.
	\end{align*}
	Using the Fourier transform, we can define the fractional derivative $|\na|\coloneqq\F^{-1}|\xi|\F $ and $|\na|^s\coloneqq\F^{-1}|\xi|^s\F. $
	Moreover, we also need the following homogeneous Littlewood-Paley dyadic operators: for any $N\in2^\Z$,
	\begin{align*}
		f_N=P_N f\coloneqq &\F^{-1}\big(\varphi_N(\xi) \hat{f}(\xi)\big).
	\end{align*}
	Then by definition, we have $f=\sum_{N\in 2^{\Z}}f_N.$ Moreover, we also need the following: for any  $N\in2^\Z$,
	\begin{align*}
		f_{\leq N}=P_{\leqslant N} f\coloneqq&\F^{-1}\big(\chi_N(\xi) \hat{f}(\xi)\big), \\
		f_{\ll N}=P_{\ll N} f \coloneqq&\F^{-1}\big(\chi_N(2^{5}\xi) \hat{f}(\xi)\big), \\
		f_{\les N}=P_{\les N} f \coloneqq &\F^{-1}\big(\chi_N(2^{-5}\xi) \hat{f}(\xi)\big), 
	\end{align*}
	and 
	\begin{align*}
		f_{\sim N}= P_{\sim N} f\coloneqq f_{\les N}-f_{\ll N}.
	\end{align*}
	We also denote that 
	$f_{\geq N}=P_{\geq N} f\coloneqq f-f_{\leq N}, f_{\gg N}=P_{\gg N} f\coloneqq f-f_{\les N}, \mbox{ and } f_{\ges N}=P_{\ges N} f \coloneqq f-P_{\ll N} f.$

	Let $\mathscr{S}(\R^d)$ be the Schwartz space, $\mathscr{S}'(\R^d)$ be the tempered distribution space, and $C^\infty_0(\R^d)$ be the space of all the smooth compact-supported functions. Given $1\leq p\leq \infty$, $L^p(\R^d)$ denotes the usual Lebesgue space. For any  $s\in\R $ and $1< p<\infty,$ we define the homogeneous Sobolev space
		\begin{align*}
			\dot{W}^{s,p}(\R^d)\coloneqq \bigl\{f\in \mathscr{S}'(\R^d)/\mathscr{P}(\R^d):\|f\|_{\dot{W}^{s,p}(\R^d)}\coloneqq \||\na|^sf\|_{L^p(\R^d)}<+\infty\bigr\},
		\end{align*} 
		where $\mathscr{S}'(\R^d)/\mathscr{P}(\R^d)$ denotes a quotient space of  tempered distributions modulo polynomials.
		We denote that $\dot{H}^s(\R^d)\coloneqq \dot{W}^{s,2}(\R^d).$ For $s\geq 0$, the inhomogeneous spaces are defined by 
		\begin{align*}
			W^{s,p}(\R^d)=	\dot{W}^{s,p}(\R^d)\cap L^{p}(\R^d);\quad H^s(\R^d)=\dot{H}^s(\R^d)\cap L^2(\R^d).
		\end{align*} 
	Let $I\subset \R$ be a time interval. We often use the abbreviations $L^p=L^p(\R^3)$, $ H^s=H^s(\R^3)$, $ L^q_tL^r_x(I)=L^q_tL^r_x(I\times\R^3)$, and  $ L^q_tH^s_x(I)=L^q_tH^s_x(I\times\R^3)$. Moreover, we denote that
	$
	\|f\|_{X\cap Y}\coloneqq \|f\|_X+\|f\|_Y$ and $ \|(f,g)\|_{X\times Y}\coloneqq \|f\|_X+\|f\|_Y$.
	Next, we show the Triebel-Lizorkin Spaces $\dot{F}_{p}^{a, q}$ with the corresponding norm as follows,
	\EQ{
		\|u\|_{\dot{F}_{p}^{a, q}}=\|N^{a}P_Nu\|_{ L_x^p \ell_{N\in 2^{\mathbb{Z}}}^q}.
	}
	For any $1\leq p<\infty,$ define $\ell_N^p=\ell^p_{N\in2^\Z}$ by its norm
	\begin{align*}
		\|c_N\|_{\ell_N^p}^p\coloneqq \sum_{N\in2^\Z} |c_N|^p.
	\end{align*}
	In this paper, we use the following abbreviations:
	\begin{align*}
		\sum_{N:N\leq N_1}\coloneqq\sum_{N\in2^\Z:N\leq N_1} \mbox{ and }
		\sum_{N\leq N_1}\coloneqq\sum_{N,N_1\in2^\Z:N\leq N_1}.
	\end{align*}
We also use $\lceil x\rceil$ to represent the smallest integer greater than or equal to $x$.
	
	The following definition of ``admissible exponent pair'' for the Strichartz estimates will be used frequently.
	\begin{definition}[Admissible pair]
		(1) We say that the exponent pair $(q,r)\in\R^2$ is sharp $\sigma$-admissible if $q,r\geq2,(q,r,\sigma)\ne(2,\I,1)$ and 
		\begin{align*}
			\frac{1}{q}+\frac{\sigma}{r}=\frac{\sigma}{2}.
		\end{align*}
		For a given dimension $d,$ we say that a pair $(q,r)$ of exponents is wave-admissible if $d\geq2$ and $(q,r)$ is sharp $\frac{d-1}{2}$-admissible, and Schr\"odinger-admissible if $d\geq1$ and $(q,r)$ is sharp $\frac{d}{2}$-admissible.
		
		(2) For any $0\leq \gamma\leq1$, we say that the exponent pair $(q,r)\in\R^2$ is $\dot H^\gamma$-admissible, if $\frac{2}{q}+\frac{d}{r}=\frac{d}{2}-\ga$, $2\leq q\leq\I$, $2\leq r\leq\I$, and $(q,r,d)\ne(2,\I,2)$. If $\ga=0$, we say that $(q,r)$ is $L^2$-$admissible$.
	\end{definition}
	We define homogeneous spaces 
	$\mathcal{S}^{\gamma}(I), 
	\mathcal{W}^{\gamma}(I)$
	with norm:
	\begin{align*}
		\|f\|_{\mathcal{S}^{\gamma}(I)}&\coloneqq \||\na|^{\gamma}f\|_{L^{\infty}_tL^2_{x}(I)}+\||\na|^{\gamma}f\|_{L^{2}_tL^6_{x}(I)},\\
		\|g\|_{\mathcal{W}^{\gamma}(I)}&\coloneqq \||\na|^{\gamma}g\|_{L^{\infty}_tL^2_{x}(I)}.
	\end{align*} 
	We also define inhomogeneous spaces $X^{k}, Y^{l}$ as follows:
	\begin{align*}
		\|f\|_{X^k(I)}&\coloneqq \|\jb{\nabla}^kf\|_{L^{\infty}_tL^2_{x}(I)}+\|\jb{\nabla}^kf\|_{L^{2}_tL^6_{x}(I)},
		\notag\\
		\|g\|_{Y^l(I)}&\coloneqq \|\jb{\nabla}^{l}g\|_{L^{\infty}_tL^2_{x}(I)}.
	\end{align*}

\subsection{Atom space and bounded variation space} In this section, we present the definitions of $U^p$ and $V^p,$ and summarize their relevant properties for our purposes. The $U^p$-$V^p$ methodology was pioneered by Koch-Tataru \cite{Koch-Tataru2005-CPAM}, and we direct readers to the comprehensive treatments in \cite{Candy2018Ann.PDE,Hadac-Herr-Koch-Poincare,Koch-Tataru2005-Duke,Koch-Tataru-Visan:Book} for complete theoretical details and further applications of these spaces.
\begin{definition}
	Let $1<p<\infty$ and $ \mathcal{Z}$ be the set of finite partitions $-\infty<t_0<t_1<...<t_K\leq\infty.$ 

(1) We call the function $a$ is a $U^p$-atom, if there exists a finite partition $\{t_k\}_{k=0}^{K} \in \mathcal{Z}$ and a collection $\{\phi_k\}_{k=0}^{K-1} \subset L_x^2$ with $\sum_{k=0}^{K-1} \|\phi_k\|_{L^2_x}^p=1,$  such that  $a:\R\rightarrow L^2_x$ given by $a=\sum_{k=1}^K\mathbbm{1}_{[t_{k-1},t_k)}\phi_{k-1}.$ Furthermore, the atomic space $U^p$ is then defined as 
		\begin{align}
			\label{def:Up}
			U^p(\R;L^2)\coloneqq \bigl\{u=\sum_{j=1}^{\infty}\lambda_ja_j:a_j \mbox{ are } U^p\mbox{-atoms},\lambda_j \in \C \mbox{ with } \sum_{j=1}^{\infty}|\lambda_j|<\infty\bigr\},
		\end{align}
		with the induced norm
		\begin{align}
			\label{norm:up}
			\|u\|_{U^p(\R;L^2)}\coloneqq \inf\bigl\{\sum_{j=1}^{\infty}|\lambda_j|:u=\sum_{j=1}^{\infty}\lambda_ja_j:a_j \mbox{ are } U^p\mbox{-atoms},\lambda_j \in \C\bigr\}.
		\end{align}
		
		\
		
(2) We define the space $	V^p(\R;L^2)$ as the normed space of all functions $v:\R\rightarrow L^2_x$ such that
			\begin{align}
				\label{norm:Vp}
				\|v\|_{V^p(\R;L^2)}\coloneqq\sup_{\{t_k\}_{k=0}^{K} \in \mathcal{Z}}\big(\sum_{k=1}^{K}  \|v(t_k)-v(t_{k-1})\|_{L^2_x}^p\big)^{\frac{1}{p}}.
			\end{align}
			is finite, where we use the convention $v(t_K)=v(\infty)=0.$ $V_{rc}^p$ denotes the closed subspace of all right-continuous $V^p$ functions with $\lim_{t\rightarrow-\infty}v(t)=0.$
			
			\
			
(3)  We define $U_{\De}^2(\R;L^2_x)$ as the adapted normed space
			\begin{align}
			U_{\De}^2(\R;L^2_x)\coloneqq\bigl\{u:\|u\|_{U_{\De}^2(\R;L^2_x)}=\|e^{\frac{1}{2}it\De}u\|_{U^2(\R;L^2_x)}<\infty\bigr\}.
			\end{align}
			Similarly,  $V_{\De}^2(\R;L^2_x)$ denotes the adapted normed space
				\begin{align}
				V_{\De}^2(\R;L^2_x)\coloneqq\bigl\{u:\|u\|_{	V_{\De}^2(\R;L^2_x)}=\|e^{\frac{1}{2}it\De}u\|_{V^2(\R;L^2_x)}<\infty, e^{\frac{1}{2}it\De}u\in V_{rc}^2\bigr\}.
			\end{align}
			
			\
			
(4)  We define $U_{\pm|\na|}^2(\R;L^2_x)$ as the adapted normed space:
			\begin{align}
				U_{\pm|\na|}^2(\R;L^2_x)\coloneqq\bigl\{u:\|u\|_{	U_{\pm |\na|}^2(\R;L^2_x)}=\|e^{\pm it|\na|}u\|_{U^2(\R;L^2_x)}<\infty\bigr\}.
			\end{align}
Elements of $U_{\pm|\na|}^2(\R;L^2_x)$ can be regarded as  being close to solutions to the linear half-wave equation. Indeed, the atoms in $U_{\pm|\na|}^2(\R;L^2_x)$ are piecewise solutions to  $(i\partial_t \mp |\na|)w_{\pm}=0.$
		Similarly,  $V_{\pm|\na|}^2(\R;L^2_x)$ denotes the adapted normed space
			\begin{align}
				V_{\pm|\na|}^2(\R;L^2_x)\coloneqq\bigl\{u:\|u\|_{	V_{\pm|\na|}^2(\R;L^2_x)}=\|e^{\pm it|\na|}u\|_{V^2(\R;L^2_x)}<\infty, e^{\pm it|\na|}u\in V_{rc}^2\bigr\}.
			\end{align}
\end{definition}

In this paper, we will use restriction spaces on an interval $$I\subset \R:U^p(I;L^2),V^p(I;L^2),	U_{\De}^2(I;L^2_x),	V_{\De}^2(I;L^2_x),U_{\pm|\na|}^2(I;L^2_x), \mbox{ and } V_{\pm|\na|}^2(I;L^2_x).$$ See Remark 2.23 in \cite{Hadac-Herr-Koch-Poincare} for more details.

Note that for $1\leq p<q<\infty,$ the embeddings 
\begin{align*}
	U^p(\R;L^2_x)\hookrightarrow L^\infty_t(\R;L^2_x), 	V^p(\R;L^2_x)\hookrightarrow L^\infty_t(\R;L^2_x), 
\end{align*}
and $U^p\hookrightarrow V_{rc}^p\hookrightarrow U^q$ are continuous.

 We need the following classical linear estimate and duality formula:
\begin{lem}(see \cite{Hadac-Herr-Koch-Poincare})
	\label{Lem:linearforup}
	Let $I$ be an interval such that $0=\inf I.$ Then for any $f\in L^2_x$,
		\begin{align*}
		\|e^{-\frac{1}{2}it\De }f\|_{U^2_{\De}(I;L^2_x)}\les \|f\|_{L^2_x},\\
		\|e^{\mp it|\na| }f\|_{U^2_{\pm |\na|}(I;L^2_x)}\les \|f\|_{L^2_x},
	\end{align*}
and for $F(t,x)\in L^1_tL^2_x(I),$
\begin{align*}
\big\|\int_{0}^t e^{-\frac{1}{2}i(t-s)\De }F(s)\ds\big\|_{U^2_{\De}(I;L^2_x)}&=\sup_{\|g\|_{V^2_{\De}(I;L^2_x)}=1}\left|\int_{I}\int_{\R^d} F(t)\bar{g}(t)\dx\dt\right|,
	\\
\big\|\int_{0}^t e^{\mp i(t-s)|\na| }F(s)\ds\big\|_{U^2_{\pm|\na|}(I;L^2_x)}&=\sup_{\|g\|_{V^2_{\pm|\na|}(I;L^2_x)}=1}\left|\int_{I}\int_{\R^d} F(t)\bar{g}(t)\dx\dt\right|.
\end{align*}
\end{lem}
		\subsection{Useful lemmas}
		Firstly, we introduce the following Bernstein estimates that will be used frequently.
		\begin{lem}[Bernstein estimates\label{lem:Bernstein}]
			For any $1\leq p \leq q \leq \infty$, $s\geq 0$, and $f\in L_x^p(\R^d)$,
			\begin{align*}
				\|P_{\geq N} f\|_{L_x^p(\R^d)}&\lesssim  N^{-s}\||\nabla|^{ s}P_{\geq N} f\|_{L_x^p(\R^d)},\\
				\||\nabla|^{ s}P_{\leq N} f\|_{L_x^p(\R^d)}&\lesssim N^s\|P_{\leq N} f\|_{L_x^p(\R^d)},\\
				\||\nabla|^{\pm s}P_N f\|_{L_x^p(\R^d)}&\sim N^{\pm s}\|P_N f\|_{L_x^p(\R^d)},\\
				\|P_{\leq N} f\|_{L_x^q(\R^d)}&\lesssim  N^{\frac dp-\frac dq}\|P_{\leq N} f\|_{L_x^p(\R^d)},\\
				\|P_N f\|_{L_x^q(\R^d)}&\lesssim  N^{\frac dp-\frac dq}\|P_N f\|_{L_x^p(\R^d)}.
			\end{align*}
		\end{lem}
		\begin{lem}(Strichartz estimates, see \cite{Tao-stri})
			\label{lem:stri}
			For any functions $\phi(x)$ and $f(t,x), $\\
			(1) If $(q,r)$ and $(\tilde{q},\tilde{r})$ are Schr\"{o}dinger-admissible, 
			then
			\begin{align}
				\label{hom-Nls-str}
				\big\|e^{-\frac{1}{2}it\Delta}\phi(x)\big\|_{L_{t}^qL^{r}_{x}(I\times \R^d)}&\lesssim\|\phi(x)\|_{L^2_x(\R^d)},\\
				\label{inhom-Nls-str}
				\Big\|\int^{t}_{t_0}e^{-\frac{1}{2}i(t-s)\Delta}f(s,x)\ds\Big\|_{L_{t}^qL^{r}_{x}(I\times \R^d)}&\lesssim\|f(t,x)\|_{L_t^{\wt{q}'}L_x^{\wt{r}'}(I\times \R^d)}.
			\end{align}
			(2) If both $(q,r)$ and $(\tilde{q},\tilde{r})$ are wave-admissible,
			then
			\begin{align}
				\label{hom-wave-str}
				\big\|e^{\pm it|\na|}\phi(x)\big\|_{L_{t}^q\dot{W}^{\frac{1}{q}+\frac{d}{r}-\frac{d}{2}, r}_{x}(I\times \R^d)}&\lesssim\|\phi(x)\|_{L^2_x(\R^d)},\\
				\label{inhom-wave-str}
				\Big\|\int^{t}_{t_0}e^{\pm i(t-s)|\na|}f(s,x)\ds\Big\|_{L_{t}^q\dot{W}_{x}^{\frac{1}{q}+\frac{d}{r}-\frac{d}{2}, r}(I\times \R^d)}
				&\lesssim\|f(t,x)\|_{L_t^{\wt{q}'}\dot{W}_{x}^{\frac{d}{2}-\frac{1}{\tilde{q}}-\frac{d}{\tilde{r}}, \wt{r}'}(I\times \R^d)}.
			\end{align}
		\end{lem}
		
		\begin{lem}[Schur's test]
			\label{lem:Schur}
			For any  $a>0$, let sequences $\{a_N\},\{b_N\}\in \ell^2_{N\in2^{\Z}},$ then we have
			\begin{align*}
				\sum_{N_1\leq N}\big(\frac{N_1}{N}\big)^a a_N b_{N_1}\les \|a_N\|_{\ell^2_{N\in2^{\Z}}}\|b_N\|_{\ell^2_{N\in2^{\Z}}}.
			\end{align*}
		\end{lem}
		
		We also need the following Littlewood-Paley theory, see Remark 2.2.2 in \cite{Gra-14}.
		\begin{lem}[Square function estimate\label{lem:Little}]
			Let $1<p<\infty$. For any $a \in \R$, we have
			\EQ{
				\|f\|_{\dot{F}_{p}^{a, 2}}\sim \||\na|^{a}f\|_{L_x^p}.
			}
		\end{lem}
		
		
		The following is Coifman-Meyer's Multiplier Theorem, see \cite{Coifman-Meyer}.
		\begin{lem}
			\label{lem:Coif}
			Let $m\in  L^\infty(\R^{nd})$ be smooth away from the origin and satisfy that for any multi-indices
			$a=\{a_1,a_2,\cdots,a_n\}\in \Z^{nd}$ with  $0\leq a_j\leq 2d+1 $ and any $\xi_1,\cdots,\xi_n\in \R^d\setminus\{0\}$,
			$$
			\big|\partial_{\xi_1}^{a_1}\partial_{\xi_2}^{a_2}\cdots \partial_{\xi_n}^{a_n}m(\xi_1,\xi_2,\cdots,\xi_n)\big|
			\le C(a)(|\xi_1|+|\xi_2|+\cdots+|\xi_n|)^{-|a|}.
			$$
			Then for any $f_j,j=1,2,\cdots,n$,
			\begin{align*}
				&\left\|\int_{\xi=\xi_1+\cdots+\xi_n}  e^{i\xi\cdot x} m(\xi_1,\xi_2,\cdots,\xi_n)\widehat{f_1}(\xi_1)\widehat{f_2}(\xi_2)
				\cdots\widehat{f_n}(\xi_n)\,d\xi_1d\xi_2\cdots d\xi_n\right\|_{L^p(\R^d)}\\
				&\qquad \qquad \le C(p,p_1,\cdots,p_n)\|f_1\|_{L^{p_1}(\R^d)}\|f_2\|_{L^{p_2}(\R^d)}\cdots \|f_n\|_{L^{p_n}(\R^d)},
			\end{align*}
			where
			$$
			0< p< +\infty, 1<p_j\le+\infty \mbox{ for } j=1,2,\cdots,n, \frac1p=\frac1{p_1}+\frac1{p_2}+\cdots +\frac1{p_n}.
			$$
		\end{lem}
		
		The Coifman-Meyer multiplier theorem reduces to the Mihlin-H\"ormander
		multiplier theorem when \(n=1\) and \(1<p<\infty\).

	\begin{lem}(Bi-linear estimates for Schr\"odinger-wave)
	\label{lem:bi-W-N}
	Let $I\subset \R $ and $M,N \in 2^{\Z}$ with $N\nsim1$. Suppose that $f, g\in L^2(\R^3), $ then,
	\begin{align*}
	\|[e^{\pm it|\na|}f_M][e^{-\frac{1}{2}it\De}g_N]\|_{L^{2}_{t,x}(I)} \les  |1-N|^{-\frac{1}{2}}\min\{N,M\}\|f_M\|_{L^2}\|g_N\|_{L^2}.
	\end{align*}
\end{lem}

The proof of Lemma~\ref{lem:bi-W-N} is classical, and we give it in
Appendix~\ref{sec:appendix}. For a general multi-scale bi-linear restriction
theory for general phases, see Candy~\cite{Candy2018Ann.PDE}. In this paper,
we use the corresponding $U^2$-space version, which follows easily from the
transference principle:
\begin{cor}
\label{lem:bi-linear-W}
Let $I\subset \R $  and $M,N\in 2^{\Z}$ with $N\ll 1$. Suppose that  $\hat{w}_{\pm}(t,\cdot)$ is supported on $\{\eta:|\eta|\sim M\} $ and $\hat{v}(t,\cdot)$ is supported on $\{\xi:|\xi|\sim N\}$. Then,
\begin{align}
\|P_Mw_{\pm}P_Nv\|_{L^2_{t,x}(I)}\les \min\{N,M\}\|P_Mw_{\pm}\|_{U_{\pm |\na|}^{2}(I;L^2)}\|P_Nv\|_{U_{\De}^{2}(I;L^2)}.
\end{align}
\end{cor}
		
The Kato-Ponce inequality will be frequently used in this paper. The result was originally proved in \cite{Kato-Ponce} and then extended to the endpoint case in \cite{BoLi-KatoPonce, Li-KatoPonce}.
		\begin{lem}[Kato-Ponce's inequality] \label{lem:kato-Ponce}
			For $s>0$, $1<p\le \infty$, $1<p_1,p_3< \infty$, and $1<p_2,p_4 \le \infty$ satisfying $\frac1p=\frac1{p_1}+\frac1{p_2}$, and $\frac1p=\frac1{p_3}+\frac1{p_4}$, the following inequality holds:
			\begin{align*}
				\big\||\nabla|^s(fg)\big\|_{L^p}\le C\big( \||\nabla|^sf\|_{L^{p_1}}\|g\|_{L^{p_2}}+ \||\nabla|^sg\|_{L^{p_3}}\|f\|_{L^{p_4}}\big),
			\end{align*}
			where the constant $C>0$ depends on $s,p,p_1,p_2,p_3,p_4$.
		\end{lem}

		In this paper, we will rely on the local well-posedness of the 3D cubic nonlinear Schr\"odinger equation, see \cite{BookNLS2003}:
	\begin{lem}\label{lem-NLS}
	For any $v_0 \in H^1$, there exists $T=T(\|v_0\|_{H^1})>0$ such that the Cauchy problem
	\begin{equation}\label{eq:NLS}
	\begin{cases}
	2i\partial_t v - \Delta v = |v|^2 v,\\[2mm]
	v(0,x) = v_0(x),
	\end{cases}
	\end{equation}
	admits a unique solution 
	$	v\in X^1([-T,T])
	$	satisfying
	$	\|v\|_{X^1([-T,T])}\lesssim \|v_0\|_{H^1}.$
	Moreover, if $v_0\in H^s$ for some $s>1$, then
	$v\in X^s([-T,T])$ and satisfies
$$	\|v\|_{X^s([-T,T])}\lesssim \|v_0\|_{H^s}.$$

	\end{lem}

			\vspace{0.5cm} 
		\section{Local well-posedness for the Zakharov system}
		\label{sec:local}
		
		\vspace{0.5cm}
		
		In this section, we establish the local well-posedness of the Zakharov system \eqref{equ-Z} uniformly in $\al$, stated in Proposition \ref{prop:lwpZ}. This will also be used later to overcome derivative-loss in the analysis for wave-profile remainder $q$ in \eqref{equ-r,q}.

%

		
		\subsection{Local well-posedness when $\al=1$}
		Local well-posedness for the Cauchy problem of the Zakharov system in the case of sound speed $\al=1$:
		\begin{align}
			\label{equ-z}\tag{$\text{Zak}_1$}
			\left\{
			\begin{aligned}
				&2i\partial_{t}u-\Delta u=-nu, \\
				&\partial_{tt}n-\De n=\De|u|^2,
			\end{aligned}
			\right.
		\end{align}
		with initial data $(u_0, n_0, n_1)\in H^k \times H^{l}\times H^{l-1} $ has been extensively studied in the literature, see, e.g., \cite{Bourgain1996lwp, ChenWu2021, Tsutsumi1997, Sanwal2022} and the references therein. In this paper, we need the following local well-posedness result for the 3D Zakharov system, which is included in the classical work \cite{Tsutsumi1997} by  Ginibre, Tsutsumi, and Velo:
		
		\begin{lem}[LWP for 3D \eqref{equ-z}, \cite{Tsutsumi1997}]
			\label{lem:lwp97}
			Assume that $(k,l)\in \R^2$ such that
			\begin{align*}
				0\le l\leq k\leq l+1; \quad 2k\geq l+1.
			\end{align*}
			Then, for any initial data $(u_0, n_0, n_1)\in H^k \times H^{l}\times H^{l-1}$, there exists $T_0>0$ that depends on $\big\|(u_0,n_0,n_1)\big\|_{H^k \times H^{l}\times H^{l-1}}$ such that the Zakharov system \eqref{equ-z} admits a unique solution  $(u,n,n_t) \in C([0,T_0];H^k\times H^l\times H^{l-1})$ such that
			\begin{align}
				\label{Tsulwp}
				\norm{(u,n,n_t)}_{L_t^\I([0,T_0];H^k\times H^l\times H^{l-1})} \lesssim  \norm{(u_0,n_0,n_1)}_{H^k\times H^l\times H^{l-1}}.
			\end{align}
		\end{lem}
		
		\begin{remark}
			One may notice that there is a difference between the assumption on low-frequency part of $n_1$ in Lemma \ref{lem:lwp97} ($n_1\in L_x^2$) and our main result in Theorem \ref{mainthm} ($n_1\in\dot H_x^{-1}$). In fact, the change of variable to the first order system in \cite{Tsutsumi1997} shows that the estimates for $n$ and $|\nabla|^{-1}\pd_t n$ are identical. Therefore, Lemma \ref{lem:lwp97} holds if we change the assumption $n_1,n_t\in H_x^{l-1}$ to $|\nabla|^{-1}n_1,|\nabla|^{-1}n_t\in H_x^l$. In the following, we will use Lemma \ref{lem:lwp97} in the case when $|\nabla|^{-1}n_1,|\nabla|^{-1}n_t\in H_x^l$. 
		\end{remark}
		
		By rescaling, we are able to obtain the local well-posedness for \eqref{equ-Z} on some local existence time depending on $\al$:
		\begin{cor}
			\label{lem:lwp1}
			Let $d=3 $ and $l\geq0$. Then, for any initial data $\vec u_0\in \HH^{l+1}$, there exists $T_0^{(\al)}>0$ such that the Zakharov system \eqref{equ-Z} admits a unique solution such that $\vec u \in C([0,T_0^{(\al)}];\HH^{l+1})$ and
			\begin{align}
				\label{sca-lwp}
				\quad\|\vec u\|_{L_t^\I([0,T_0^{(\al)}];\HH^{l+1})} \leq C(\al)\|\vec u_0\|_{ \HH^{l+1}},
			\end{align}
			where $T_0^{(\al)}$ and $C(\al)$ depend on $\al$ and $\|\vec u_0\|_{\HH^{l+1}}$.
		\end{cor}
		\begin{proof}
	For any spacetime functions $f$ and $g$, we introduce the scaling transforms $S_\al$ and $W_\al$ as follows:
			\begin{align}
				\label{scaling-S,W}
				S_\al f\coloneqq \al^{-1}f(\al^{-2}t, \al^{-1}x),\quad
				W_\al g\coloneqq \al^{-2}g(\al^{-2}t, \al^{-1}x).
			\end{align}
			Moreover, if the functions $f$ and $g$ are independent of $t$, we adopt definition by
			\EQn{
				S_\al f\coloneqq \al^{-1}f( \al^{-1}x),\quad
				W_\al g\coloneqq \al^{-2}g( \al^{-1}x).
			}
If $(u,n,|\al\na|^{-1}\p_tn)$ solves the system \eqref{equ-Z} with initial data $(u_0,n_0,|\al\na|^{-1}n_1)$, then $$(S_\al u ,W_\al n ,|\na|^{-1}\p_tW_\al n)$$ solves the system \eqref{equ-z} with initial data $(S_\al u_0, W_\al n_0, \al^{-2}|\na|^{-1}W_\al n_1)$. 
Applying Lemma \ref{lem:lwp97} to \eqref{equ-z} with initial data $(S_\al u_0, W_\al n_0, \al^{-2}|\na|^{-1}W_\al n_1)$, we obtain a solution
\begin{align*}
(S_\al u ,W_\al n ,|\na|^{-1}\p_tW_\al n) \in  C\big([0,T_1^{(\al)}];\HH^{l+1}\big),
\end{align*}
where $$T_1^{(\al)}\coloneqq T_0\big(\big\|\big(S_\al u_0,W_\al n_0, \al^{-2}|\na|^{-1}W_\al n_1\big)\big\|_{\HH^{l+1}}\big),$$
and $T_0$ is given in Lemma \ref{lem:lwp97}. Moreover, the solution satisfies the following bound:
\begin{align}
\label{est:Sal(u,n)}
&\quad\big\|(S_\al u ,W_\al n ,|\na|^{-1}\p_tW_\al n)\big\|_{L_t^\I([0,T_1^{(\al)}];\HH^{l+1})} 
\notag\\
&\leq C\big\|\big(S_\al u_0,W_\al n_0, W_\al\big(|\al\na|^{-1}n_1\big) \big)\big\|_{\HH^{l+1}}
\notag\\
&\le C_1(\al),
\end{align}
where we denote that
\begin{align*}
C_1(\al)\coloneqq
& C\Big[ \alpha^{\frac{1}{2}}\|u_0\|_{L^2} + \alpha^{-\frac{1}{2}}\|n_0\|_{L^2} + \alpha^{-\frac{1}{2}}\||\al\na|^{-1}n_1\|_{L^2} \\
& \quad + \alpha^{-l-\frac{1}{2}}\big(\|u_0\|_{\dot{H}^{l+1}} + \|n_0\|_{\dot{H}^{l}}+\|\al^{-1}n_1\|_{\dot{H}^{l-1}} \big)\Big]. 
\end{align*}
It is difficult to track the explicit dependence on $\al$ in $T_1^{(\al)}$ from $C_1(\al)$. Next, let $T_0^{(\al)}\coloneqq\al^{-2}T_1^{(\al)}$. Applying $S_\al^{-1}$ and $W_\al^{-1}$, we obtain a local solution $$(u, n, |\al\na|^{-1}n_t)\in C([0,T_0^{(\al)}];\HH^{l+1})$$ such that
\begin{align*}
&\quad\|(u, n, |\al\na|^{-1}n_t)\|_{L_t^\I([0,T_0^{(\al)}];\HH^{l+1})} \\&\leq (\al^{-\frac{1}{2}}+\al^{l+\frac{1}{2}}+\al^{\frac{1}{2}})C_1(\al) \notag\\
&\leq C_2 \al^{l+1}\big\|(u_0,n_0,|\al\na|^{-1}n_1)\big\|_{\HH^{l+1}},
\end{align*}
for some constant $C_2>0$. Thus, \eqref{sca-lwp} follows with $C(\al)=C_2 \al^{l+1}$.	
\end{proof}
			
			
			%

			\subsection{High order energy identities}
			The derivative  $D^k$ is defined by
			\begin{align*}
				D^k \coloneqq 
				\begin{cases}
					\Delta^{p} & \text{for  } k=2p, \\
					\nabla \Delta^{p} & \text{for  } k=2p+1,
				\end{cases}
			\end{align*}
			where $p $ is a non-negative integer. We observe that $D^ku$ is a vector-valued function when 
			$k$ is odd, and a scalar-valued function when 
			$k$ is even. To handle products between derivatives of high orders in a uniform way, we define an operation $\diamond$ as follows:
			\begin{align*}
				D^{k}f \diamond  D^{k}g&\coloneqq
				\begin{cases}
					D^{k}f\cdot D^{k}g &  \text{ for } k \text{ is odd}, \\
					D^{k}f \, D^{k}g &  \text{ for } k \text{ is even},
				\end{cases}
				\\
				D^{a}f \diamond  D^{k-a}g \diamond  D^{k}h &\coloneqq
				\begin{cases}
					D^{k-a}g \,  D^{a}f\cdot D^{k}h  &  \text{ for } k,a \text{ are odd}, \\
					D^{a}f  \, D^{k-a}g \cdot D^{k}h&  \text{ for } k \text{ is odd, and } a \text{ is even},\\
					D^{k-a}g \, 	D^{a}f  \cdot D^{k}h&  \text{ for } k \text{ is even, and } a \text{ is odd},\\
					D^{a}f \,  D^{k-a}g  \, D^{k}h &  \text{ for } k,a \text{ are even}.
				\end{cases}
			\end{align*}
			Now for each integer $k\geq1,$  we define the following energy functionals:
			\begin{align}
				\mathcal{M}_k(t) & \coloneqq \| D^k u(t) \|_{L^2_x}^2, 
				\label{def-m_M}
				\\
				\mathcal{E}^{L}_{k}(t) & \coloneqq \| D^{k+1} u(t) \|_{L^2_x}^2 + \frac{1}{2} \| D^k n(t) \|_{L^2_x}^2 + \frac{1}{2} \| \alpha^{-1} D^{k-1} \partial_t n(t) \|_{L^2_x}^2, 
				\label{def-LmEm} 
				\\
				\mathcal{E}^{W}_{k}(t) &\coloneqq 
				\int_{\mathbb{R}^3} D^k \bigl( |u|^2 \bigr) \diamond D^k n(t,x) \dx,
				\label{def-WmEm}
				\\
				\label{def-m_Em}
				\mathcal{E}_k (t)&\coloneqq	\mathcal{E}^L_k (t)+	\mathcal{E}^{W}_{k}(t).
			\end{align}
			For \(k=0\), these functionals coincide with the conserved mass and energy \eqref{mass}--\eqref{energy}:
			\begin{align}
			\label{def:mE_0}
				\mathcal{M}_0(t)\coloneqq M(t)=M(0),\quad \mathcal{E}_0 (t)\coloneqq E(t)=\mathcal{E}^L_0 (t)+\mathcal{E}^W_0 (t)=\mathcal{E}_0 (0),
			\end{align}
			where
			\begin{align}
			\label{def:mE_L0}
				\mathcal{E}^L_0 (t)&\coloneqq\|\na u(t) \|_{L^2_x}^2+ \frac{1}{2} \|  n(t) \|_{L^2_x}^2+ \frac{1}{2} \|  |\al\na|^{-1} \partial_t n(t) \|_{L^2_x}^2,
			\\
			\label{def:mE_W0}
				\mathcal{E}^W_0 (t)&\coloneqq  \int_{\mathbb{R}^3}  n|u|^2(t,x)   \dx. 
			\end{align}
			We also denote that for integer $k\geq0,$
\begin{align}
\label{def-LmE}
	\mathcal{U}^L_k(t)&\coloneqq \mathcal{E}^L_0 (t)+\mathcal{E}^L_k (t),
			\end{align}
then we have
			\begin{align}
			\label{cal-LmE}
				\mathcal{U}^L_k(t)=\mathcal{E}_0 (t)+\mathcal{E}_k(t)-\mathcal{E}^W_0 (t)-\mathcal{E}^W_k (t).
			\end{align}
The following lemma gives the conservation identities for the modified mass 
$\mathcal{M}_k(t)$ and energy $\mathcal{E}_k(t)$.
			\begin{lem}
				\label{lem:mE}
				Let $k\geq2$ be an integer and let $\vec u \in C([0,T]; \HH^{k+1})$ be a solution to the system \eqref{equ-Z} on time interval $[0,T]$. The modified mass $\mathcal{M}_{k}(t)$ and energy $\mathcal{E}_k(t)$ defined in \eqref{def-m_M} and \eqref{def-m_Em} satisfy the following identities:
				\begin{align}
					\label{m_M(T)}
					\mathcal{M}_{k}(t) &= \mathcal{M}_{k}(0) -\im \int_{0}^{t} \int_{\mathbb{R}^3} D^{k}(n u) \diamond D^{k}\bar{u}(s,x) \dx  \ds, \\
					\label{m_E(T)}
					\mathcal{E}_{k}(t) &= \mathcal{E}_{k}(0) + \mathcal{E}^{I}_{k}(t),
				\end{align}
				where the interaction term $\mathcal{E}^{I}_{k}(t)$ admits the decomposition
		\begin{subequations}		
		\begin{align}
\label{main-EI}
	\mathcal{E}^{I}_{k}(t) &\coloneqq  2\re \int_{0}^{t} \int_{\mathbb{R}^3} \left[\bar{u}_s D^{k}u  \diamond D^{k}n - n D^{k} u  \diamond D^{k}\bar{u}_s \right](s,x) \dx \ds \\
	&\quad + 2 \re\int_{0}^{t} \int_{\mathbb{R}^3} \big[\mathcal{R}(u,\bar{u}_s)\diamond D^kn -\mathcal{R}(u,n)\diamond D^{k}\bar{u}_s\big] (s,x)\dx \ds \label{eq:EI-high}
\end{align}
\end{subequations}
with $\mathcal{R}(f,g)\coloneqq D^{k}(fg)-f D^{k}g-gD^{k}f $. 
\end{lem}

			\begin{proof}
				Applying the operator $D^{k},D^{k-2} (k\geq2)$ to the first and second equation of \eqref{equ-Z} respectively, we obtain 
				\begin{align}
					\label{equ-DmZs}
					&2i\partial_{t}D^k u-D^{k+2} u=-D^{k}(nu), 
					\\
					\label{equ-DmZw}
					&\al^{-2}\partial_{tt}D^{k-2}n-D^{k}  n= D^{k}(|u|^2).
				\end{align}
				Taking inner product by \eqref{equ-DmZs} with $D^{k}\bar{u}, $ and considering the imaginary part, we have
				\begin{align}
					\label{m_m1}
					\frac{d}{dt}\big\|D^{k}u\big\|^2_{L^2_x}=-\im \int_{\R^3} D^{k}(nu)\diamond D^{k}\bar{u} \dx.
				\end{align}
				Taking inner product by \eqref{equ-DmZs} with $D^{k}\bar{u}_t, $ and considering the real part, we have
				\begin{align}
					\label{m_e1}
					\frac{1}{2}\frac{d}{dt}\big\|D^{k+1}u\big\|^2_{L^2_x}=-\re \int_{\R^3} D^{k}(nu)\diamond  D^{k}\bar{u}_t \dx.
				\end{align} 
				Taking inner product by \eqref{equ-DmZw} with $D^{k} n_t, $ we have
				\begin{align}
					\label{m_e2}
					-\frac{1}{2}\frac{d}{dt}\big\|\al^{-1}D^{k-1}n_t\big\|^2_{L^2_x}-\frac{1}{2}\frac{d}{dt}\big\|D^{k}n\big\|^2_{L^2_x}= \int_{\R^3} D^{k}(|u|^2)\diamond D^{k}n_t\dx.
				\end{align}
				By integration-by-parts with respect to $t$, we have
				\begin{equation}
					\label{m_e2R}
					\begin{aligned}	
						\text{(RHS) of }\eqref{m_e2}&= \frac{d}{dt}\int_{\R^3}D^{k}(|u|^2)\diamond D^{k}n \dx
						-2\re \int_{\R^3}D^{k}(u\bar{u}_t)\diamond D^{k}n \dx
						\\
						&= \frac{d}{dt}\int_{\R^3}D^{k}(|u|^2)\diamond D^{k}n \dx-2 \re \int_{\R^3}u D^{k}\bar{u}_t\diamond D^{k}n \dx
						\\
						&\quad -2\re\int_{\R^3}\bar{u}_tD^{k}u\diamond D^{k}n \dx
						-2\re\int_{\R^3}\mathcal{R}(u,\bar{u}_t)\diamond D^kn \dx,
					\end{aligned}
				\end{equation}
where the remainder $\mathcal{R}(u,\bar{u}_t) $ is denoted by $D^{k}(u\bar{u}_t)-u D^{k}\bar{u}_t-\bar{u}_tD^{k}u$ satisfying the estimate
$$\|\mathcal{R}(u,\bar{u}_t)\diamond D^kn\|_{L^1_x} \les \sum_{a=1}^{k-1}\|D^au\|_{L_x^{p_1}}\|D^{k-a}\bar{u}_t\|_{L_x^{p_2}}\|D^{k}n\|_{L_x^{p_3}}, $$ for any $p_1,p_2,p_3\geq 1$ satisfying $1=\frac{1}{p_1}+\frac{1}{p_2}+\frac{1}{p_3}$. Direct calculation also gives that
	\begin{equation} \label{m_e1R}
		\begin{aligned}
		\text{(RHS) of }\eqref{m_e1}
		&=-\re \int_{\R^3} uD^{k}n\diamond D^{k}\bar{u}_t \dx-\re \int_{\R^3} nD^{k} u\diamond D^{k}\bar{u}_t\dx
						\\
		&\quad- \re\int_{\R^3} \mathcal{R}(u,n)\diamond D^k\bar{u}_t \dx,
					\end{aligned}
				\end{equation}	
	where the remainder $\mathcal{R}(u,n)$ is denoted by $D^{k}(un)-u D^{k}n-nD^{k}u$ satisfying the estimate
				$$\|\mathcal{R}(u,n)\diamond D^k\bar{u}_t\|_{L^1_x} \les \sum_{a=1}^{k-1}\|D^an\|_{L_x^{q_1}}\|D^{k-a}u\|_{L_x^{q_2}}\|D^{k}\bar{u}_t\|_{L_x^{q_3}}, $$ for any $q_1,q_2,q_3\geq 1$ satisfying $1=\frac{1}{q_1}+\frac{1}{q_2}+\frac{1}{q_3}$. 
				In the later estimates, only the main terms need to be focused on.
				Inserting \eqref{m_e2R}, \eqref{m_e1R} into \eqref{m_e1} and \eqref{m_e2} respectively, calculating $2\times \eqref{m_e1}-\eqref{m_e2},$ 
				then we obtain the desired result.
			\end{proof}
			
			\subsection{Local-in-time a priori estimate}
			Now, we prove $\mathit{a\ priori}$ estimate for 3D Zakharov system with the initial data $\vec u_0\in \HH^{k+1}$ on a time interval that is independent of $\al$.
			\begin{prop}[A priori estimate]
				\label{prop:energy estimate}
				Let $k \geq 2$ be an integer, and $\vec u_0\in \HH^{k+1}$. Then, there exists $T = T\big( \| \vec u_0 \|_{\HH^3} \big)$ such that: If $\vec u \in C([0,T]; \HH^{k+1})$ denotes a solution to the system \eqref{equ-Z}, then it holds that
				\begin{align}
					\label{lwp}
					\big\| \vec u \big\|_{L^{\infty}_t([0,T]; \HH^{k+1})} 
					\leq C \big\| \vec u_0 \big\|_{\HH^{k+1}},
				\end{align}
				where the constant $C$ is independent of $\alpha$, and depends only on $\|u_0\|_{H^{k-1}}$.
			\end{prop}
			
			\begin{remark}
				The condition for integer $k\geq2$ comes from the use of Sobolev's embedding $H^k(\R^3)\hookrightarrow L^{\infty}(\R^3)$.
			\end{remark}

			\begin{proof}
 Recall the expression of $\mathcal{E}^L_k(t)$ and $\mathcal{E}^L_0(t)$ in \eqref{def-LmEm} and \eqref{def:mE_L0}, from \eqref{def-LmE}, we only need to estimate $\sup_{t\in[0,T]}\mathcal{U}^L_k(t)$. We denote $$\mathcal{A}_k^L \coloneqq\sup_{t\in[0,T]}\mathcal{U}^L_k(t).$$ 
 Combining \eqref{def:mE_0}, \eqref{cal-LmE}, and \eqref{m_E(T)} in Lemma \ref{lem:mE}, we have
	\begin{align}
\mathcal{U}^{L}_k(t)&=\mathcal{E}_0 (t)+\mathcal{E}_k(t)-\mathcal{E}^W_0 (t)-\mathcal{E}^W_k (t)
\notag\\
&=\mathcal{E}_0(0)+\mathcal{E}_k(0)+\mathcal{E}^{I}_k(t)-\mathcal{E}_0^{W}(t)-\mathcal{E}_k^{W}(t),
				\end{align}
then, we have
\begin{align}
\label{estEL1}
\mathcal{A}_k^L
&\le
|\mathcal{E}_0(0)+\mathcal{E}_k(0)|
+\sup_{0\le t\le T}|\mathcal{E}^{I}_k(t)|
+\sup_{0\le t\le T}\big(|\mathcal{E}^{W}_0(t)|+|\mathcal{E}^{W}_k(t)|\big).
\end{align}
We will estimate $\mathcal{A}_k^L$ via the following steps.

\textit{Step 1. Estimate for $\sup_{0\le t\le T}\big(|\mathcal{E}^{W}_0(t)|+|\mathcal{E}^{W}_k(t)|\big)
$.}
From the expression of $\mathcal{M}_k(t)$ on \eqref{m_M(T)}, using H\"{o}lder's inequality, we have
\begin{align}
\label{Mest1}
\mathcal{M}_{k}(t)
&\leq \mathcal{M}_{k}(0)+C\int_{0}^t \big(\|D^k n\|_{L_x^2}\|D^k u\|_{L_x^2}\|u\|_{L_x^{\infty}}+\|n\|_{L_x^{3}}\|D^k u\|_{L_x^6}\|D^k u\|_{L_x^2}\big)(s) \ds
\notag\\
&\leq \mathcal{M}_{k}(0)+C\int_{0}^t \big[\mathcal{U}_{k}^{L}(s)\big]^{\frac{1}{2}}\big[\|D^k u\|_{L_x^2}\big(\|u\|_{L^\infty_{x}}+\|n\|_{L_x^{3}}\big)\big](s)\ds.
\end{align}
	By using  Gagliardo-Nirenberg's inequality, we have
\begin{align}
\label{Dm1}
\|D^k u(s)\|_{L_x^2}&\lesssim	\|D^{k+1} u(s)\|^{\frac{k-2}{k-1}}_{L_x^2}\|D^2 u(s)\|^{\frac{1}{k-1}}_{L_x^2}\lesssim\big[\mathcal{U}_{k}^{L}(s)\big]^{\frac{1}{2}},
\\
\label{inf1}
	\| u(s)\|_{L_x^{\infty}}&\lesssim	\|D^3 u(s)\|^{\frac{1}{4}}_{L_x^2}\|D u(s)\|^{\frac{3}{4}}_{L_x^2}\lesssim\big[\mathcal{U}_{2}^{L}(s)\big]^{\frac{1}{2}},
					\\
	\label{1/2:k-0}
	\|n(s)\|_{L_x^{3}}&\lesssim	\|D n(s)\|^{\frac{1}{2}}_{L_x^2}\| n(s)\|^{\frac{1}{2}}_{L_x^2}\lesssim\big[\mathcal{U}_{2}^{L}(s)\big]^{\frac{1}{2}}.
				\end{align}
				Inserting \eqref{Dm1}, \eqref{inf1}, and \eqref{1/2:k-0} into \eqref{Mest1}, we have
				\begin{align}
					\label{Mest}
					\sup_{t\in[0,T]}\mathcal{M}_{k}(t)&\leq \mathcal{M}_{k}(0)+CT\mathcal{A}_{k}^{L}\big(\mathcal{A}_{2}^{L}\big)^{\frac{1}{2}}.
				\end{align}
				Recall the expression of $\mathcal{E}^W_k(t)$ in \eqref{def-WmEm}, by Gagliardo-Nirenberg's inequality, H\"{o}lder's inequality, and \eqref{mass}, we have
\begin{align}
		\label{est-EWm1}
|\mathcal{E}_k^W(t)|&=\left|\int_{\R^3} D^k|u|^2 \diamond D^k n\dx\right|\notag\\
	&\leq C \|D^k u\|_{L_x^2}\|D^k n\|_{L_x^2}\|u\|_{L_x^\infty}\notag\\
	&\leq C \|D^k u\|_{L_x^2}\|D^k n\|_{L_x^2}\|D^3u\|^{\frac{1}{2}}_{L_x^2}\|u\|^{\frac{1}{2}}_{L_x^2}\notag\\
	&	\leq C(\|u_0\|_{L^2}) \big[\mathcal{M}_k(t)\big]^{\frac{1}{2}} \big[\mathcal{U}_k^{L}(t)\big]^{\frac{1}{2}}\big[\mathcal{U}_2^{L}(t)\big]^{\frac{1}{4}}.
\end{align}
Inserting \eqref{Mest} into \eqref{est-EWm1}, we have 
	\begin{align}
		\label{est-EWm}
	\sup_{t\in[0,T]}|\mathcal{E}_k^W(t)|&\leq  C \big[\mathcal{M}_k(0)\big]^{\frac{1}{2}}\big(\mathcal{A}^{L}_k\big)^{\frac{3}{4}}
	+CT^{\frac{1}{2}}\mathcal{A}^{L}_k\big(\mathcal{A}_2^{L}\big)^{\frac{1}{2}}.
\end{align}
Similarly, we have
\begin{align}
|\mathcal{E}_0^W(t)|=\left|\int_{\R^3} n|u|^2 \dx\right|
&\leq C \|n\|_{L_x^2}\| u\|_{L_x^2}\|u\|_{L_x^\infty}\notag\\
&\leq C \|n\|_{L_x^2}\| u\|_{L_x^2}\|D^{k+1}u\|_{L_x^2}^{\frac{3}{2(k+1)}}\|u\|_{L_x^2}^{1-\frac{3}{2(k+1)}}\notag\\
&	\leq C  \big[\mathcal{E}^{L}_k(t)\big]^{\frac{3}{4(k+1)}}\big[\mathcal{E}^{L}_0(t)\big]^{\frac{1}{2}}\| u_0\|^{2-\frac{3}{2(k+1)}}_{L^2},
\end{align}
which implies
\begin{align} \label{estEW0}
\sup_{t\in[0,T]}|\mathcal{E}_0^W(t)|
	&	\leq C(\| u_0\|_{L^2}) \big(\mathcal{A}^{L}_k\big)^{\frac{2k+5}{4(k+1)}}\| u_0\|^{2\frac{2k-1}{4(k+1)}}_{L^2}.
				\end{align}
Combining \eqref{est-EWm} and \eqref{estEW0}, we obtain 
the estimate of $\sup_{0\le t\le T}\big(|\mathcal{E}^{W}_0(t)|+|\mathcal{E}^{W}_k(t)|\big)$.

\textit{Step 2. Estimate for $|\mathcal{E}_0(0)+\mathcal{E}_k(0) |$.}
				Recall the expression of $\mathcal{E}_0(0)$ and $\mathcal{E}_k(0)$ from \eqref{def-LmEm}--\eqref{def:mE_W0}, we have
	\begin{equation}
			\begin{aligned}
			|\mathcal{E}_0(0)+\mathcal{E}_k(0) |\leq& \|(u_0, n_0, |\al \na|^{-1}n_1)\|^2_{H^{k+1}\times H^{k}\times H^{k}}
		\\	&+	\left|\int_{\R^3}n_0|u_0|^2\dx\right|+	\left|\int_{\R^3}D^{k}|u_0|^2\diamond D^kn_0\dx\right|.
	\end{aligned}
				\end{equation}
				Using H\"older's inequality, Sobolev's embedding $\dot{H}^{1}\hookrightarrow L^6,$ and $\dot{H}^{\frac{1}{2}}\hookrightarrow L^3$ in $x$, we have
				\begin{align}
					\label{estEW00}
				\left|\int_{\R^3}n_0|u_0|^2\dx\right|
						&\leq C\|u_0\|_{L^6}\|u_0\|_{L^3}\|n_0\|_{L^2}
					\leq C(\|u_0\|_{H^1})\|(u_0, n_0)\|^2_{\dot{H}^{1}\times L^{2}},\\
					\label{estEW0m}
					\left|\int_{\R^3}D^{k}|u_0|^2\diamond D^k n_0\dx\right|	
					&\leq C \|u_0\|_{L^3}\|D^k n_0\|_{L^2}\|D^k u_0\|_{L^6}
					\leq C(\|u_0\|_{H^1})\|(u_0, n_0)\|^2_{\dot{H}^{k+1} \times \dot{H}^{k}}.
				\end{align}
				Then,
				\begin{align}
					\label{estE0}
	|\mathcal{E}_0(0)+\mathcal{E}_k(0) |\leq C(\|u_0\|_{H^1}) \|(u_0, n_0, |\al\nabla|^{-1}n_1)\|^2_{H^{k+1}\times H^{k}\times H^{k}}.
				\end{align}
				\textit{Step 3. Estimate for $	\sup_{t\in[0,T]}|\mathcal{E}^{I}_k(t)|. $}
				Recall the expression of $\mathcal{E}^{I}_k(t) $ in \eqref{main-EI}--\eqref{eq:EI-high}, we only need to estimate the main term \eqref{main-EI}, while the lower order terms in \eqref{eq:EI-high} can be handled by interpolation. Then we have
				\begin{align}
					\eqref{main-EI}&=2\re \int_{0}^{t}\int_{\R^3} n D^{k} u\diamond D^{k}\bar{u}_s (s,x)\dx\ds +2\re \int_{0}^{t}\int_{\R^3}\bar{u}_s D^{k}u \diamond D^{k}n(s,x) \dx\ds\notag\\
					&\coloneqq I_1(t)+I_2(t).
				\end{align}
				Using the first equation of \eqref{equ-Z} and integration-by-parts, we have
				\begin{align}
					\left|I_1(t)\right|
					&\les  \int_{0}^{t}\int_{\R^3} \left|nD^{k} u\diamond \big[D^{k+2}\bar{u}+D^{k}(n\bar{u})\big] \right|\dx\ds \notag\\
					\label{esEI1}
					&\les \int_{0}^{t}\int_{\R^3} \left|n D^{k+1} u\diamond D^{k+1}\bar{u}\right|\dx\ds  \\
					\label{esEI2}	
					&\quad +\int_{0}^{t}\int_{\R^3}\left| \na n \diamond D^{k} u\diamond D^{k+1}\bar{u}\right|\dx\ds\\
					\label{esEI3}
					&\quad +\int_{0}^{t}\int_{\R^3} \left|nD^{k} u\diamond D^{k}(n\bar{u}) \right|\dx\ds . 
				\end{align}  
				Using H\"{o}lder's inequality and Gagliardo-Nirenberg's inequality, we have
	\begin{align*}
	\eqref{esEI1}&\les \int_0^t \big(\|D^{k+1} u\|^2_{L^2_x}\|n\|_{L_x^{\infty}}\big) (s)\ds
	\les \int_0^t \big(\|D^{k+1} u\|^2_{L^2_x}\|D^2 n\|^{\frac{3}{4}}_{L^2_x}\| n\|^{\frac{1}{4}}_{L^2_x}\big) (s)\ds\\
	&\les \int_0^t  \mathcal{U}^L_k(s)\big[\mathcal{U}^L_2(s)\big]^{\frac{1}{2}}\ds,\\
	\eqref{esEI2}&\les \int_0^t \big( \|D^{k+1} u\|_{L^2_x}\|D^k u\|_{L^6_x}\| Dn\|_{L^3_x} \big) (s)\ds
	\les \int_0^t \big( \|D^{k+1} u\|^2_{L^2_x}\|D^2 n\|^{\frac{3}{4}}_{L^2_x}\| n\|^{\frac{1}{4}}_{L^2_x}\big)(s)\ds\\
	&\les\int_0^t \mathcal{U}^L_k(s)\big[\mathcal{U}^L_2(s)\big]^{\frac{1}{2}}\ds,
			\\
	\eqref{esEI3}&\les \int_0^t \big( \|D^k u\|_{L^6_x}\|n\|_{L^3_x}\|D^k n\|_{L^2_x}\|u\|_{L_x^\infty}\big)(s)\ds+\int_0^t \big(\|D^k u\|^2_{L^6_x}\|n\|_{L^3_x}\| n\|_{L^2_x}\big)(s)\ds\\
	&\les \int_0^t \mathcal{U}^L_k(s)\mathcal{U}^L_2(s)\ds.
	\end{align*}
Combining the three estimates above, we have
\begin{align*}
	\sup_{t\in[0,T]}\left|I_1(t)\right|\les  T \mathcal{A}^L_k\Big[\big(\mathcal{A}^L_2\big)^{\frac{1}{2}}+\mathcal{A}^L_2\Big].
\end{align*}
Similarly, we have that $I_2$ has the same estimate as $I_1,$
and then 
\begin{align}
		\label{estmainEI}
	\sup_{t\in[0,T]}|\mathcal{E}^{I}_k(t)|\lesssim    T \mathcal{A}^L_k\Big[\big(\mathcal{A}^L_2\big)^{\frac{1}{2}}+\mathcal{A}^L_2\Big].
\end{align}

\textit{Step 4. Estimate for $\mathcal{A}^L_k$.} 
In the preceding steps, we have derived estimates for \[\sup_{t\in[0,T]}\Big(|\mathcal{E}_0^{W}(t)|+|\mathcal{E}_k^{W}(t)|\Big), \quad |\mathcal{E}_0(0)+\mathcal{E}_k(0)|,  \mbox{ and }	\sup_{t\in[0,T]}|\mathcal{E}^{I}_k(t)|.\] 
We now turn to the estimation of $\mathcal{A}^{L}_k$.
Inserting \eqref{est-EWm}, \eqref{estEW0}, \eqref{estE0}, and \eqref{estmainEI} into \eqref{estEL1}, we have
\begin{align}
\label{bootE1-pre}
\mathcal{A}^L_k
\leq& C_1(\|u_0\|_{H^1}) \|\vec u_0\|^2_{\HH^{k+1}} \notag\\
&+C_2(\|u_0\|_{H^1})
\big(\mathcal{A}^{L}_k\big)^{\frac{2k+5}{4(k+1)}}
\| u_0\|^{2\frac{2k-1}{4(k+1)}}_{L_x^2}
+C_3(\|u_0\|_{H^1})
\big(\mathcal{A}^{L}_k\big)^{\frac{3}{4}}
\| u_0\|_{\dot{H}^k} \notag\\
&+C_4(\|u_0\|_{H^1})T^{\frac{1}{2}}
\mathcal{A}^{L}_k\big(\mathcal{A}^{L}_2\big)^{\frac{1}{2}}
+C_5 T \mathcal{A}^L_k
\Big[\big(\mathcal{A}^L_2\big)^{\frac{1}{2}}+\mathcal{A}^L_2\Big].
\end{align}
For the two sublinear terms in \(\mathcal A_k^L\), Young's inequality gives
\begin{align}
\label{young-Ak-1}
C_2(\|u_0\|_{H^1})
\big(\mathcal{A}^{L}_k\big)^{\frac{2k+5}{4(k+1)}}
\| u_0\|^{2\frac{2k-1}{4(k+1)}}_{L_x^2}
&\leq
\frac18\mathcal{A}^{L}_k
+
C(\|u_0\|_{H^1})\|u_0\|_{L_x^2}^{2},
\\
\label{young-Ak-2}
C_3(\|u_0\|_{H^1})
\big(\mathcal{A}^{L}_k\big)^{\frac34}
\|u_0\|_{\dot H^k}
&\leq
\frac18\mathcal{A}^{L}_k
+
C(\|u_0\|_{H^1})\|u_0\|_{\dot H^k}^{4}.
\end{align}
Therefore, from
\eqref{bootE1-pre}--\eqref{young-Ak-2}, we obtain
\begin{align}
\label{bootE1}
\mathcal{A}^L_k
\leq&
C_1(\|u_0\|_{H^1})
\big(\|\vec u_0\|^2_{\HH^{k+1}}+\|u_0\|_{\dot{H}^{k}}^{4}\big)
+\frac{1}{4}\mathcal{A}^{L}_k \notag\\
&+
C_2(\|u_0\|_{H^1})\mathcal{A}^{L}_k
\Big[
T^{\frac{1}{2}}\big(\mathcal{A}^{L}_2\big)^{\frac{1}{2}}
+
T \big(\mathcal{A}^{L}_2\big)^{\frac{1}{2}}
+
T\mathcal{A}^{L}_2
\Big].
\end{align}
We denote $a_k \coloneqq 2C_1\big(\|\vec u_0\|^2_{\HH^{k+1}}+\|u_0\|_{\dot{H}^{k}}^{4}\big).$ Let $k=2$. From \eqref{bootE1}, we have
\begin{equation}
		\label{bootE2}
		\begin{aligned}	
	\mathcal{A}^L_2
		\leq& C_1(\|u_0\|_{H^1}) \big(\|\vec u_0\|^2_{\HH^3}+\|u_0\|^4_{\dot{H}^{2}}\big)
	+\frac{1}{4}\mathcal{A}^{L}_2							\\
&+C_2(\|u_0\|_{H^1})\mathcal{A}^{L}_2\Big[T^{\frac{1}{2}}\big(\mathcal{A}^{L}_2\big)^{\frac{1}{2}}
+ T \big(\mathcal{A}^{L}_2\big)^{\frac{1}{2}}+T\mathcal{A}^{L}_2\Big].
					\end{aligned}
				\end{equation}
Assume that $\mathcal{A}^L_2\leq a_2=2C_1\big(\|\vec u_0\|^2_{\HH^{3}}+\|u_0\|_{\dot{H}^{2}}^{4}\big)$. There exists $T=T(\|\vec u_0\|_{\HH^{3}})$  such that
	\begin{align}
		\label{conditionT1}
	&\quad C_2\Big[T^{\frac{1}{2}}\big(\mathcal{A}^{L}_2\big)^{\frac{1}{2}}
	+ T \big(\mathcal{A}^{L}_2\big)^{\frac{1}{2}}+T\mathcal{A}^{L}_2\Big]\notag\\
	&\leq C_2\big[T^{\frac{1}{2}}(a_2)^{\frac{1}{2}}+T(a_2)^{\frac{1}{2}}+Ta_2\big]\leq \frac{1}{8},
	\end{align}
	then  for any $t\in [0,T], $ we have
	\begin{equation}
		\label{res-E3}
	\begin{aligned}
	\quad\sup_{t\in [0,T]}\|(u, n,|\al\na|^{-1}n_t)\|^2_{H_x^{3}\times H_x^{2}\times H_x^{2}} 
		&\leq \frac{8}{5}C_1(\|u_0\|_{H^1}) \big(\|\vec u_0\|^2_{\HH^{3}}+\|u_0\|_{\dot{H}^{2}}^{4}\big)\\
			&\leq \frac{4}{5}a_2.
	\end{aligned}
\end{equation}
Then let $k\geq 2$. From \eqref{bootE1} and \eqref{conditionT1}, using Gagliardo-Nirenberg's inequality, we have
	\begin{equation}
		\label{res-E}
		\begin{aligned}
	\quad\sup_{t\in [0,T]}\|\vec u\|^2_{\HH_x^{k+1}} 
		&\leq C(\|u_0\|_{H^1}) \big(\|\vec u_0\|^2_{\HH^{k+1}}+\|u_0\|_{\dot{H}^{k}}^{4}\big)\\
		&\leq C(\|u_0\|_{H^1}) \big(1+\|u_0\|_{\dot{H}^{k-1}}^{2}\big)\|\vec u_0\|^2_{\HH^{k+1}}.
		\end{aligned}
		\end{equation}
Thus we complete the proof of Proposition \ref{prop:energy estimate}. 
\end{proof}

			\subsection{Proof of the uniform local well-posedness.}
			\begin{proof}[Proof of Proposition \ref{prop:lwpZ}]
				Let \(\vec u_0\in \HH^{k+1}\) with \(k\geq 2\), and set
				\[
				A:=C(\|u_0\|_{H^{k-1}})\|\vec u_0\|_{\HH^{k+1}},
				\]
				where \(C\) is the constant in \eqref{lwp}. By Corollary \ref{lem:lwp1}, there exists
				\(T_1^{(\al)}=T_1^{(\al)}(\|\vec u_0\|_{\HH^{k+1}},\al)>0\) such that the solution satisfies
				\[
				\vec u\in C\big([0,T_1^{(\al)}];\HH^{k+1}\big).
				\]
				On the other hand, Proposition \ref{prop:energy estimate} gives the uniform $\mathit{a\ priori}$ bound
				\[
				\sup_{0\leq t\leq T_1^{(\al)}}
				\|\vec u(t)\|_{\HH^{k+1}}
				\leq A,
				\]
	and
				$\vec u (T_1^{(\al)})\in \HH^{k+1}. $ 
				Using Corollary \ref{lem:lwp1} again, there exists 
			\[T_2^{(\al)}=T_2^{(\al)}\big(\big\|\vec u(T_1^{(\al)})\big\|_{\HH^{k+1}},\al\big)>0, \]  such that the solution can be extended further
				$$\vec u \in C\big([T_1^{(\al)},T_1^{(\al)}+T_2^{(\al)}];\HH^{k+1}\big). $$
				Therefore, by the continuation criterion associated with
				Corollary \ref{lem:lwp1}, the local solution can be extended
				to the interval \([0,T]\), that is
				\[
				\vec u\in C\big([0,T];\HH^{k+1}\big),
				\qquad
				\|\vec u\|_{L_t^\infty([0,T];\HH^{k+1})}\leq A.
				\]
				This completes the proof of Proposition \ref{prop:lwpZ}.
			\end{proof}

			\vspace{0.5cm} 
			
			\section{Proof of the subsonic limit}
			\label{sec:proofofmain}
			
			\vspace{0.5cm}
			
			\subsection{Basic settings for the subsonic limit}
			\begin{assu} 
				\label{assu:main}
				We make the following assumptions:
				\begin{enumerate}
					\item Let $\vec u_0\in \HH^3$. Let
					\(\vec u(t)\), \(v\), and \(w\) be the corresponding solutions to
					\eqref{equ-Z}, \eqref{equ-NLS}, and \eqref{equ-w}, respectively, with
					\[
					v_0=u_0,\qquad
					w_0=n_0+|u_0|^2,\qquad
					w_1=n_1+\partial_t|u|^2(0).
					\]
					\item Define the error functions
					    \[
					    r:=u-v,\qquad q:=n+|u|^2-w.
					    \]
				
				
				\end{enumerate}
			\end{assu}
\begin{remark}
		Under this assumption, one has $v_0\in H^3$, $w_0\in H^2$, and
		\(|\alpha\nabla|^{-1}w_1\in H^2\).
			\end{remark}
			
			Therefore, under Assumption \ref{assu:main}, $(r,q)$ solves the following system:
			\begin{align}
				\label{equ-r,q}
				\left\{
				\begin{aligned}
					&2i\partial_{t}r-\Delta r=|u|^2u-|v|^2v-(w+q)(v+r),\\
					&\alpha^{-2}\partial_{tt}q-\Delta q=\al^{-2}\partial_{tt}|u|^2,\\
					&r(0,x)\coloneqq r_0(x), \quad q(0,x)\coloneqq q_0(x), \quad \partial_{t}q(0,x)\coloneqq q_1(x),
				\end{aligned}
				\right.
			\end{align}
			with the initial data condition:
			\EQn{
				r_0\coloneqq u_0-v_0=0;\quad q_0\coloneqq n_0+|u_0|^2-w_0=0; \quad q_1\coloneqq n_1+\partial_t|u|^2(0)-w_1=0.
			}			
Next, we apply the scaling transform. Recall the definition $ S_\al f(t,x)=\al^{-1}f(\al^{-2}t, \al^{-1}x) $ in \eqref{scaling-S,W}, and we will frequently use the following relationship:
			\begin{align}
				\label{rel:Salf-f}
				\|S_\al f\|_{\mS^{\lambda}([0,\al^2T])}=\al^{-\lambda+\frac{1}{2}} \| f\|_{\mS^{\lambda}([0,T])}, \quad \mbox{ for } \lambda\in\R.
			\end{align} 
			We denote that
			\EQn{\label{eq:scaling}
			(u_{\al},n_{\al},v_{\al},w_{\al},\Phi_{\al},\Psi_{\al})\coloneqq S_{\al}(u,n,v,w,r,q).
			}
			By \eqref{equ-NLS} and \eqref{equ-w}, the rescaled profiles $v_{\al}$ and $w_{\al}$ satisfy the following equations, respectively:
			\begin{align}
				\label{equ-v_al}
				&	\left\{
				\begin{aligned}
					&2i\partial_{t}v_\al-\Delta v_\al=|v_\al|^2v_\al,\\
					&v_\al(0,x)=S_{\al}v_0(x),
				\end{aligned}
				\right.
				\\
				\label{equ-Ial}
				& 	\left\{
				\begin{aligned}
					&\partial_{tt}w_{\al}-\De w_{\al}=0, \\
					&w_{\al}(0,x)=S_{\al}w_0(x),\quad \partial_{t}w_{\al}(0,x)=\al^{-2}S_{\al}w_1(x). 
				\end{aligned}
				\right.	
			\end{align}
By \eqref{equ-Z}, we obtain that the rescaled solution $(u_\al,n_\al)$ satisfies the following system:
\begin{align}
\label{equ-u,n_al}
	&	\left\{
	\begin{aligned}
&2i\partial_{t}u_\al-\Delta u_\al=-\al^{-1}n_\al u_\al,\\
&\partial_{tt}n_\al -\De n_\al =\al \De|u_\al|^2,\\
&u_\al (0,x)=S_{\al}u_0(x),\quad n_\al (0,x)=S_{\al}n_0(x),\quad  \partial_tn_\al (0,x)=\al^{-2}S_{\al}n_1(x).
\end{aligned}
\right.
\end{align}

Moreover, by \eqref{equ-r,q}, we obtain the equations for the rescaled error functions $\Phi_{\al}$ and $\Psi_{\al}$,
			$\Phi_{\al}$ satisfies the equation:
			\begin{align}
				\label{equ-Phial}
				\left\{
				\begin{aligned}
					&2i\partial_{t}\Phi_{\al}-\De \Phi_{\al}=|u_{\al}|^2u_{\al}-|v_{\al}|^2v_{\al}-\al^{-1}(w_{\al}+\Psi_{\al})(v_{\al}+\Phi_{\al}), \\
					&\Phi_{\al}(0,x)=S_\al r_0=0,
				\end{aligned}
				\right.
			\end{align}
			and $\Psi_{\al}$ satisfies the equation:
			\begin{align}
				\label{equ-Psial}
				\left\{
				\begin{aligned}
					&\partial_{tt}\Psi_{\al}-\De \Psi_{\al}=\al \partial_{tt}|u_\al|^2, \\
					&\Psi_{\al}(0)=S_\al q_0, \quad \partial_t \Psi_{\al}(0,x)= \al^{-2}S_\al q_1=0.
				\end{aligned}
				\right.
			\end{align}

Before proving the subsonic limit, we give the local theory of the Zakharov system and the Schr\"odinger profile, including local well-posedness and uniform bound. Applying Proposition \ref{prop:lwpZ} and Lemma \ref{lem-NLS}, we have that
\begin{lem}[Local theory for subsonic limit]\label{lem:local-theory}
Let Assumption \ref{assu:main} hold. Then, there exists $T=T(\norm{\vec u_0}_{\HH^3},\norm{v_0}_{H^1})$ such that
\begin{enumerate}
\item \eqref{equ-Z} admits a unique solution $\vec u\in C([0,T];\HH^3)$, and \eqref{equ-NLS} admits a unique solution $v\in C([0,T];H^1)$. Moreover,
\EQ{
\| \vec u\|_{L_t^\I([0,T];\HH^{3})}   \leq C\|\vec u_0\|_{\HH^{3}},
}
where the constant $C$ depends on $\|\vec u_0\|_{\HH^{1}}$.
\item 
Given $k>2$. If further assume that $\vec u_0\in\HH^{\lceil k+1\rceil}$, then the solution of \eqref{equ-Z}  satisfies $\vec u\in C([0,T];\HH^{\lceil k+1\rceil})$ and
\EQ{
\| \vec u\|_{L_t^\I([0,T];\HH^{ k+1})}   \leq C\|\vec u_0\|_{\HH^{\lceil k+1\rceil}},
}
where the constant $C$ depends on $\|\vec u_0\|_{\HH^{\lceil k-1\rceil}}$.
\item 
If $v_0\in H^s$ for some $s\geq1$, then  $\norm{v}_{X^s(I)}\le C\norm{v_0}_{H^s}$.
\end{enumerate}
In the following, we denote
\[
I\coloneqq [0,T], \qquad I^\alpha\coloneqq [0,\alpha^2T].
\]
We keep the dependence on \(\vec u_0\), \(v_0\), and
\((w_0,|\alpha\nabla|^{-1}w_1)\) explicit in the intermediate estimates
in order to indicate the origin of the constants. Under Assumption \ref{assu:main}, these
dependencies will be absorbed into the dependence on the Zakharov initial
data in the proof of Theorem \ref{mainthm}.
\end{lem}

	\subsection{Main nonlinear estimates}
	Now, we prove the nonlinear estimate in the more regular case. Recall that $\Phi_{\alpha}$ and $\Psi_{\alpha}$ solve the equations \eqref{equ-Phial} and \eqref{equ-Psial}, respectively.
\begin{prop}[Nonlinear estimates]
	\label{prop:estPsiPhi}
	Let Assumption \ref{assu:main} hold. 
		Given $\gamma\in[0,\frac{1}{2}]$, $0\leq \kappa_1 \leq \frac{1}{2},$ and $ -\frac{1}{2}\leq \kappa_2\leq 1$. Assume that
	$\vec u_0\in \HH^{\lceil3+\gamma-\kappa_1\rceil}\cap\HH^{\lceil3+\gamma+\kappa_2\rceil}$.
		Then for any $\al\ges1$,
		\begin{equation}
		\label{eq:nonlinear-estimate-psi}
		\begin{aligned}
&\quad\big\|(\Phi_\al,\Psi_\al)\big\|_{\mS^{\gamma}\times\mW^{\gamma}(I^\al)}\\
	&\leq 
C_1 \al^{-\gamma-\frac{3}{2}+\kappa_1 }\|\vec{u}\|_{L^\infty_t(I;\HH^{3+\gamma-\kappa_1})}
\Bigl[1+\|(w_0, |\al\na|^{-1}w_1)\|_{H^{1+\gamma}\times H^{1+\gamma}}\Bigr] \\
	&\quad +C_1 \al^{-\gamma-\frac{1}{2}-\kappa_2} \|\vec{u}\|_{L^\infty_t(I;\HH^{3+\gamma+\kappa_2})}  \\
&\quad + C_2 \|(w_0,|\al\na|^{-1}w_1)\|_{H^{1+\gamma}\times H^{1+\gamma}}
		\Bigl(\al^{-\gamma-\frac{1}{2}} + \al^{-\gamma-1+\frac{\varepsilon}{2}}\|v\|_{X^{2+\gamma}(I)}\Bigr) \\
		&\quad + C_2 \|(\Phi_\alpha, \Psi_\alpha)\|_{\mS^\gamma \times \mW^\gamma(I^\alpha)}
		\Bigl( \|\Phi_{\al}\|^2_{\mS^{\frac{1}{2}}(I^\al)}
		+ T^{\frac{1}{2}} \big\|(\Phi_{\al},\Psi_\al)\big\|_{\mS^{\frac{1}{2}}\times \mW^{\frac{1}{2}}(I^{\al})}
		+ T^{\frac{1+2\gamma}{4}} \Bigr);
				\end{aligned}
		\end{equation} 	
	where the constants \(C_1,C_2\) are given by \[
	C_1=C_1\big(\|\vec u_0\|_{\HH^3}\big),\qquad
	C_2=C_2\big(
\|(v_0,w_0,|\al\na|^{-1}w_1)\|_{H^{1+\gamma}\times H^{1+\gamma}\times H^{1+\gamma}}
	\big).
	\]
		\end{prop}
					
		We will prove Proposition \ref{prop:estPsiPhi} in Sections \ref{sec:estq} and \ref{sec:estr}. To start with,  we gather some useful estimates on $(u_\al,n_\al)$ of \eqref{equ-u,n_al}, the profiles $v_\al$ and $w_\al$. First, we need the $U^2$-estimate for $u_\al$:
\begin{lem}
\label{lem:estU2-u_al}
Let Assumption \ref{assu:main} hold. Assume $\vec{u}_0\in\HH^{\lceil s\rceil+1}$ with $s\geq0$, then
\begin{align}
&\||\na|^{s}P_N u_\al\|_{\ell^2_NU_{\De}^2(I^\al;L^2_x)} \leq C\al^{-s+\frac{1}{2}}\|(u,n)\|_{L^\infty_t(I;H^s\times H^s)} \Big(1+T\|(u,n)\|_{L^\infty_t(I;H^2\times H^2)}\Big),
\end{align}
where $I^\al,I$, and $(u_\al,n_\al)$ are given as above.
\end{lem}

\begin{proof}
Using Duhamel's formula of $u_\al$, Lemma \ref{Lem:linearforup}, the embedding $V_\De^2(\R;L^2_x)\hookrightarrow L^\infty_tL^2_x(\R),$  H\"older's inequality, Minkowski's inequality, and Lemma \ref{lem:Little},
\begin{align}
\label{est:upu_al}
&\quad\||\na|^sP_Nu_\al\|_{\ell^2_NU_{\De}^2(I^\al;L^2_x)}\notag\\
&\les\|e^{-\frac{i}{2}t\De}|\na|^{s}P_NS_\al u_0\|_{\ell^2_NU_{\De}^2(I^\al;L^2_x)}
+\al^{-1}\left\|\big\|\int_{0}^{t}e^{-\frac{i}{2}(t-s)\De}|\na|^{s}P_N(n_\al u_\al)\ds \big\|_{U_{\De}^2(I^\al;L^2_x)}\right\|_{\ell^2_N}\notag\\
&\les \|P_NS_\al u_0\|_{\ell^2_N\dot{H}^s}+\al^{-1}\Big\|\sup_{\|h\|_{V^2_{\De}(I^\al;L^2_x)}=1}\big|\int_{I^\al}\int|\na|^{s}P_N(n_\al u_\al)\bar{h}(t) \dx\dt\big|\Big\|_{\ell^2_N}
\notag\\
&\les  \|S_\al u_0\|_{\dot{H}^s}+\al^{-1}\||\na|^{s}P_N(n_\al u_\al)\|_{\ell^2_NL^1_tL^2_x(I^\al)}
\notag\\
&\les \|S_\al u_0\|_{\dot{H}^s}+\al^{-1}\||\na|^{s}P_N(n_\al u_\al)\|_{L^1_tL^2_x\ell^2_N(I^\al)}
\notag\\
&\les \|S_\al u_0\|_{\dot{H}^s}+\al^{-1}\||\na|^{s}(n_\al u_\al)\|_{L^1_tL^2_x(I^\al)}.
\end{align}
Combining \eqref{rel:Salf-f} and H\"{o}lder's inequality, we have
\begin{align*}
&\quad\al^{-1}\||\na|^{s}(n_\al u_\al)\|_{L^1_tL^2_x(I^\al)}\\
&\les \al^{-1}|I^\al|\big[\||\na|^sn_\al \|_{L^\infty_tL^2_x(I^\al)}\|u_\al \|_{L^\infty_{t,x}(I^\al)}+\||\na|^su_\al \|_{L^\infty_tL^2_x(I^\al)}\|n_\al \|_{L^\infty_{t,x}(I^\al)}\big]
\\
&\les  \al^{-s+\frac{1}{2}}T\|(u,n)\|_{L^\infty_t(I;H^s\times H^s)}\|(u,n)\|_{L^\infty_t(I;H^2\times H^2)}.
\end{align*}
Thus we complete the proof. 
\end{proof}

We also need the following estimate for the solution $v_\al$ of \eqref{equ-v_al}:
\begin{lem}
Let $ s\geq0$, $v_0\in H^s$, and let $v_\al$ and $v$ be the solutions of \eqref{equ-v_al} and \eqref{equ-NLS}, respectively. Then
\begin{align}
\label{lem:est-v_al}
&\|v_\al\|_{\mS^s (I^\al)} =\al^{-s+\frac{1}{2}}\|v\|_{L^\infty_t\dot{H}^s_x\cap L^2_t\dot{W}^{s,6}_x (I)} \leq C \al^{-s+\frac{1}{2}}\|v_0 \|_{H^s\cap H^1},
\end{align}
where $I^\al,I,v_\al$ are given as above.
\end{lem}
\begin{proof}
The proof of \eqref{lem:est-v_al} can be obtained by \eqref{rel:Salf-f} and Lemma \ref{lem-NLS}.
\end{proof}

			We define $w_{\al,\pm}(t,x)\coloneqq w_\al \pm i\p_t |\na|^{-1} w_\al$, then 
			\begin{align}
				\label{deco:w_al}
				w_\al=\frac{1}{2}\big(w_{\al,+}+w_{\al,-}\big),
			\end{align}
			and $w_{\al,\pm}$ solve the following equation:
			\begin{align}
				\label{equ-wal_pm}
				& 	\left\{
				\begin{aligned}
					&(i\p_t\mp|\na|)w_{\al,\pm}=0, \\
					&w_{\al,\pm}(0)=S_\al\big( w_0 \pm i |\al\na|^{-1}w_1\big). 
				\end{aligned}
				\right.	
			\end{align}
			Then the estimate of the solution $w_\al$ in \eqref{equ-Ial} can be obtained from Lemma \ref{lem:stri} as follows:
			\begin{lem}
				\label{lem-wave}
				Let $s\geq 0$, $(w_0,|\al\na|^{-1}w_1) \in \dot{H}^s\times \dot{H}^s$, and $w_{\al,\pm}$ be the solution to \eqref{equ-wal_pm}. Then 
\begin{align}
\label{lem:est:w_al(in,2)}
&\|w_{\al,\pm}\|_{L^\infty_t \dot{H}^{s}_x(\R)} \leq C\al^{-s+\frac{1}{2}} \norm{(w_0,|\al\na|^{-1}w_1)}_{\dot{H}^s\times \dot{H}^s},	
		\\
\label{lem:est:w_al(2,in)}
&	\|w_{\al,\pm}\|_{
	L^\frac{2}{1-\varepsilon}_t \dot{W}^{s-1+\varepsilon,\frac{2}{\varepsilon}}_x(\R)} \leq C\al^{-s+\frac{1}{2}} \norm{(w_0,|\al\na|^{-1}w_1)}_{\dot{H}^s\times \dot{H}^s},
		\\
\label{lem:est:U2w_al}
&\||\na|^sP_Nw_{\al,\pm}\|_{\ell^2_NU_{\pm|\na|}^2(\R;L^2_x)}
\leq C\al^{-s+\frac{1}{2}} \norm{(w_0,|\al\na|^{-1}w_1)}_{\dot{H}^s\times \dot{H}^s}.	
\end{align}
\end{lem}
\begin{proof}
The proof can be obtained directly by Lemma \ref{Lem:linearforup}, Lemma \ref{lem:stri},  and scaling.

\end{proof}

			\begin{lem}
				\label{lem:part}
				For any  $f(t,x)$ such that $f\in L_t^\I \dot H_x^{-1}$ and $\pd_tf \in L_t^1 \dot H_x^{-1}$, then
				\begin{align}
					\label{part-inter}
					\Big\|\int^{t}_{0}e^{- i(t-s)|\na|}f(s)\ds\Big\|_{L_{t}^\infty L^2_x(\R)}
					&\lesssim\||\na|^{-1}f\|_{L_{t}^\infty L^2_x(\R)}+\||\na|^{-1}\partial_tf\|_{L_{t}^1 L^2_x(\R)}.
				\end{align}
			\end{lem}
			
			\begin{proof}
				The proof can be obtained by Lemma \ref{lem:stri} and integration-by-parts. Since $e^{it|\na|}$ is a unitary operator, we only need to consider $\|\int^{t}_{0}e^{is|\na|}f(s,x)\ds\|_{L^\infty_tL^2_x(\R)}$.
				Taking the Fourier transform of $\int^{t}_{0}e^{is|\na|}f(s,x)\ds$, we have 
				\begin{align}
					\label{lem-F1}
					\F \big(\int^{t}_{0} e^{is|\na|}f(s,x)\ds\big)(\xi)
					&=	\int_{0}^{t}
					e^{is|\xi|} 	\hat{f}(s,\xi)\ds.
				\end{align}
				Then using integration-by-parts, we further get that 
				\begin{align*}
					\eqref{lem-F1}
					=-i
					e^{is|\xi|} |\xi|^{-1}	\hat{f}(s,\xi)\Big|_{s=0}^{s=t}
					+i\int_{0}^{t}
					e^{is|\xi|} |\xi|^{-1}	\partial_s\hat{f}(s,\xi)\ds.
				\end{align*}
				We now take the inverse Fourier transform of above identity to obtain
				\begin{align*}
					\int^{t}_{0} e^{is|\na|}f(s,x)\ds
					=&-ie^{is|\na|}|\na|^{-1}f(s,x)\Big|_{s=0}^{s=t}
					+i\int_{0}^{t}e^{is|\na|}|\na|^{-1}\partial_sf(s,x)\ds.
				\end{align*}
				Then using Lemma \ref{lem:stri}, we complete the proof. 
			\end{proof}
			
\subsection{Nonlinear estimates I: the wave remainder}\label{sec:estq}
			
	\begin{lem}
	\label{lem:nonlinear-estimates-psi}
Under the assumptions of Proposition \ref{prop:estPsiPhi}, for any $\al\ges1$,
\begin{align*}
	\big\|\Psi_\al\big\|_{\mW^{\gamma}(I^\al)}
	&\leq   C_1\al^{-\gamma+\frac{1}{2}}\Big\{
			 \al^{-2+\kappa_1}\|\vec{u}\|_{L^\infty_t(I;\HH^{3+\gamma-\kappa_1}) }\big[1+\|(w_0, |\al\na|^{-1}w_1)\|_{H^{1+\gamma}\times H^{1+\gamma}}\big]\\
			 &\qquad\qquad\qquad+
			 \al^{-1-\kappa_2}\|\vec{u}\|_{L^\infty_t(I;\HH^{3+\gamma+\kappa_2})}\Big\},
	\end{align*}
	where $C_1$ is a constant only depending on $\|\vec u_0\|_{\HH^3}$.
\end{lem}
			
			Recalling the equation of $\Psi_{\al}$ in \eqref{equ-Psial}, 
			and denoting $\Gamma_\al\coloneqq|\nabla|^{-1}(\partial_t-i|\nabla|)\Psi_{\al}$, we have 
			\begin{align}
				\label{exp:ptPsi}
				\Psi_{\al}=-\im\Gamma_\al,\quad  \partial_t\Psi_{\al}=\re|\na|\Gamma_\al.
			\end{align}
			$\Gamma_\al$ satisfies the equation:
			\begin{align}
				\label{equ-Gamma_al}
				\left\{
				\begin{aligned}
					&\big(\partial_{t}+i|\nabla|\big) \Gamma_{\al}=\al|\nabla|^{-1}\partial_{tt}|u_\al|^2, \\
					&\Gamma_{\al}(0)=-iS_\al\big[q_0+i|\al\nabla|^{-1}q_1\big]=0.
				\end{aligned}
				\right.
			\end{align}
			By \eqref{equ-u,n_al}, we have
			\begin{align}
				\label{equ-u_al}
				2i\partial_{t}u_\al-\De u_\al=-\al^{-1}n_\al u_\al. 
			\end{align}
			Multiplying \eqref{equ-u_al} by $i\bar{u}_\alpha$ and taking the real part, we have
			\begin{align}
				\label{ptu2}
				\partial_{t}|u_\al|^2=-\re (\De u_{\al}i\bar{u}_\alpha)=-\re \nabla\cdot  (\nabla u_{\al} i\bar{u}_\alpha).
			\end{align}
Differentiating  with respect to $t$ yields
\begin{align}
\label{exp:ptu2}
\partial_{tt}|u_\al|^2=-\re \nabla\cdot  (\nabla\partial_t u_{\al} i\bar{u}_\alpha)-\re \nabla\cdot  (\nabla u_{\al} i\partial_t\bar{u}_\alpha).
\end{align}
			Using \eqref{equ-u_al}, we have
			\begin{equation}
				\label{pttu2}
				\begin{aligned}
					\partial_{tt}|u_\al|^2&=-\frac{1}{2}\re \nabla\cdot  (\nabla\De u_{\al} \bar{u}_\alpha)+\frac{1}{2}\re \nabla\cdot  (\De u_\alpha \nabla \bar{u}_{\al} )
					\\
					&\quad +\frac{1}{2}\al^{-1}\re \nabla\cdot  \big[\nabla (n_\al u_{\al}) \bar{u}_\alpha\big]-\frac{1}{2}\al^{-1}\re \nabla\cdot  (n_\al\bar{u}_\alpha \nabla  u_{\al} ).
				\end{aligned}
			\end{equation}
In addition, we rewrite the equation of $\Gamma_\al$ in integral form as
			\begin{subequations}
				\begin{align}	
					\Gamma_\al(t,x)	=
					&-	\frac{1}{2}\al \int_{0}^{t}e^{-i(t-s)|\na|}|\nabla|^{-1}\re \nabla\cdot  (\nabla\De u_{\al} \bar{u}_\alpha)\ds\label{2-1}
					\\
					&+\frac{1}{2}\al \int_{0}^{t}e^{-i(t-s)|\na|}|\nabla|^{-1}\re \nabla\cdot  (\De u_\alpha \nabla \bar{u}_{\al} )\ds
					\label{2-2}\\
					&+\frac{1}{2} \int_{0}^{t}e^{-i(t-s)|\na|}|\nabla|^{-1}\re \nabla\cdot  \big[\nabla (n_\al u_{\al}) \bar{u}_\alpha-n_\al\bar{u}_\alpha \nabla  u_{\al}\big]\ds \label{2-3}.
				\end{align}
			\end{subequations}

\subsubsection{ Quadratic term estimates.} 
We deal with the terms \eqref{2-1} and \eqref{2-2}.
First of all, we focus on  \eqref{2-1}. Since $$\re(\bar{f}g)=\frac{1}{2}\big(\bar{f}g+f\bar{g}\big).$$ In our proof, the term  $\nabla\De u_{\al} \bar{u}_\alpha $ makes no essential difference from $\nabla\De \bar{u}_{\al} u_\alpha. $ To simplify the analysis, we therefore only need to consider the following nonlinear term in our estimation:
\begin{align}
\label{main2-1}
\frac{1}{2}\al \int_{0}^{t}e^{-i(t-s)|\na|}|\nabla|^{-1} \nabla\cdot  (\nabla\De u_{\al} \bar{u}_\alpha)\ds.
\end{align}
Using Littlewood-Paley dyadic projection operator, we have
	\begin{subequations}
\begin{align}\label{N1gesN2}
||\na|^{\gamma}\eqref{main2-1}|\les &
\al \Big|\sum_{N_1\les N_2}\int_{0}^{t}e^{-i(t-s)|\na|}|\na|^{-1+\gamma}\nabla\cdot  (P_{N_1}\nabla\De u_{\al} \,P_{N_2}\bar{u}_\alpha) \ds\Big|
	\\
\label{N1llN2low}
&+	\al \Big|\sum_{N_1\gg N_2:N_1\ll1}\int_{0}^{t}e^{-i(t-s)|\na|}|\na|^{-1+\gamma}\nabla\cdot  (P_{N_1}P_{\ll1}\nabla\De u_{\al}\, P_{N_2}\bar{u}_\alpha) \ds\Big|
\\
\label{N1llN21}
&+\al \Big|\sum_{N_1\gg N_2: N_1\sim1}\int_{0}^{t}e^{-i(t-s)|\na|}|\na|^{-1+\gamma}\nabla\cdot (P_{N_1}P_{\sim1}\nabla\De u_{\al}\, P_{N_2}\bar{u}_\alpha) \ds \Big|
		\\
\label{N1llN2high}
&+\al \Big| \sum_{N_1\gg N_2:N_1\gg1}\int_{0}^{t}e^{-i(t-s)|\na|}|\na|^{-1+\gamma}\nabla\cdot (P_{N_1}P_{\gg1}\nabla\De u_{\al}\, P_{N_2}\bar{u}_\alpha)\ds\Big|.
\end{align}
\end{subequations}
Then we consider \eqref{N1gesN2}, \eqref{N1llN2low}, \eqref{N1llN21}, and \eqref{N1llN2high} case by case.
			
$\bullet$ Estimate for  $\eqref{N1gesN2}$. 
Using Lemma \ref{lem:part}, we have
\begin{subequations}
\begin{align}
\label{N1gesN2-boud}
\|\eqref{N1gesN2}\|_{L^\infty_tL^{2}_x(I^\al)}
&\les
\al\big\|\sum_{N_1\les N_2}|\na|^{-1+\gamma} \big(P_{N_1}|\na|^3 u_{\al} \,P_{N_2}\bar{u}_\alpha  \big)\big\|_{L^\infty_tL^{2}_x(I^\al)}
\\
\label{N1gesN2-inte1}
&\quad
+ \al\big\|\sum_{N_1\les N_2}|\na|^{-1+\gamma} \big(P_{N_1}|\na|^3 \partial_t u_{\al} \,P_{N_2}\bar{u}_\alpha  \big)\big\|_{L^1_tL^{2}_x(I^\al)}
\\
	\label{N1gesN2-inte2}
&\quad
+ \al\big\|\sum_{N_1\les N_2}|\na|^{-1+\gamma} \big(P_{N_1}|\na|^3  u_{\al} \, P_{N_2}\partial_t\bar{u}_\alpha \big)\big\|_{L^1_tL^{2}_x(I^\al)}.
\end{align}
\end{subequations}
Since the treatments of \eqref{N1gesN2-inte1} and \eqref{N1gesN2-inte2}  are completely analogous, we provide the detailed proof only for the former two terms. Using Sobolev's embedding $\dot{W}^{1,\frac{6}{5}}\hookrightarrow L^2$ in $x$, Bernstein's inequality and Lemma \ref{lem:Schur}, we have
\begin{align}
\label{step1ofN1gesN2}
\eqref{N1gesN2-boud}
&\les \al\big\|\sum_{N_1\les N_2}|\na|^{-1+\gamma} \big(P_{N_1}|\na|^3 u_{\al}\, P_{N_2}\bar{u}_\alpha  \big)\big\|_{L^\infty_tL^{2}_x(I^\al)}
\notag\\
&\les \al\big\|\sum_{N_1\les N_2}|\na|^{\gamma} \big(P_{N_1}|\na|^3 u_{\al}\, P_{N_2}\bar{u}_\alpha  \big)\big\|_{L^\infty_tL^{\frac{6}{5}}_x(I^\al)}
	\notag \displaybreak[3] \\
&\les \al\big\|\sum_{N_1\les N_2}\frac{N_1^{\frac{3}{2}}}{N_2^{\frac{3}{2}}} \big\| N_1^{-\frac{3}{2}} P_{N_1}|\na|^3 u_{\al}\, N_2^{\frac{3}{2}+\gamma}P_{N_2}\bar{u}_\alpha \big\|_{L^{\frac{6}{5}}_x}\big\|_{L^\infty_t(I^\al)}	\notag\\
&\les \al\big\|\sum_{N_1\les N_2}\frac{N_1^{\frac{3}{2}}}{N_2^{\frac{3}{2}}} \big\| N_1^{-\frac{3}{2}} P_{N_1}|\na|^3 u_{\al}\big\|_{L^{3}_x}\, \big\|N_2^{\frac{3}{2}+\gamma}P_{N_2}\bar{u}_\alpha \big\|_{L^{2}_x}\big\|_{L^\infty_t(I^\al)}
	\notag\\
&\les \al\big\| N_1^{-\frac{3}{2}} P_{N_1}|\na|^{\frac{7}{2}} u_{\al}\big\|_{L^\infty_tL^{2}_x\ell_{N_1}^2(I^\al)}\big\|N_2^{\frac{3}{2}+\gamma} P_{N_2}\bar{u}_\alpha \big\|_{L^\infty_tL^2_x\ell_{N_2}^2(I^\al)},
 \end{align}
then by Lemma \ref{lem:Little} and \eqref{rel:Salf-f}, we further have
\begin{align}
\label{est:N1gesN2-boud}
\eqref{N1gesN2-boud}
&\les \al\|u_\al\|_{L^\infty_t\dot{F}^{2,2}_{2}(I^\al)}\|u_\al\|_{L^\infty_t\dot{F}^{\frac{3}{2}+\gamma,2}_{2}(I^\al)}	
\notag\\
&\lesssim
\al\|u_\al\|_{L^\infty_t\dot{H}^{2}_x(I^\al)}\|u_\al\|_{L^\infty_t\dot{H}^{\frac{3}{2}+\gamma}_x(I^\al)}
	\notag\\
&\lesssim
\al^{-\frac{3}{2}-\gamma}\|u\|_{L^\infty_t\dot{H}^{2}_x(I)}\|u\|_{L^\infty_t\dot{H}^{\frac{3}{2}+\gamma}_x(I)}.
\end{align}
From the equation of $u_\al$ in  \eqref{equ-u_al} and Sobolev's embedding $\dot{W}^{1,\frac{6}{5}}\hookrightarrow L^2$ in $x$, we have
\begin{subequations}
\begin{align}
\label{N1gesN2-inte1-1}
\eqref{N1gesN2-inte1}
&\les
\al\big\|\sum_{N_1\les N_2}|\na|^{\gamma} \big(P_{N_1}|\na |^5u_{\al}\, P_{N_2}\bar{u}_\alpha  \big)\big\|_{L^1_tL^{\frac{6}{5}}_x(I^\al)}
\\
\label{N1gesN2-inte1-2}
&\quad
+\big\|\sum_{N_1\les N_2}|\na|^{\gamma} \big(P_{N_1}|\na|^3(n_\al u_{\al})\,  P_{N_2}\bar{u}_\alpha  \big)\big\|_{L^1_tL^{\frac{6}{5}}_x(I^\al)}.
\end{align}
\end{subequations}
Using the same argument as above, we obtain
\begin{align}
\label{est:N1gesN2-inte1-1}
\eqref{N1gesN2-inte1-1}
&\les \al\big\|\sum_{N_1\les N_2}\frac{N_1^{\frac{5}{2}}}{N_2^{\frac{5}{2}}}  N_1^{-\frac{5}{2}} \big(P_{N_1}|\na|^5 u_{\al} \,N_2^{\frac{5}{2}+\gamma}P_{N_2}\bar{u}_\alpha  \big)\big\|_{L^1_tL^{\frac{6}{5}}_x(I^\al)}
\notag\\
&\les\al|I^\al|\big\| N_1^{-\frac{5}{2}} P_{N_1}|\na|^{\frac{11}{2}} u_{\al}\big\|_{L^\infty_tL^{2}_x\ell_{N_1}^2(I^\al)}\big\|N_2^{\frac{5}{2}+\gamma}P_{N_2}\bar{u}_\alpha \big\|_{L^\infty_tL^2_x\ell_{N_2}^2(I^\al)}
\notag\\
&\les
\al^{-\frac{3}{2}-\gamma}T\|u\|_{L^\infty_t\dot{H}^{3}_x(I)}\|u\|_{L^\infty_t\dot{H}^{\frac{5}{2}+\gamma}_x(I)},
\end{align}
similarly,
\begin{align}
\label{est:N1gesN2-inte1-2}
\eqref{N1gesN2-inte1-2}
&\les \big\|\sum_{N_1\les N_2}\frac{N_1^{2}}{N_2^{2}}  N_1^{-2} \big(P_{N_1}|\na|^3(n_ {\al}u_\al) \, N_2^{2+\gamma} P_{N_2}\bar{u}_\alpha \big)\big\|_{L^1_tL^{\frac{6}{5}}_x(I^\al)}
\notag\\
&\les|I^\al|\big\| N_1^{-2} P_{N_1}|\na|^{\frac72}(n_\al u_{\al}) \big\|_{L^\infty_tL^{2}_x\ell_{N_1}^2(I^\al)}\big\|N_2^{2+\gamma}P_{N_2}\bar{u}_\alpha \big\|_{L^\infty_tL^2_x\ell_{N_2}^2(I^\al)}
\notag\\
&\les
\al^{-\frac{3}{2}-\gamma}T\|u\|_{L^\infty_t\dot{H}^{2+\gamma}_x(I)}\|(u,n)\|^2_{L^\infty_t(I;H^{\frac{5}{2}}\times H^{\frac{3}{2}})}.
\end{align}
Combining \eqref{est:N1gesN2-boud}, \eqref{est:N1gesN2-inte1-1},
and  \eqref{est:N1gesN2-inte1-2},
we have
\begin{align}
\label{est:I1}
\|\eqref{N1gesN2}\|_{L^\infty_tL^{2}_x(I^{\al})}
\les \al^{-\frac{3}{2}-\gamma}\|u\|_{L^\infty_tH^{\frac{5}{2}+\gamma}_x(I)}\big[\|u\|_{L^\infty_tH^{3}_x(I)}+T\|\vec{u}\|^2_{L^\infty_t(I;\HH^{3})}\big].
\end{align}

	$\bullet$ Estimate for  $\eqref{N1llN2low}$. We have
\begin{align*}
\|\eqref{N1llN2low}\|_{L^\infty_tL^{2}_x(I^{\al})}
&\les \al \Big\|\sum_{N_1\gg N_2:N_1\ll1} \int_{0}^{t} e^{is|\na|}|\na|^{\gamma}P_{N_1}|\na|^3u_\al \, P_{N_2}\bar{u}_{\al}\ds\Big\|_{L^\infty_tL^{2}_x(I^{\al})}.
\end{align*}
Recalling the equation of  $u_\al$ in  \eqref{equ-u_al}, by Duhamel's formula, we have
\begin{align}
\label{dum:v_al}
u_{\al}(s)&=e^{-\frac{i}{2}s\De}u_{\al}(0)+\al^{-1}\frac{i}{2}\int_{0}^{s}	e^{-\frac{i}{2}(s-y)\De} n_\al u_\al(y) \dy
\notag\\
&\coloneqq e^{-\frac{i}{2}s\De}\rho_{\al}(s,x)
\end{align} 
and 
\begin{align}
	\label{psrho}
\partial_s\rho_{\al}(s,x)=\al^{-1}\frac{i}{2}e^{\frac{i}{2}s\De}n_\al u_\al(s).
\end{align}
Taking the Fourier transform, we have that
\begin{align}
\label{foueier2-1}
&\quad\F \Big(\sum_{N_1\gg N_2:N_1\ll1} \int_{0}^{t} e^{is|\na|}|\na|^{\gamma}P_{N_1}|\na|^3u_\al P_{N_2}\bar{u}_{\al}\ds\Big)	(\xi)
\notag\\
&= \sum_{N_1\gg N_2:N_1\ll1} \int_{0}^{t}\int e^{is\phi_1}|\xi|^{\gamma}|\xi_1|^3\big(\psi_{N_1}\widehat{P_{\ll1}\rho_\al}\big)(\xi_1)\big(\psi_{N_2}\hat{\bar{\rho}}_{\al}\big)(\xi-\xi_1)\dxio\ds,
\end{align}
where $\phi_1$ is the phase function denoted by $|\xi|+\frac{|\xi_1|^2}{2}-\frac{|\xi-\xi_1|^2}{2}.$
We observe that 
$$\phi_1\sim |\xi|\sim |\xi_1|,\mbox{ when }1\gg|\xi_1|\sim|\xi|\gg|\xi-\xi_1|.$$  
Noting that \begin{align*} e^{is\phi_1}=\frac{1}{i\phi_1}\frac{d}{ds}\big(e^{is\phi_1}\big), \end{align*} then using integration-by-parts, we further get that \begin{equation} \begin{aligned} &\eqref{foueier2-1} \notag\\ =& \sum_{N_1\gg N_2:N_1\ll1}\Big[e^{is\phi_1}\int \frac{|\xi|^{\gamma}|\xi_1|^{1-\gamma}}{i\phi_1}|\xi_1|^{2+\gamma}\big(\psi_{N_1}\widehat{P_{\ll1}\rho_\al}\big)(\xi_1)\big(\psi_{N_2}\hat{\bar{\rho}}_{\al}\big)(\xi-\xi_1) \dxio\Big] \Big|_{s=0}^{s=t} \\ & -\sum_{N_1\gg N_2:N_1\ll1} \int_{0}^{t}\int e^{is\phi_1}\frac{|\xi|^{\gamma}|\xi_1|^{1-\gamma}}{i\phi_1} |\xi_1|^{2+\gamma}\big(\psi_{N_1}\partial_s\widehat{P_{\ll1}\rho_\al}\big)(\xi_1) \big(\psi_{N_2}\hat{\bar{\rho}}_{\al}\big)(\xi-\xi_1)\dxio\ds \\ & -\sum_{N_1\gg N_2:N_1\ll1} \int_{0}^{t}\int e^{is\phi_1}\frac{|\xi|^{\gamma}|\xi_1|^{1-\gamma}}{i\phi_1}\big(\psi_{N_1}\widehat{P_{\ll1}\rho_\al}\big)(\xi_1) \big(\psi_{N_2}\partial_s\hat{\bar{\rho}}_{\al}\big)(\xi-\xi_1)\dxio\ds. 
\end{aligned} 
\end{equation}
Denote the bi-linear multiplier  $T_{m_1}(f,g)(x)$ by 
\begin{align}
\label{T_m_3}
\F\big(T_{m_1}(f,g)\big)(\xi)\coloneqq \sum_{N_1\gg N_2:N_1\ll1}  
\int\frac{|\xi|^{\gamma}|\xi_1|^{1-\gamma}}{i\phi_1}
\psi_{N_1}(\xi_1)\psi_{N_2}(\xi-\xi_1)  \hat{f}(\xi_1)\hat{g}(\xi-\xi_1)\dxio
\end{align}
with the symbol
\begin{align}
\label{def:m_1}
m_1(\xi_1,\xi-\xi_1)= \sum_{N_1\gg N_2:N_1\ll1} \frac{|\xi|^{\gamma}|\xi_1|^{1-\gamma}}{i\phi_1}	\psi_{N_1}(\xi_1)\psi_{N_2}(\xi-\xi_1).  
\end{align}
Note that for any $|\xi_1|\sim N_1$ and $|\xi-\xi_1|\sim N_2$ with $N_1\gg N_2$ and $N_1\ll 1$,
\EQ{
\absb{\frac{|\xi|^{\gamma}|\xi_1|^{1-\gamma}}{i\phi_1}} \lesssim \frac{|\xi|}{|\phi_1|} \lesssim 1.}
Then we have
\begin{align*}
|m_1(\xi_1,\xi-\xi_1)|\les \sum_{N_1\gg N_2:N_1\ll1}  |\psi_{N_1}(\xi_1)\psi_{N_2}(\xi-\xi_1)|\les \psi(\xi_1)\psi(\xi-\xi_1)\les 1,
\end{align*} 
and $	m_1(\xi_1,\xi_2)$ satisfies the condition in Lemma \ref{lem:Coif}.
Moreover, by $e^{is\phi_1} = e^{is|\xi|} e^{\frac{i}{2}s|\xi_1|^2} e^{-\frac{i}{2}s|\xi-\xi_1|^2},$ and \eqref{dum:v_al}, it allows us to write 
\begin{subequations}
\begin{align}	
&\al\sum_{N_1\gg N_2:N_1\ll1} \int_{0}^{t} e^{is|\na|}|\na|^{\gamma}P_{N_1}|\na|^3u_\al P_{N_2}\bar{u}_{\al}\ds\notag\\
=&\al\Big[e^{is|\na|} T_{m_1}\big(P_{\ll1}|\na|^{2+\gamma}u_{\al}, \bar{u}_{\al}\big)(s,x)\Big] \Big|_{s=0}^{s=t}
	\label{F2-1intparts1b}		\\
\label{F2-1intparts1-in-1}	&-\int_{0}^{t} e^{is|\na|}T_{m_1}\big(P_{\ll1}|\na|^{2+\gamma}(n_\al u_{\al}), \bar{u}_{\al}\big)(s,x)\ds
\\
\label{F2-1intparts1-in-2}	&-\int_{0}^{t} e^{is|\na|}T_{m_1}\big(P_{\ll1}|\na|^{2+\gamma} u_{\al},n_\al\bar{u}_{\al}\big)(s,x)\ds.
\end{align}
\end{subequations}
Using Lemma \ref{lem:stri}, we have
\begin{subequations}
\begin{align}
\label{est:I2-b}	\|\eqref{N1llN2low}\|_{L^{\infty}_tL^2_x(I^\al)}
\les
&\al \big\|T_{m_1}\big(P_{\ll1}|\na|^{2+\gamma}u_{\al}, \bar{u}_{\al}\big)\big\|_{L^{\infty}_tL^2_x(I^\al)}\\
\label{est:I2-in1}	&+\big\|T_{m_1}\big(P_{\ll1}|\na|^{2+\gamma}(n_\al u_{\al}), \bar{u}_{\al}\big)\big\|_{L^{1}_tL^2_x(I^\al)}
\\
\label{est:I2-in2}
&+\big\|T_{m_1}\big(P_{\ll1}|\na|^{2+\gamma}u_{\al},n_\al\bar{u}_{\al}\big)\big\|_{L^{1}_tL^2_x(I^\al)}.
\end{align}
\end{subequations}
Using Lemma \ref{lem:Coif}, Sobolev's embedding $H^{\frac{3}{2}+\varepsilon}\hookrightarrow L^\infty$ in $x,$ and \eqref{rel:Salf-f}, for the boundary term $\eqref{est:I2-b}$, we have 
\begin{align}
\label{est-I2-b}
\eqref{est:I2-b}
\les
&\al \big\||\na|^{2+\gamma}u_\al\big\|_{L^{\infty}_tL^{2}_x(I^\al)}\big\| u_\al\big\|_{L^{\infty}_{t,x}(I^\al)}
	\notag
\\
\leq &C\al^{-\frac{3}{2}-\gamma}\|u\|_{L^{\infty}_tH^{2+\gamma}_x(I)}\| u\big\|_{L^{\infty}_tH^{\frac{3}{2}+\varepsilon}_x(I)}.
\end{align}
From Lemma \ref{lem:Bernstein}, we see that low frequencies lower the derivative, for $\kappa_1 \in [0,\frac{1}{2}]$, 
\begin{align*}
\eqref{est:I2-in1}+\eqref{est:I2-in2}
\lesssim
&\big\|T_{m_1}\big(|\na|^{2+\gamma-\kappa_1}(n_\al u_{\al}), \bar{u}_{\al}\big)\big\|_{L^{1}_tL^2_x(I^\al)}
\notag
\\
&+\big\|T_{m_1}\big(|\na|^{2+\gamma-\kappa_1}u_{\al},n_\al\bar{u}_{\al}\big)\big\|_{L^{1}_tL^2_x(I^\al)},
\end{align*}
By Lemma \ref{lem:Coif}, Sobolev's embedding $H^{2}\hookrightarrow L^\infty$ in $x,$ and \eqref{rel:Salf-f}, we have
\begin{align}
\label{est-I2-in}
\eqref{est:I2-in1}+\eqref{est:I2-in2}
\les
&|I^\al|\big\||\na|^{2+\gamma-\kappa_1}n_\al\big\|_{L^{\infty}_tL^{2}_x(I^\al)}\big\|u_\al\big\|^2_{L^{\infty}_{t,x}(I^\al)}
\notag\\
&+|I^\al|\big\||\na|^{2+\gamma-\kappa_1}u_\al\big\|_{L^{\infty}_tL^{2}_x(I^\al)}\big\|n_\al\big\|_{L^{\infty}_{t,x}(I^\al)}\big\| u_\al \big\|_{L^{\infty}_{t,x}(I^\al)}\notag\\
\leq & C\al^{-\frac{3}{2}-\gamma+\kappa_1}T\|(u,n)\|_{L^\infty_t(I;H^{2+\gamma-\kappa_1}\times H^{2+\gamma-\kappa_1})}\|(u,n)\|^2_{L^\infty_t(I;H^{2}\times H^{2})}.
\end{align}
Combining \eqref{est-I2-b} and \eqref{est-I2-in}, for $\kappa_1 \in [0,\frac{1}{2}]$ and $\alpha\ges1$, we obtain that
\begin{align}
	\label{est:I2}
\|\eqref{N1llN2low}\|_{L^\infty_tL^{2}_x(I^\al)}
\leq & C\al^{-\frac{3}{2}-\gamma+\kappa_1}\|\vec{u}\|_{L^\infty_t(I;\HH^{3+\gamma-\kappa_1}) }\big[\|\vec{u}\|_{L^\infty_t(I;\HH^2)}
+T\|\vec{u}\|^2_{L^\infty_t(I;\HH^{3}) }\big].
\end{align}

$\bullet$ Estimate for  $\eqref{N1llN21}$.
Using Lemma \ref{lem:Bernstein}, \ref{lem:stri}, and H\"{o}lder's inequality, for $\kappa_2\in[-\frac{1}{2},1]$,
\begin{align}
\label{est:I31}
\|\eqref{N1llN21}\|_{L^{\infty}_tL^{2}_x(I^\al)}
&\les \al \big\||\na|^{\gamma}P_{\sim1}\big(|\na|^3P_{\sim1}u_{\al} \bar{u}_{\al}\big)\big\|_{L^{1}_tL^2_x(I^\al)}
\notag\\
&\les \al |I^\al|\big\||\na|^{3+\gamma}P_{\sim1}u_{\al}\big\|_{L^{\infty}_tL^2_x(I^\al)} \big\|\bar{u}_{\al}\big\|_{L^{\infty}_{t,x}(I^\al)}\notag\\
&\les \al^{-\frac{1}{2}-\gamma-\kappa_2}T\|\vec{u}\|_{L^{\infty}_t(I;\HH^{3+\gamma+\kappa_2})}\|\vec{u}\|_{L^{\infty}_t(I;\HH^{2})}.
\end{align}

$\bullet$ Estimate for  $\eqref{N1llN2high}$.
Similarly as \eqref{N1llN2low}, we have
\begin{align*}
\|\eqref{N1llN2high}\|_{L^\infty_tL^{2}_x(I^{\al})}
&\les \al \left\|\sum_{N_1\gg N_2:N_1\gg1} \int_{0}^{t} e^{is|\na|}|\na|^{\gamma}P_{N_1}P_{\gg1}|\na|^3u_\al P_{N_2}\bar{u}_{\al}\ds\right\|_{L^\infty_tL^{2}_x(I^{\al})}.
\end{align*}
 Taking the Fourier transform, we have 
\begin{equation}
\label{foueier2-4}
\begin{aligned}
	& \F  \Big(\sum_{N_1\gg N_2:N_1\gg1} \int_{0}^{t} e^{is|\na|}|\na|^{\gamma}P_{N_1}P_{\gg1}|\na|^3u_\al P_{N_2}\bar{u}_{\al}\ds\Big)	(\xi) \\
	= &  \sum_{N_1\gg N_2:N_1\gg1} \int_{0}^{t}\int e^{is\phi_1}|\xi|^{\gamma}|\xi_1|^3\psi_{N_1}(\xi_1)\psi_{N_2}(\xi-\xi_1)
\widehat{(P_{\gg1}\rho_\al)}(\xi_1)\hat{\bar{\rho}}_{\al}(\xi-\xi_1)\dxio\ds.
\end{aligned}
\end{equation}
It is worth noting that when $1\ll|\xi_1|\sim|\xi|, $ the phase function $\phi_1=|\xi|+\frac{|\xi_1|^2}{2}-\frac{|\xi-\xi_1|^2}{2}\sim |\xi_1|^2.$
Similarly to \eqref{N1llN2low}, using integration-by-parts together with Lemma \ref{lem:Coif}, we have
\begin{equation}
	\label{F2-4intparts1}
	\begin{aligned}
	&\sum_{N_1\gg N_2:N_1\gg1}\alpha\int_{0}^{t} e^{is|\na|}|\na|^{\gamma}P_{N_1}P_{\gg1}|\na|^3u_\al P_{N_2}\bar{u}_{\al}\ds
	\\
	=&\al \Big[e^{is|\na|} T_{m_2}(|\na|^{1+\gamma}P_{\gg1}u_{\al}, \bar{u}_{\al})(s,x)\Big]\Big|_{s=0}^{s=t}
	 \\
	&-\int_{0}^{t} e^{is|\na|}T_{m_2}(|\na|^{1+\gamma}P_{\gg1}\big(n_\al u_{\al}\big), \bar{u}_{\al})(s,x)\ds
	\\
	&-\int_{0}^{t} e^{is|\na|}T_{m_2}\big(|\na|^{1+\gamma}P_{\gg1} u_{\al},n_\al\bar{u}_{\al}\big)(s,x)\ds,
	\end{aligned}
\end{equation}
where $T_{m_2}(f,g)$ is a bilinear operator whose symbol $m_2$ is given by
\begin{align}
\label{def:m_2}
m_2(\xi_1,\xi-\xi_1)=\sum_{N_1\gg N_2:N_1\gg1}\frac{|\xi|^{\gamma}|\xi_1|^2}{i|\xi_1|^{\gamma}\phi_1}\psi_{N_1}(\xi_1)\psi_{N_2}(\xi-\xi_1).
\end{align}
This symbol satisfies the condition required by Lemma \ref{lem:Coif}.
Then, using high-frequency components can enhance the derivative in Lemma \ref{lem:Bernstein}, for $\kappa_1\in[0,\frac{1}{2}]$, we have 
\begin{subequations}
	\begin{align}
\label{F2-4intparts1b}
\|\eqref{F2-4intparts1}\|_{L^\infty_tL^2_x(I^\al)}
\les &\al \big\| T_{m_2}\big(|\na|^{2+\gamma}u_{\al}, \bar{u}_{\al}\big)(t,x)\big\|_{L^\infty_tL^2_x(I^\al)}
	\\
\label{F2-4intparts1-in-1}
&+\big\|T_{m_2}\big(|\na|^{2+\gamma-\kappa_1}(n_\al u_{\al}), \bar{u}_{\al}\big)(t,x)\big\|_{L^1_tL^2_x(I^\al)}
		\\
\label{F2-4intparts1-in-2}
&+\big\|T_{m_2}\big(|\na|^{2+\gamma-\kappa_1}u_{\al},n_\al\bar{u}_{\al}\big)(t,x)\big\|_{L^1_tL^2_x(I^\al)},
\end{align}
\end{subequations}
where \eqref{F2-4intparts1b} enjoys the same estimates as \eqref{est-I2-b}, and \eqref{F2-4intparts1-in-1}--\eqref{F2-4intparts1-in-2} enjoy the same estimates as \eqref{est-I2-in}. Combining \eqref{est:I1}, \eqref{est:I2}, and \eqref{est:I31}, for $\alpha\ges1$, $\kappa_1\in[0,\frac{1}{2}]$ and $\kappa_2\in[-\frac{1}{2},1]$, 
\begin{align}
\label{est:Psi2}
&
\al\big\| \int_{0}^{t}e^{-i(t-s)|\na|}|\nabla|^{-1}\re \nabla\cdot  (\nabla\De u_{\al} \bar{u}_\alpha)\ds\big\|_{\mW^{\gamma}(I^\al)}
\notag\\
\leq & C \al^{-\frac{3}{2}-\gamma}\|u\|_{L^\infty_tH^{\frac{5}{2}+\gamma}_x(I)}\big[\|u\|_{L^\infty_tH^{3}_x(I)}+T\|\vec{u}\|^2_{L^\infty_t(I;\HH^{3})}\big]
\notag\\
&+C\al^{-\frac{3}{2}-\gamma+\kappa_1}\|\vec{u}\|_{L^\infty_t(I;\HH^{3+\gamma-\kappa_1}) }\big[\|\vec{u}\|_{L^\infty_t(I;\HH^2)}
+T\|\vec{u}\|^2_{L^\infty_t(I;\HH^{3}) }\big]
\notag\\
&+C\al^{-\frac{1}{2}-\gamma-\kappa_2}T\|\vec{u}\|_{L^\infty_t(I;\HH^{3+\gamma+\kappa_2})}\|\vec{u}\|_{L^\infty_t(I;\HH^{2})}
\notag\\
\leq & C(\|\vec{u}_0\|_{\HH^3})\big[\al^{-\frac{3}{2}-\gamma+\kappa_1}\|\vec{u}\|_{L^\infty_t(I;\HH^{3+\gamma-\kappa_1})}+\al^{-\frac{1}{2}-\gamma-\kappa_2}\|\vec{u}\|_{L^\infty_t(I;\HH^{3+\gamma+\kappa_2})}\big],
\end{align} 
 here we used  Lemma \ref{lem:local-theory} in the last step.		
						
Clearly, the estimate for \eqref{2-2} is simpler and follows similarly to \eqref{2-1}. For brevity, we only sketch the key steps here. Using the Littlewood-Paley dyadic projection operator, we have
\begin{subequations}
\begin{align}
&\quad\|\eqref{2-2}\|_{\mW^{\gamma}(I^\al)}\notag\\
&\les \al\Big\| \sum_{N_1\les N_2}\int_{0}^{t}e^{-i(t-s)|\na|}|\na|^{\gamma}\big(P_{N_1}\De u_\alpha P_{N_2} \nabla \bar{u}_{\al}\big) \ds\Big\|_{L^\infty_tL^2_x(I^\al)}
\label{2N1gesN2}\\
&\quad+\al\Big\| \sum_{N_1\gg N_2:N_1\ll1}\int_{0}^{t}e^{-i(t-s)|\na|}|\na|^{\gamma}\big(P_{N_1}P_{\ll1}\De u_\alpha P_{N_2} \nabla \bar{u}_{\al}\big) \ds\Big\|_{L^\infty_tL^2_x(I^\al)}
\label{2N1llN2low}
\\
\label{2N1llN21}
	&\quad+\al \Big\| \sum_{N_1\gg N_2:N_1\ges 1}\int_{0}^{t}e^{-i(t-s)|\na|}|\na|^{\gamma}\big(P_{N_1}P_{\ges 1}\De u_\alpha P_{N_2} \nabla \bar{u}_{\al}\big) \ds\Big\|_{L^\infty_tL^2_x(I^\al)}.
\end{align}
\end{subequations}
The estimate for $\eqref{2N1gesN2}$ follows analogously to that for $\eqref{N1gesN2}$, yielding a result identical to \eqref{est:I1}, we therefore omit the details. Proceeding similarly as \eqref{N1llN2low} to the estimate for $\eqref{2N1llN2low}$, we obtain the following estimates, which are subsequently absorbed into \eqref{est:Psi2}.
\begin{align*}
\eqref{2N1llN2low}
\lesssim&
\al\big\|T_{m_1}\big(P_{\ll1}|\na|^{1+\gamma}u_{\al},\na \bar{u}_{\al} \big)\big\|_{L^\infty_tL^2_x(I^\al)}
\\
&+\big\|T_{m_1}\big(P_{\ll1}|\na|^{1+\gamma}(n_\al u_{\al}),\na \bar{u}_{\al}\big)\big\|_{L^1_tL^2_x(I^\al)}
\\
&+\big\|T_{m_1}\big(P_{\ll1}|\na|^{1+\gamma}u_{\al}, \na(n_\al \bar{u}_{\al})\big)\big\|_{L^1_tL^2_x(I^\al)}
\\
\les& \al^{-\frac{3}{2}-\gamma}\|\vec{u}\|_{L^\infty_t(I;\HH^{\frac{5}{2}+\gamma} )}\big(\|\vec{u}\|_{L^\infty_t(I;\HH^2)}+T\|\vec{u}\|^2_{L^\infty_t(I;\HH^2)}\big),
\end{align*}
where $m_1$ is given in \eqref{def:m_1}.
By Lemma \ref{lem:Bernstein}, \ref{lem:Little}, and \eqref{rel:Salf-f}, for $\kappa_1\in [0,\frac{1}{2}]$,
\begin{align*}
\eqref{2N1llN21}
\lesssim&
\al\Big\| \big\| \sum_{N_1:N_1\ges 1}|\na|^{\gamma} P_{\sim N_1}\big(P_{N_1}\De u_\alpha  P_{\ll N_1}\nabla \bar{u}_{\al}\big)\big\|_{L^2_x} \Big\|_{L^1_t(I^\al)}
\\
\lesssim&
\al\Big\| \sum_{N_1:N_1\ges 1}\big\|N_1^{\gamma}P_{N_1}\De u_\alpha  \big\|_{L^2_x}\big\| \na \bar{u}_{\al}\big\|_{L^\infty_{x}} \Big\|_{L^1_t(I^\al)}
\\
\lesssim&
\al\Big\| \sum_{N_1:N_1\ges 1}N_1^{-1+\kappa_1}\big\|N_1^{1+\gamma-\kappa_1}P_{N_1}\De u_\alpha  \big\|_{L^2_x}\big\| \na \bar{u}_{\al}\big\|_{L^\infty_{x}} \Big\|_{L^1_t(I^\al)}
\\
\lesssim&
\al \Big\| \big\|N_1^{1+\gamma-\kappa_1}P_{N_1}\De u_\alpha  \big\|_{l^2_{N_1}L^2_x}\big\| \na \bar{u}_{\al}\big\|_{L^\infty_{x}} \Big\|_{L^1_t(I^\al)}
\\
\lesssim&
\al^{-\frac{3}{2}-\gamma+\kappa_1}T\big\|\vec{u}\big\|_{L^\infty_t(I;\HH^{3+\gamma-\kappa_1})}\big\|\vec{u}\big\|_{L^\infty_t(I;\HH^{3})},
\end{align*}
which can also be contained in \eqref{est:Psi2}.

\subsubsection{ Cubic term estimates.} 

We now turn to the treatment of \eqref{2-3}, which can be divided into the following three parts:  
\begin{subequations}
\begin{align}
\|\eqref{2-3}\|_{\mW^{\gamma}(I^\al)}\les 
\label{4-1}
&\Big\|\int_{0}^t e^{-i(t-s)|\na|}|\na|^{1+\gamma}
\big(P_{\ges1}n_\al \,|u_\al|^2\big)\ds\Big\|_{L^\infty_tL^2_x(I^\al)}
\\
\label{4-2}
&+\Big\|\int_{0}^t e^{-i(t-s)|\na|}
 |\na|^{1+\gamma}
\big[n_\al \big(P_{\ges1}u_\al\,\bar{u}_\al+u_\al\, P_{\ges1}\bar{u}_\al\big) \big]\ds\Big\|_{L^\infty_tL^2_x(I^\al)}
\\
\label{4-3}
&+\Big\|\int_{0}^t e^{-i(t-s)|\na|}|\na|^{1+\gamma}\big(P_{\ll1}n_{\al}\,|P_{\ll1}u_\al|^2\big)\ds\|_{L^\infty_tL^2_x(I^\al)}.
\end{align}
\end{subequations}
Due to the presence of high-frequency parts in the first two terms, using  Lemma \ref{lem:Bernstein}, Lemma \ref{lem:kato-Ponce}, \eqref{rel:Salf-f}, Sobolev's embedding $\dot{H}^{\kappa_1}\hookrightarrow L^{\frac{6}{3-2\kappa_1}}$, $H^{\frac{3}{2}-\kappa_1+\varepsilon}\hookrightarrow L^{\frac{3}{\kappa_1}}$, $\dot{H}^{1-\kappa_1}\hookrightarrow L^{\frac{6}{1+2\kappa_1}}$, and $\dot{H}^{\frac{1}{2}+\kappa_1}\hookrightarrow L^{\frac{3}{1-\kappa_1}}$ in $x$ for $\kappa_1\in [0,\frac{1}{2}]$, we have
\begin{align}
\label{est:4-1+4-2}
\eqref{4-1}+\eqref{4-2}
&\les |I^\al|\norm{|\na|^{2+\gamma-\kappa_1}P_{\ges1}n_\al}_{L^\infty_tL^2_x(I^\al)}\norm{u_\al}^2_{L^\infty_{t,x}(I^\al)}
\notag\\
&\quad+|I^\al|\norm{|\na|^{1-\kappa_1}P_{\ges1}n_\al}_{L^\infty_tL^\frac{6}{3-2\kappa_1}_x(I^\al)}\norm{|\na|^{1+\gamma}u_\al}_{L^\infty_tL^\frac{3}{\kappa_1}_x(I^\al)}\norm{u_\al}_{L^\infty_{t,x}(I^\al)}
\notag\\
&\quad+ |I^\al|\norm{|\na|^{2+\gamma-\kappa_1}P_{\ges1}u_\al}_{L^\infty_tL^6_x(I^\al)}\norm{n_\al}_{L^\infty_tL^6_x(I^\al)}\norm{u_\al}_{L^\infty_tL^6_x(I^\al)}
\notag\\
&\quad+|I^\al|\norm{|\na|^{1+\gamma}n_\al}_{L^\infty_tL^\frac{6}{1+2\kappa_1}_x(I^\al)}\norm{|\na|^{1-\kappa_1}P_{\ges1}u_\al}_{L^\infty_tL^{\frac{3}{1-\kappa_1}}_x(I^\al)}\norm{u_\al}_{L^\infty_{t,x}(I^\al)}
\notag\\
&\les \al^{-\frac{3}{2}-\gamma+\kappa_1}T\|(u,n)\|_{L^\infty_t(I;H^{3+\gamma-\kappa_1}\times H^{2+\gamma-\kappa_1})}\|(u,n)\|^2_{L^\infty_t(I;H^{2}\times H^{1})}
\notag\\
&\les \al^{-\frac{3}{2}-\gamma+\kappa_1}T\|\vec{u}\|_{L^\infty_t(I;\HH^{3+\gamma-\kappa_1})}\|\vec{u}\|^2_{L^\infty_t(I;\HH^{2})},
\end{align}
which can be absorbed into \eqref{est:Psi2} by Lemma \ref{lem:local-theory}. Using the expression of 
$n_\al=-\al|u_\al|^2+w_\al+\Psi_\al$,
 we can further decompose 
\eqref{4-3} into the following three terms:
\begin{subequations}
\begin{align}
\eqref{4-3}\les 
\label{5-1}
&\al \Big\|\int_{0}^t e^{-i(t-s)|\na|}|\na|^{1+\gamma}
\big(P_{\ll1}|u_\al|^2\,|P_{\ll1}u_\al|^2\big)\ds\Big\|_{L^\infty_tL^2_x(I^\al)}
\\
\label{5-2}
&+\Big\|\int_{0}^t e^{-i(t-s)|\na|}
 |\na|^{1+\gamma}
\big(P_{\ll1}w_\al\, |P_{\ll1}u_\al|^2\big) \ds\Big\|_{L^\infty_tL^2_x(I^\al)}
\\
\label{5-3}
&+\Big\|\int_{0}^t e^{-i(t-s)|\na|}|\na|^{1+\gamma}\big(P_{\ll1}\Psi_{\al}\,|P_{\ll1}u_\al|^2\big)\ds\|_{L^\infty_tL^2_x(I^\al)}.
\end{align}
\end{subequations}
By Lemma \ref{lem:part},  \eqref{rel:Salf-f}, and $\partial_t |u_\al|^2=-\re(\Delta u_\al i\bar{u}_\alpha)$ in \eqref{ptu2}, we have
\begin{align}
\label{est:5-1}
\eqref{5-1}
\lesssim
&\al\norm{|\na|^{\gamma}|u_\al|^4}_{L^\infty_tL^2_x(I^\al)}+\al\norm{|\na|^{\gamma}\big(\partial_t |u_\al|^2 |u_\al|^2\big)}_{L^1_t L^2_x(I^\al)}
\notag\\
\les&\al\norm{|\na|^{\gamma}|u_\al|^4}_{L^\infty_tL^2_x(I^\al)}
+\al\norm{|\na|^{\gamma}\big(\De u_\al |u_\al|^3\big)}_{L^1_t L^2_x(I^\al)}
\notag\\
\les&
\al^{-\frac{3}{2}-\gamma}(1+T)\|u\|_{L^\infty_tH^{2+\gamma}_x(I)}\|u\|^3_{L^\infty_tH^{2}_x(I)},
\end{align}
We note that \eqref{est:5-1} can be absorbed into \eqref{est:Psi2} by Lemma \ref{lem:local-theory}. By Lemma \ref{lem:Bernstein}, we have
\begin{align}
\label{est:5-3}
\eqref{5-3}
&\les 
|I^\al|\norm{|\na|^{1+\gamma}P_{\ll1}\Psi_{\al}}_{L^\infty_tL^2_x(I^\alpha)}
\norm{u_{\al}}^2_{L^\infty_{t,x}(I^\alpha)}
\notag\\
&\quad
+|I^\al|\norm{P_{\ll1}\Psi_{\al}}_{L^\infty_tL^3_x(I^\alpha)}
\norm{|\na|^{1+\gamma}u_{\al}}_{L^\infty_tL^6_x(I^\alpha)}\norm{u_{\al}}_{L^\infty_{t,x}(I^\alpha)}
\notag\\
&\les T\|\Psi_\al\|_{\mW^{
\gamma}(I^\alpha)}\|u\|^2_{L^\infty_tH^{2}_x(I)} +\al^{-\frac{1}{2}-\gamma}T\|\Psi_\al\|_{\mW^{
0}(I^\alpha)}\|u\|_{L^\infty_tH^{2+\gamma}_x(I)}\|u\|_{L^\infty_tH^{2}_x(I)}.
\end{align}
Now, only the estimate for \eqref{5-2} remains to be established.
Using the Littlewood-Paley square function estimate, it suffices to
estimate the output dyadic pieces in
\(\ell_N^2L_t^\infty L_x^2\). By symmetry, we may assume $N_3\leq N_2$, thus we have
\begin{subequations}
\begin{align}
&\quad\eqref{5-2}\notag\\
&\lesssim
\Big\|
\sum_{\substack{N_1:1\gg N_1\gtrsim N}}
\sum_{\substack{N_3\leq N_2\\ N_2\ll1}}
\frac{N^{1+\gamma}}{N_1^{1+\gamma}}
\big\|
\int_0^t e^{is|\nabla|}P_N
\big(
N_1^{1+\gamma}P_{N_1}w_\alpha\,
P_{N_2}u_\alpha\,
P_{N_3}\bar u_\alpha
\big)\,ds
\big\|_{L_t^\infty L_x^2(I^\alpha)}
\Big\|_{\ell_N^2}
\label{5-3-1}
\\
&\quad+
\Big\|
\sum_{\substack{N_1:N_1\ll N\ll1}}
\sum_{\substack{N_3\leq N_2\\ N\les N_2\ll1}}
\frac{N^{1+\gamma}}{N_2^{1+\gamma}}
\big\|
\int_0^t e^{is|\nabla|}P_N
\big(
P_{N_1}w_\alpha\,
N_2^{1+\gamma}P_{N_2}u_\alpha\,
P_{N_3}\bar u_\alpha
\big)\,ds
\big\|_{L_t^\infty L_x^2(I^\alpha)}
\Big\|_{\ell_N^2}.
\label{5-3-2}
\end{align}
\end{subequations}
The argument for \eqref{5-3-2} is similar to that for \eqref{5-3-1};
we therefore omit it and present details only for \eqref{5-3-1}.
Since the two half-wave components \(w_{\alpha,+}\) and \(w_{\alpha,-}\)
are treated in the same way, we only consider \(w_{\alpha,+}\).
By the embedding
\(V^2(\mathbb R;L^2)\hookrightarrow L_t^\infty L_x^2(\mathbb R)\),
Lemma \ref{Lem:linearforup}, and \(\ell^2\)-duality in the output
frequency \(N\), we have
\begin{align*}
&\quad\eqref{5-3-1}\\
&\lesssim
\Big\|
\sum_{\substack{N_1:1\gg N_1\gtrsim N}}
\sum_{\substack{N_3\leq N_2\\ N_2\ll1}}
\frac{N^{1+\gamma}}{N_1^{1+\gamma}}
\big\|
\int_0^t e^{is|\nabla|}P_N
\big(
N_1^{1+\gamma}P_{N_1}w_{\alpha,+}\,
P_{N_2}u_\alpha\,
P_{N_3}\bar u_\alpha
\big)\,ds
\big\|_{L_t^\infty L_x^2(I^\alpha)}
\Big\|_{\ell_N^2}
\\
&\lesssim
\Big\|
\sum_{\substack{N_1:1\gg N_1\gtrsim N}}
\sum_{\substack{N_3\leq N_2\\ N_2\ll1}}
\frac{N^{1+\gamma}}{N_1^{1+\gamma}}
\big\|
\int_0^t e^{-i(t-s)|\nabla|}P_N
\big(
N_1^{1+\gamma}P_{N_1}w_{\alpha,+}\,
P_{N_2}u_\alpha\,
P_{N_3}\bar u_\alpha
\big)\,ds
\big\|_{V^2_{+|\nabla|}(I^\alpha;L_x^2)}
\Big\|_{\ell_N^2}
\\
&\lesssim
\sup_{\|P_N g\|_{\ell_N^2 U_{+|\nabla|}^2(I^\alpha;L_x^2)}=1}
\sum_{\substack{N_1\gtrsim N \\ N_1\ll1}}
\sum_{\substack{N_3\leq N_2\\ N_2\ll1}}
\frac{N^{1+\gamma}}{N_1^{1+\gamma}}
\Big|\int_{I^\alpha}\int N_1^{1+\gamma}P_{N_1}w_{\alpha,+}\,
P_{N_2}u_\alpha\,
P_{N_3}\bar u_\alpha\,
P_N\bar{g}
\,dx\,dt
\Big|,
\end{align*}
where $\|P_N g\|_{\ell_N^2 U_{+|\nabla|}^2(I^\alpha;L_x^2)}:=
\Big(
\sum_{N\ll1}
\|P_N g\|_{U_{+|\nabla|}^2(I^\alpha;L_x^2)}^2
\Big)^{1/2}$. By H\"older's inequality, we have 
\begin{align*}
	&\quad \left|\int_{I^\al}\int N_1^{1+\gamma}P_{N_1}w_{\al,+}\,
			P_{N_2}u_\al\, P_{N_3}\bar{u}_\al \,P_N\bar{g}\dx\dt\right|\notag\\
	&\lesssim
	\big\|N_1^{1+\gamma}P_{N_1}w_{\al,+}\,
	P_{N_2}u_\al\big\|_{L^2_{t,x}(I^\al)} 	\big\|P_{N_3}u_\al \,P_{N}g\big\|_{L^2_{t,x}(I^\al)}.
\end{align*}
From Corollary \ref{lem:bi-linear-W}, for $N_1,N_2,N_3\ll1$, we have
\begin{align*}
	\|N_1^{1+\gamma}P_{N_1}w_{\al,+}\,P_{N_2}u_\al\|_{L^2_{t,x}(I^\al)} &\les \min\{N_1,N_2\}
	\|N_1^{1+\gamma}P_{N_1}w_{\al,+}\|_{U_{+|\na|}^2(I^\al;L^2_x)}\|P_{N_2}u_\al\|_{U_{\De}^2(I^\al;L^2_x)}
	\notag\\
	&\les 
		\|N_1^{1+\gamma}P_{N_1}w_{\al,+}\|_{U_{+|\na|}^2(I^\al;L^2_x)}\|N_2P_{N_2}u_\al\|_{U_{\De}^2(I^\al;L^2_x)},
		\\
	\|P_{N_3}u_\al\, P_{N}g\|_{L^2_{t,x}(I^\al)} &\les \|N_3P_{N_3}u_\al\|_{U_{\De}^2(I^\al;L^2_x)}\|P_{N}g\|_{U_{+|\na|}^2(I^\al;L^2_x)}.
\end{align*}
Combining the above estimates, 
\begin{align*}
\eqref{5-3-1}&\les 
\sum_{N_3\leq N_2:N_2\ll1 }\|N_2P_{N_2}u_\al\|_{U_{\De}^2(I^\al;L^2_x)}\|N_3P_{N_3}u_\al\|_{U_{\De}^2(I^\al;L^2_x)}
\\
&\quad\quad\cdot \sup_{\|P_Ng\|_{\ell^2_N U_{+|\na|}^2(I^\al;L^2_x)}=1}\sum_{\substack{N_1\gtrsim N \\ N_1\ll1}}\frac{N^{1+\gamma}}{N_1^{1+\gamma}}
\|P_{N_1}|\na|^{1+\gamma}w_{\al,+}\|_{U_{+|\na|}^2(I^\al;L^2_x)}\|P_{N}g\|_{U_{+|\na|}^2(I^\al;L^2_x)}.
\end{align*}
Using Lemma \ref{lem:Schur} and \eqref{lem:est:U2w_al}, we have
\begin{align*}
	&\quad\sup_{\|P_Ng\|_{\ell^2_N U_{+|\na|}^2(I^\al;L^2_x)}=1}\sum_{\substack{N_1\gtrsim N \\ N_1\ll1}}\frac{N^{1+\gamma}}{N_1^{1+\gamma}}
	\|P_{N_1}|\na|^{1+\gamma}w_{\al,+}\|_{U_{+|\na|}^2(I^\al;L^2_x)}\|P_{N}g\|_{U_{+|\na|}^2(I^\al;L^2_x)}
	\notag\\
	&\les \sup_{\|P_Ng\|_{\ell^2_N U_{+|\na|}^2(I^\al;L^2_x)}=1}
	\||\na|^{1+\gamma}P_{N}w_{\al,+}\|_{\ell^2_NU_{+|\na|}^2(I^\al;L^2_x)}\|P_{N}g\|_{\ell^2_NU_{+|\na|}^2(I^\al;L^2_x)}\\
	&\les\al^{-\gamma-\frac{1}{2}}\|(w_0, |\al\na|^{-1}w_1)\|_{H^{1+\gamma}\times H^{1+\gamma}}. 
\end{align*}
Using Lemma \ref{lem:Schur}, Bernstein's inequality, and Lemma  \ref{lem:estU2-u_al}, we have
\begin{align*}
	&\quad \sum_{N_3\leq N_2:N_2\ll1}N_3^{\varepsilon}N_2^{-\varepsilon}\||\na|^{1+\varepsilon}P_{N_2}u_\al\|_{U_{\De}^2(I^\al;L^2_x)}\||\na|^{1-\varepsilon}P_{N_3}u_\al\|_{U_{\De}^2(I^\al;L^2_x)}
\notag\\
	&\les  \||\na|^{1+\varepsilon}P_{N_2}u_\al\|_{\ell^2_{N_2}U_{\De}^2(I^\al;L^2_x)}
	\||\na|^{1-\varepsilon}P_{N_3}u_\al\|_{\ell^2_{N_3}U_{\De}^2(I^\al;L^2_x)}\\
	&\les \al^{-1}\|(u,n)\|_{L^\infty_t(I;H^{1+\varepsilon}_x\times H^{1+\varepsilon})}\big[1+T\|(u,n)\|_{L^\infty_t(I;H^{2}_x\times H^{2}_x)}\big]\notag\\
	&\qquad \cdot\|(u,n)\|_{L^\infty_t(I;H^{1-\varepsilon}_x\times H^{1-\varepsilon}_x)}\big[1+T\|(u,n)\|_{L^\infty_t(I;H^{2}_x\times H^{2}_x)}\big]
	\notag\\
  &\les \al^{-1}\|\vec{u}\|^2_{L^\infty_t(I;\HH^{2+\varepsilon})}\big[1+T^2\|\vec{u}\|^2_{L^\infty_t(I;\HH^{3})}\big].
\end{align*}
Then we further obtain
\begin{align}
	\label{est:5-3-1}
		\eqref{5-3-1}&\les \al^{-\frac{3}{2}-\gamma}\|(w_0, |\al\na|^{-1}w_1)\|_{H^{1+\gamma}\times H^{1+\gamma}}\|\vec{u}\|^2_{L^\infty_t(I;\HH^{2+\varepsilon})}\big[1+T^2\|\vec{u}\|^2_{L^\infty_t(I;\HH^{3})}\big],
		\\
		\label{est:5-3-2}
	\eqref{5-3-2}&\les \al^{-\frac{3}{2}-\gamma}\|(w_0, |\al\na|^{-1}w_1)\|_{H^{1}\times H^{1}}\|\vec{u}\|_{L^\infty_t(I;\HH^{2+\gamma+\varepsilon})}\|\vec{u}\|_{L^\infty_t(I;\HH^{2})}\big[1+T^2\|\vec{u}\|^2_{L^\infty_t(I;\HH^{3})}\big].
\end{align}
  Combining \eqref{est:4-1+4-2}, \eqref{est:5-1}, \eqref{est:5-3}, \eqref{est:5-3-1}, \eqref{est:5-3-2}, and Lemma \ref{lem:local-theory}, for $\alpha\ges1$ and $\kappa_1\in[0,\frac{1}{2}]$, we obtain the estimate of cubic term \eqref{2-3} as follows
  \begin{equation}
  \label{est:Psi3}
\begin{aligned}
    &\quad\big\|\int_{0}^{t}e^{-i(t-s)|\na|}|\nabla|^{-1}\re \nabla\cdot  \big[\nabla (n_\al u_{\al})\, \bar{u}_\alpha-n_\al\bar{u}_\alpha \nabla  u_{\al}\big]\ds\big\|_{\mW^{\gamma}(I^\al)}
\\
     \leq& 
   \al^{-\frac{3}{2}-\gamma}(1+T)\|\vec{u}\|_{L^\infty_t(I;\HH^{2+\gamma})}\|u\|^3_{L^\infty_t(I;\HH^{2})}
+T\al^{-\frac{3}{2}-\gamma+\kappa_1}\|\vec{u}\|_{L^\infty_t(I;\HH^{3+\gamma-\kappa_1})}\|\vec{u}\|^2_{L^\infty_t(I;\HH^{2})}
\\
    &+T\|\Psi_\al\|_{\mW^{
    \gamma}(I^\alpha)}\|u\|^2_{L^\infty_tH^{2}_x(I)} +\al^{-\frac{1}{2}-\gamma}T\|\Psi_\al\|_{\mW^{
    0}(I^\alpha)}\|u\|_{L^\infty_tH^{2+\gamma}_x(I)}\|u\|_{L^\infty_tH^{2}_x(I)}
\\
    &+
    \al^{-\frac{3}{2}-\gamma}\|(w_0, |\al\na|^{-1}w_1)\|_{H^{1+\gamma}\times H^{1+\gamma}}\big[1+T^2\|\vec{u}\|^2_{L^\infty_t(I;\HH^{3})}\big]
 \\
 &\qquad\qquad \cdot \big[\|\vec{u}\|^2_{L^\infty_t(I;\HH^{2+\varepsilon})}+\|\vec{u}\|_{L^\infty_t(I;\HH^{2+\gamma+\varepsilon})}\|\vec{u}\|_{L^\infty_t(I;\HH^{2})}\big]
\\
         \leq& C(\|\vec u_0\|_{\HH^3})\al^{-\frac{3}{2}-\gamma+\kappa_1}\|\vec{u}\|_{L^\infty_t(I;\HH^{3+\gamma-\kappa_1})}\big[1+\|(w_0, |\al\na|^{-1}w_1)\|_{H^{1+\gamma}\times H^{1+\gamma}}\big]
         \\
         &+ C(\|\vec u_0\|_{\HH^3})T\big[\|\Psi_\al\|_{\mW^{\gamma}(I^\alpha)}+\al^{-\frac{1}{2}-\gamma}\|\Psi_\al\|_{\mW^{
             0}(I^\alpha)}\big].
    \end{aligned}
  \end{equation} 
 Combining \eqref{est:Psi2}, \eqref{est:Psi3}, and Lemma \ref{lem:local-theory}, for $\al\ges1$,$\gamma\in[0,\frac{1}{2}]$,  $\kappa_1\in[0,\frac{1}{2}]$, and $\kappa_2\in[-\frac{1}{2},1]$,
 \begin{equation}
\label{est:Psigamma}
\begin{aligned}
 \big\|\Gamma_\al\big\|_{\mW^{\gamma}(I^\al)}
	&\leq  
 	C\al^{-\frac{3}{2}-\gamma+\kappa_1}\|\vec{u}\|_{L^\infty_t(I;\HH^{3+\gamma-\kappa_1})}\big[1+\|(w_0, |\al\na|^{-1}w_1)\|_{H^{1+\gamma}\times H^{1+\gamma}}\big]\\
 	&\quad+C\al^{-\frac{1}{2}-\gamma-\kappa_2}\|\vec{u}\|_{L^\infty_t(I;\HH^{3+\gamma+\kappa_2})}
 	\\
     &\quad+C\al^{-\frac{1}{2}-\gamma}T\|\Psi_\al\|_{\mW^{
     0}(I^\alpha)}+CT\|\Psi_\al\|_{\mW^{\gamma}(I^\alpha)},
 \end{aligned}
 \end{equation}where $C$ is a constant only depending on $\|\vec u_0\|_{\HH^3}$.
 
 With the above estimates, we can now provide the proof of Lemma \ref{lem:nonlinear-estimates-psi}.
 \begin{proof}[Proof of Lemma \ref{lem:nonlinear-estimates-psi}]
 First, we choose $T$ suitably such that $C(\|\vec u_0\|_{\HH^3})T \leq \frac{1}{4}$, by \eqref{est:Psigamma} and  $\Psi_\al=-\im \Gamma_\al$ in \eqref{exp:ptPsi}, we obtain 
   \begin{equation}
  \label{est:Psigamma1}
  \begin{aligned}
   \big\|\Gamma_\al\big\|_{\mW^{\gamma}(I^\al)}
   	&\leq 
C(\|\vec u_0\|_{\HH^3})\al^{-\frac{3}{2}-\gamma+\kappa_1}\|\vec{u}\|_{L^\infty_t(I;\HH^{3+\gamma-\kappa_1}) }\big[1+\|(w_0, |\al\na|^{-1}w_1)\|_{H^{1+\gamma}\times H^{1+\gamma}}\big]\\
   	&\quad+C(\|\vec u_0\|_{\HH^3})\al^{-\frac{1}{2}-\gamma-\kappa_2}\|\vec{u}\|_{L^\infty_t(I;\HH^{3+\gamma+\kappa_2})}
 +C(\|\vec u_0\|_{\HH^3})T\al^{-\frac{1}{2}-\gamma}\|\Gamma_\al\|_{\mW^{
       0}(I^\alpha)}.
   \end{aligned}
   \end{equation}
  Let $\gamma=0$ in \eqref{est:Psigamma1}. Then we have
 \begin{equation}
 \label{est-Psi0}
 \begin{aligned}
\big\|\Gamma_\al\big\|_{\mW^{0}(I^\al)}
	&\leq  
C(\|\vec u_0\|_{\HH^3})\al^{-\frac{3}{2}+\kappa_1}\|\vec{u}\|_{L^\infty_t(I;\HH^{3-\kappa_1}) }\big[1+\|(w_0, |\al\na|^{-1}w_1)\|_{H^{1}\times H^{1}}\big]\\
  &\quad+C(\|\vec u_0\|_{\HH^3})\al^{-\frac{1}{2}-\kappa_2}\|\vec{u}\|_{L^\infty_t(I;\HH^{3+\kappa_2})}
+C(\|\vec u_0\|_{\HH^3})\al^{-\frac{1}{2}}T\|\Gamma_\al\|_{\mW^{0}(I^\alpha)}.
  \end{aligned}
  \end{equation}
  Choosing suitable $T$ such that for any $\al\ges1 $, $C(\|\vec u_0\|_{\HH^3})\al^{-\frac{1}{2}}T\leq C(\|\vec u_0\|_{\HH^3})T \leq\frac{1}{2},$ then we have
\begin{equation}
 \label{est:Psi0}
 \begin{aligned}
  \big\|\Gamma_\al\big\|_{\mW^{0}(I^\al)}
  	&\leq  
  	2C(\|\vec u_0\|_{\HH^3})\al^{-\frac{3}{2}+\kappa_1}\|\vec{u}\|_{L^\infty_t(I;\HH^{3-\kappa_1}) }\big[1+\|(w_0, |\al\na|^{-1}w_1)\|_{H^{1}\times H^{1}}\big]\\
  &\quad+2C(\|\vec u_0\|_{\HH^3})
\al^{-\frac{1}{2}-\kappa_2}\|\vec{u}\|_{L^\infty_t(I;\HH^{3+\kappa_2})},
  \end{aligned}
  \end{equation}
  then inserting \eqref{est:Psi0} into \eqref{est:Psigamma1},
  \begin{align*}
 \big\|\Gamma_\al\big\|_{\mW^{\gamma}(I^\al)}
    	&\leq 
    	C(\|\vec u_0\|_{\HH^3})\al^{-\gamma-\frac{3}{2}+\kappa_1}\big\|\vec{u}\|_{L^\infty_t(I;\HH^{3+\gamma-\kappa_1}) }\big[1+\|(w_0, |\al\na|^{-1}w_1)\|_{H^{1+\gamma}\times H^{1+\gamma}}\big]
    	\\
    	&\quad
    	+C(\|\vec u_0\|_{\HH^3})\al^{-\gamma-\frac{1}{2}-\kappa_2}\big[\|\vec{u}\|_{L^\infty_t(I;\HH^{3+\gamma+\kappa_2})}+\|\vec{u}\|_{L^\infty_t(I;\HH^{3+\kappa_2})}\big].
  \end{align*}
 Finally, by Lemma \ref{lem:local-theory}, we complete the proof of Lemma \ref{lem:nonlinear-estimates-psi}.
 \end{proof}

\subsection{Nonlinear estimates II: the Schr\"odinger remainder}\label{sec:estr}
			In the present subsection, we focus on deriving an estimate for 
			$\Phi_\al.$ The main result is as follows:
			
			\begin{lem}[Nonlinear estimate for $\Phi_\al$]
				\label{lem:nonlinear-estimates-phi}
				Let the assumptions in Proposition  \ref{prop:estPsiPhi} hold.  Then 
\begin{align*}
\left\|\Phi_{\al}\right\|_{\mS^{\gamma}(I^\al)} 
&\leq 
C_2 \|(w_0,|\al\na|^{-1}w_1)\|_{H^{1+\gamma}\times H^{1+\gamma}}\big[\al^{-\gamma-\frac{1}{2}}+\al^{-\gamma-1+\frac{\varepsilon}{2}}\left\|v \right\|_{X^{2+\gamma}(I)}\big]
					\notag\\
&\quad +C_2\big\|\big(\Phi_{\al},\Psi_\al\big)\big\|_{\mS^{\gamma}\times \mW^{\gamma}(I^{\al})}\Big(\|\Phi_{\al}\|^2_{\mS^{\frac{1}{2}}(I^\al)}
+T^{\frac{1}{2}}\big\|\big(\Phi_{\al},\Psi_\al\big)\big\|_{\mS^{\frac{1}{2}}\times \mW^{\frac{1}{2}}(I^{\al})}+T^{\frac{1+2\gamma}{4}}\Big),
\end{align*}
where  $C_2$ is the same as that in Proposition \ref{prop:estPsiPhi}.
			\end{lem}
			
	\begin{proof}
		Now, we prove the above Lemma. Recall the equation for $\Phi_\al$ in \eqref{equ-Phial} and $\Phi_\alpha(0)=S_\alpha r_0=0$ under Assumption \ref{assu:main}, then 
			\begin{subequations}
				\begin{align}
					\label{I1-1}
					\Phi_\al(t,x)= &-\frac{i}{2}\int_{0}^{t}	e^{-\frac{i}{2}(t-s)\De} \big(|u_{\al}|^2u_{\al}-|v_{\al}|^2v_{\al}\big)(s) \ds \\
					\label{I1-2}
					&+\frac{i}{2}\al^{-1}\int_{0}^{t}	e^{-\frac{i}{2}(t-s)\De} w_{\al}v_{\al}(s)  \ds \\
					\label{I1-3}
					&+\frac{i}{2}\al^{-1}\int_{0}^{t}	e^{-\frac{i}{2}(t-s)\De} w_{\al}\Phi_{\al}(s)  \ds \\
					\label{I1-4}
					&+\frac{i}{2}\al^{-1}\int_{0}^{t}	e^{-\frac{i}{2}(t-s)\De} \Psi_{\al}v_{\al}(s)  \ds \\	
					\label{I1-5}
					&+\frac{i}{2}\al^{-1}\int_{0}^{t}	e^{-\frac{i}{2}(t-s)\De} \Psi_{\al}\Phi_{\al} (s) \ds.
				\end{align}
			\end{subequations}
			Next, we will consider each term individually.
			
			\textit{Estimate for \eqref{I1-1}.} By Lemma \ref{lem:stri}, we have
			\begin{align}
		\label{est-Phi-cubic1}
	\|\eqref{I1-1}\|_{\mathcal{S}^{\gamma}(I^{\al})}
		&\les \||\na|^{\gamma}\big[\Phi_{\al}\big(|\Phi_{\al}|^2+|v_{\al}|^2\big)\big]\|_{L^{\frac{4}{3}}_{t}L^{\frac{3}{2}}_x(I^{\al})}.
			\end{align}
			For $\gamma\in[0,\frac{1}{2}]$, we note that  $(\infty,\frac{6}{1+2\gamma})$ is $\dot{H}^{1-\gamma}$-admissible, $(8, \frac{12}{5-4\gamma})$ is $\dot{H}^{\gamma}$-admissible, and $(\frac{8}{3}, 12)$ is $\dot{H}^{\frac{1}{2}}$-admissible. By Lemma \ref{lem:stri}, \ref{lem:kato-Ponce}, and \eqref{rel:Salf-f}, we have
			\begin{align}
				\label{est-Phi-cubic term1}
			\eqref{est-Phi-cubic1}
				&\les \||\na|^{\gamma}\big[\Phi_{\al}\big(|\Phi_{\al}|^2+|v_{\al}|^2\big)\big]\|_{L^{\frac{4}{3}}_{t}L^{\frac{3}{2}}_x(I^{\al})}\notag\\
				&\les \||\na|^{\gamma}\Phi_{\al}\|_{L^{\infty}_{t}L^{2}_x(I^{\al})}\big[\|\Phi_{\al}\|^2_{L^{\frac{8}{3}}_{t}L^{12}_x(I^{\al})}+\|v_{\al}\|^2_{L^{\frac{8}{3}}_{t}L^{12}_x(I^{\al})}\big]\notag\\
				&\quad +|I^\al|^{\frac{1}{4}}\||\na|^{\gamma}v_{\al}\|_{L^{\infty}_{t}L^{\frac{6}{1+2\gamma}}_x(I^{\al})}\|v_{\al}\|_{L^{\frac{8}{3}}_{t}L^{12}_x(I^{\al})}\|\Phi_{\al}\|_{L^{8}_{t}L^{\frac{12}{5-4\gamma}}_x(I^{\al})}\notag\\
				&\les \|\Phi_{\al}\|_{\mS^{\gamma}(I^{\al})}\big[\|\Phi_{\al}\|^2_{\mS^{\frac{1}{2}}(I^{\al})}
				+\|v\|^2_{L^{\frac{8}{3}}_{t}L^{12}_x(I)}+T^{\frac{1}{4}}\|v\|_{L^{\frac{8}{3}}_{t}L^{12}_x(I)}\|v\|_{\mS^1(I)}\big].
			\end{align}
			We note that $(8,12)$ is $\dot{H}^{1}$-admissible, using H\"{o}lder's inequality and \eqref{lem:est-v_al}, we have  
			\begin{align*}
				\|v\|_{L^{\frac{8}{3}}_{t}L^{12}_x(I)}&\les T^{\frac{1}{4}}	\|v\|_{L^{8}_{t}L^{12}_x(I)}\les T^{\frac{1}{4}}\|v\|_{\mS^{1}(I)}\leq C(\|v_0\|_{H^1}) T^{\frac{1}{4}},
			\end{align*}
			then we have 
			\begin{align}
				\label{est-Phi-cubic term01}
				\|\eqref{I1-1}\|_{\mathcal{S}^{\gamma}(I^{\al})}
				\leq 
				C(\|v_0\|_{H^1})  \|\Phi_{\al}\|_{\mS^{\gamma}(I^{\al})}\Big(\|\Phi_{\al}\|^2_{\mS^{\frac{1}{2}}(I^{\al})}
				+T^{\frac{1}{2}}\Big).
			\end{align}
%
			\textit{Estimate for \eqref{I1-2}.}
			Note that $w_\al = \frac12 w_{\al,+} + \frac12 w_{\al,-}$. We only deal with the term $w_{\al,+}$ for brevity, and the method for $w_{\al,-}$ is similar. Taking the Fourier transform of $\eqref{I1-2},$ 
			\begin{align*}
				\F \big(e^{\frac{i}{2}t\De} \eqref{I1-2}\big)(\xi)
				&=	\frac{i}{2}\al^{-1}
				\int_{0}^{t}
				\int_{\R^3}	e^{is\frac{|\xi|^2}{2}} 	e^{-is|\xi_1|} \widehat{w_{\al,+}(0)}(\xi_1)\hat{v}_\al(\xi-\xi_1) \dxio\ds.
			\end{align*}
			Then using integration-by-parts, we further get that 
			\begin{equation}
				\label{F1}
				\begin{aligned}
					\F \big(e^{\frac{i}{2}t\De}\eqref{I1-2}\big)(\xi) =&-\frac{1}{2}\al^{-1}
					\int_{\R^3} e^{is\frac{|\xi|^2}{2}}\frac{1 }{|\xi_1|}  \wh{w}_{\al,+}(\xi_1)\wh{v}_\al(\xi-\xi_1)\dxio\Big|_{s=0}^{s=t} 
					\\
					&	+\frac{i}{2}\al^{-1}
					\int_{0}^{t}
					\int_{\R^3}	e^{is\frac{|\xi|^2}{2}}\frac{|\xi|^2 }{2|\xi_1|} 
					\wh{w}_{\al,+}(\xi_1)\wh{v}_\al(\xi-\xi_1)\dxio\ds
					\\
					&+\frac{1}{2}\al^{-1}
					\int_{0}^{t}
					\int_{\R^3}	e^{is\frac{|\xi|^2}{2}}\frac{1}{|\xi_1|} 
					\wh{w}_{\al,+}(\xi_1)\partial_s\wh{v}_\al(\xi-\xi_1)\dxio\ds.
				\end{aligned}
			\end{equation}
			We now take the inverse Fourier transform of \eqref{F1} to obtain
			\begin{align}
				\label{lowterm-B}
				\eqref{I1-2}=&-\frac{1}{2}\al^{-1}\Big[e^{-\frac{i}{2}(t-s)\Delta}|\nabla|^{-1}w_{\al,+}\, v_{\al} (s,x)\Big]\Big|_{s=0}^{s=t}
				\\
				\label{lowterm-IN1}
				&+\frac{i}{4}\al^{-1}\int_{0}^{t}e^{-\frac{i}{2}(t-s)\Delta}\De \big(|\nabla|^{-1}w_{\al,+} \, v_\al \big)(s,x)\ds
				\\
				\label{lowterm-IN2}
				&+\frac{1}{2}\al^{-1}\int_{0}^{t}e^{-\frac{i}{2}(t-s)\Delta} \big(|\nabla|^{-1}w_{\al,+}\, \partial_sv_\al \big) (s,x)\ds.
			\end{align}
			Then using Lemma \ref{lem:stri}, we have that 
			\begin{align*}
				\|\eqref{lowterm-B}\|_{\mS^{\gamma}(I^{\al})}
				\les	& \al^{-1}\left[\big\|e^{-\frac{i}{2}t\Delta }|\nabla|^{-1}w_\al v_{\al} (0)\big\|_{\mS^{\gamma}(I^\al)}
				+\left\||\nabla|^{-1}w_\al v_{\al}(t)\right\|_{\mS^{\gamma}(I^\al)}\right]
				\notag\\
				\les&\al^{-1} \left[\left\||\nabla|^{\gamma}\big(|\nabla|^{-1}w_\al v_{\al}\big)\right\|_{L^{\infty}_{t}L^2_x(I^\al)}+\left\||\nabla|^{\gamma}\big(|\nabla|^{-1}w_\al v_{\al}\big)\right\|_{ L^{2}_{t}L^6_x(I^\al)}\right]
				\notag\\
				\les&\al^{-1}\left\| |\nabla|^{\gamma-1}w_{\al}\right\|_{{L^{\infty}_{t}L^{8}_x}(I^\al)}\big[\left\|v_{\al} \right\|_{L^{\infty}_{t}L^{\frac{8}{3}}_x(I^\al)}+\left\|v_{\al} \right\|_{L^{2}_{t}L^{24}_x(I^\al)}\big]
				\notag\\
				&+\al^{-1}\left\| |\nabla|^{-1}w_{\al}\right\|_{{L^{\infty}_{t}L^{8}_x}(I^\al)}\big[\left\||\na|^{\gamma}v_{\al} \right\|_{L^{\infty}_{t}L^{\frac{8}{3}}_x(I^\al)}+\left\||\na|^{\gamma}v_{\al} \right\|_{L^{2}_{t}L^{24}_x(I^\al)}\big]
	,
			\end{align*}
		we note that  $(\infty,\frac{8}{3})$ is $\dot{H}^{\frac{3}{8}}$-admissible, and $(2,24)$ is $\dot{H}^{\frac{3}{8}}$-admissible,	using \eqref{lem:est-v_al}, \eqref{lem:est:w_al(in,2)}, and Sobolev's embedding $\dot{H}^{\frac{9}{8}}\hookrightarrow L^{8}$ in $x$, we have
			\begin{align}
				\label{estlowterm-B}
				\|\eqref{lowterm-B}\|_{\mS^{\gamma}(I_{\al})}
				\leq& C(\|v_0\|_{H^{1+\gamma}})
				\al^{-\frac{1}{2}-\gamma}\|(w_0,|\al\na|^{-1}w_1)\|_{H^{1+\gamma}\times H^{1+\gamma}}.
			\end{align}
Similarly,  we note that $(\frac{4}{\varepsilon},\frac{6}{3-\varepsilon})$ is $L^2$-admissible with $0<\varepsilon\ll 1$, we have
\begin{equation}
	\label{est:lowin1}
\begin{aligned}
&\quad\|\eqref{lowterm-IN1}+\eqref{lowterm-IN2}\|_{\mS^{\gamma}(I^{\al})}\\
&\les \al^{-1}\left\||\nabla|^{2+\gamma}\big(|\nabla|^{-1}w_\al v_{\al}\big)\right\|_{L^{1}_t L^{2}_x(I^\al)}
+\al^{-1}\left\||\nabla|^{\gamma}\big(|\nabla|^{-1}w_\al\p_tv_{\al}\big)\right\|_{L^{1}_t L^{2}_x(I^\al)}
	\\
&\les \al^{-1}|I^{\al}|^{\frac{2}{3}}\left\| |\nabla|^{1+\gamma}w_\al\right\|_{L^{\infty}_{t}L^{2}_x(I^\al)}\left\|v_{\al} \right\|_{L^{3}_{t}L^{\infty}_x(I^\al)}\\
&\quad+\al^{-1}|I^{\al}|^{\frac{2+\varepsilon}{4}}\left\| |\nabla|^{-1+\gamma}w_\al\right\|_{L^{\frac{2}{1-\varepsilon}}_{t}L^{\frac{6}{\varepsilon}}_x(I^\al)}\left\|\p_tv_{\al} \right\|_{L^{\frac{4}{\varepsilon}}_{t}L^{\frac{6}{3-\varepsilon}}_x(I^\al)}
		\\
&\quad+
\al^{-1}|I^{\al}|^{\frac{2+\varepsilon}{4}}\left\| |\nabla|^{-1}w_\al\right\|_{L^{\frac{2}{1-\varepsilon}}_{t}L^{\frac{6}{\varepsilon}}_x(I^\al)}\left\||\na|^{2+\gamma}v_{\al}+|\na|^{\gamma}\p_tv_{\al} \right\|_{L^{\frac{4}{\varepsilon}}_{t}L^{\frac{6}{3-\varepsilon}}_x(I^\al)}.
\end{aligned}
\end{equation}
	Using the equation of $v_\al$ in \eqref{equ-v_al}, we have
			\begin{align}
				\label{est:ptval}
				\|\partial_tv_\al\|_{\mS^{k}}\les \|v_\al\|_{\mS^{k+2}}+\|v_\al\|_{\mS^{k+2}}\|v_\al\|^2_{\mS^{\frac{1}{2}}}
				\leq C(\|v_0\|_{H^1}) \|v_\al\|_{\mS^{k+2}}.
			\end{align}
 We note that $(3,\frac{18}{5})$ is $L^2$-admissible, by Sobolev's embedding  $\dot{W}^{\varepsilon,\frac{2}{\varepsilon}}\hookrightarrow L^{\frac{6}{\varepsilon}}$, $W^{\frac{5}{6}+\varepsilon,\frac{18}{5}}\hookrightarrow L^{\infty}$ in $x$, \eqref{lem:est-v_al}, \eqref{lem:est:w_al(2,in)}, and \eqref{est:ptval}, we have
			\begin{align}
				\label{est-lowterm-IN}
				&\quad\|\eqref{lowterm-IN1}+\eqref{lowterm-IN2}\|_{\mS^{\gamma}(I^{\al})}\notag\\
				&\leq 	C\al^{-1}|I^{\al}|^{\frac{2}{3}}\left\| |\nabla|^{1+\gamma}w_\al\right\|_{L^{\infty}_{t}L^{2}_x(I^\al)}\left\|v_{\al} \right\|_{L^{3}_{t}L^{\infty}_x(I^\al)}
				\notag\\
				&\quad+C(\|v_0\|_{H^1}) \al^{-1}|I^{\al}|^{\frac{2+\varepsilon}{4}}\left\||\na|^{-1+\gamma+\varepsilon} w_\al\right\|_{L^\frac{2}{1-\varepsilon}_tL^{\frac{2}{\varepsilon}}_x(I^\al)}\left\|v_{\al} \right\|_{\mS^{2}(I^\al)}\notag\\
				&\quad+ C(\|v_0\|_{H^1}) \al^{-1}|I^{\al}|^{\frac{2+\varepsilon}{4}}\left\||\na|^{-1+\varepsilon} w_\al\right\|_{L^\frac{2}{1-\varepsilon}_tL^{\frac{2}{\varepsilon}}_x(I^\al)}\left\|v_{\al} \right\|_{\mS^{2+\gamma}(I^\al)}\notag\\
				&\leq C(\|v_0\|_{ H^1})\|(w_0,|\al\na|^{-1}w_1)\|_{H^{1+\gamma}\times H^{1+\gamma}}\big[\al^{-\frac{1}{2}-\gamma}+\al^{-1-\gamma+\frac{\varepsilon}{2}}\left\|v \right\|_{X^{2+\gamma}(I)}\big].
			\end{align}
			Combining \eqref{estlowterm-B} and  \eqref{est-lowterm-IN},  we obtain 
			\begin{align}
				\label{est-qua-inte}
				\|\eqref{I1-2}\|_{\mS^{\gamma}(I^\al)}
				\leq&C(\|v_0\|_{ H^{1+\gamma}})\|(w_0,|\al\na|^{-1}w_1)\|_{H^{1+\gamma}\times H^{1+\gamma}}\al^{-\frac{1}{2}-\gamma}\big[1+\al^{\frac{-1+\varepsilon}{2}}\left\|v \right\|_{X^{2+\gamma}(I)}\big].
			\end{align}
			
			\textit{Estimate for \eqref{I1-3}-\eqref{I1-5}}. Using Lemma \ref{lem:stri}, Lemma \ref{lem:kato-Ponce}, and \eqref{lem:est:w_al(in,2)}, we have
\begin{align}
\label{est-qua21}
\|\eqref{I1-3}\|_{\mS^{\gamma}(I^{\al})}
\lesssim
& \al^{-1}\big\||\na|^{\gamma}\big(w_{\al}\Phi_{\al}\big)\big\|_{L^{\frac{4}{3}}_t L^{\frac{3}{2}}_x(I^\al)}\notag\\
\les & \al^{-1}|I^{\al}|^{\frac{1+2\gamma}{4}}\left\||\na|^{\gamma}w_{\al}\right\|_{L^{\infty}_t L^{2}_x(I^\al)}
\left\|\Phi_{\al}\right\|_{L^{\frac{2}{1-\gamma}}_t L^{6}_x(I^\al)}
				\notag\\
				&
				+\al^{-1}|I^{\al}|^{\frac{1+2\gamma}{4}}\left\|w_{\al}\right\|_{L^{\infty}_t L^{\frac{6}{3-2\gamma}}_x(I^\al)}
				\left\||\na|^{\gamma}\Phi_{\al}\right\|_{L^{\frac{2}{1-\gamma}}_t L^{\frac{6}{1+2\gamma}}_x(I^\al)}
				\notag\\
				\leq  &CT^{\frac{1+2\gamma}{4}}\|(w_0,|\al\na|^{-1}w_1)\|_{H^\gamma\times H^\gamma}\left\|\Phi_{\al}\right\|_{\mS^{\gamma}(I^{\al})},
			\end{align}
			where we note that $(\frac{2}{1-\gamma},\frac{6}{1+2\gamma})$ is $L^2$-admissible.
 Similarly, by \eqref{lem:est-v_al}, we have
\begin{align}
\label{est-qua31}
\|\eqref{I1-4}\|_{\mS^{\gamma}(I^{\al})}
\les &\al^{-1}
	\big\||\na|^{\gamma}\big(v_{\al}\Psi_{\al}\big)\big\|_{L^{2}_t L^{\frac{6}{5}}_x(I^\al)}\notag\\
				\les&
				\al^{-1}|I^{\al}|^{\frac{3}{4}}\left\||\na|^{\gamma}v_{\al}\right\|_{L^{\infty}_t L^{\frac{6}{1+2\gamma}}_x(I^\al)}\left\|\Psi_{\al}\right\|_{L^{\infty}_t L^{\frac{6}{3-2\gamma}}_x(I^\al)}\notag\\
				&+\al^{-1}|I^{\al}|^{\frac{3}{4}}\left\|v_{\al}\right\|_{L^{\infty}_t L^{6}_x(I^\al)}\left\||\na|^{\gamma}\Psi_{\al}\right\|_{L^{\infty}_t L^{2}_x(I^\al)}
				\notag\\
				\leq &C(\|v_0\|_{H^1})T^{\frac{1}{2}}\left\|\Psi_{\al}\right\|_{\mW^{\gamma}(I^{\al})},
			\end{align}
			where we note that $(\infty,\frac{6}{1+2\gamma})$ is $\dot{H}^{1-\gamma}$-admissible.
By Lemma \ref{lem:stri}, we have 		
			\begin{align}
				\label{est-qua4}
				\|\eqref{I1-5}\|_{\mS^{\gamma}(I^{\al})}
				\les &\al^{-1}
				\left\||\na|^{\gamma}\big(\Phi_{\al}\Psi_{\al}\big)\right\|_{L^{2}_t L^{\frac{6}{5}}_x(I^\al)}
				\notag\\
				\les& \al^{-1}|I^{\al}|^{\frac{1}{2}}\big[\left\|\Phi_{\al}\right\|_{L^{\infty}_t \dot{H}^\gamma_x(I^\al)}\left\|\Psi_{\al}\right\|_{L^{\infty}_t L^{3}_x(I^\al)}+\left\|\Phi_{\al}\right\|_{L^{\infty}_t L^{3}_x(I^\al)}\left\|\Psi_{\al}\right\|_{L^{\infty}_t \dot{H}^\gamma_x(I^\al)}\big]
				\notag\\
				\leq&CT^{\frac{1}{2}}\big\|\big(\Phi_{\al},\Psi_\al\big)\big\|_{\mS^{\gamma}\times \mW^{\gamma}(I^{\al})}\big\|\big(\Phi_{\al},\Psi_\al\big)\big\|_{\mS^{\frac{1}{2}}\times \mW^{\frac{1}{2}}(I^{\al})}.
			\end{align}
Therefore, combining \eqref{est-Phi-cubic term01}, \eqref{est-qua-inte}, \eqref{est-qua21}, \eqref{est-qua31}, and \eqref{est-qua4}, we obtain
\begin{align*}
\left\|\Phi_{\al}\right\|_{\mS^{\gamma}(I^\al)} &\leq 
		C_2\|(w_0,|\al\na|^{-1}w_1)\|_{H^{1+\gamma}\times H^{1+\gamma}}\big[\al^{-\frac{1}{2}-\gamma}+\al^{-1-\gamma+\frac{\varepsilon}{2}}\left\|v \right\|_{X^{2+\gamma}(I)}\big]  
\\
				&\quad +C_2\|(\Phi_{\al},\Psi_\al)\|_{\mS^{\gamma}\times \mW^{\gamma} (I^{\al})}\Big(\|\Phi_{\al}\|^2_{\mS^{\frac{1}{2}}(I^\al)}
				+T^{\frac{1}{2}}\big\|\big(\Phi_{\al},\Psi_\al\big)\big\|_{\mS^{\frac{1}{2}}\times \mW^{\frac{1}{2}}(I^{\al})}+T^{\frac{1+2\ga}{4}}\Big),
			\end{align*}
where $C_2$ is a constant only depending on $\|(v_0,w_0,|\al\na|^{-1}w_1)\|_{H^{1+\gamma}\times H^{1+\gamma}\times H^{1+\gamma}}.$ 
Thus, we complete the proof of Lemma \ref{lem:nonlinear-estimates-phi}.	
	\end{proof}

Combining Lemma  \ref{lem:nonlinear-estimates-psi} and \ref{lem:nonlinear-estimates-phi}, we complete the proof of Proposition \ref{prop:estPsiPhi}.

\subsection{ Proof of Theorem \ref{mainthm}}	
We begin the proof by establishing the nonlinear estimates of $\|(r,q)\|_{\mS^{\frac{1}{2}}\times\mW^{\frac{1}{2}}(I)} = \|\big(\Phi_\al,\Psi_\al\big)\|_{\mS^{\frac12}\times\mW^{\frac12}(I^\al)}$, which is scaling invariant with respect to the transform $S_\al$ in \eqref{rel:Salf-f}.
\begin{lem}
\label{lem:estPhiPsi-12}
Let Assumption \ref{assu:main} hold.
Then, there exists $\al_0>0$ sufficiently large such that for $\al>\al_0, $
\begin{equation}
\label{est:12H3}
\begin{aligned}
& \|(r,q)\|_{\mS^{\frac{1}{2}}\times\mW^{\frac{1}{2}}(I)} 
\leq D\big(\|\vec{u}_0\|_{\HH^3},\|(v_0,w_0,|\al\na|^{-1}w_1)\|_{H^{\frac{5}{2}}\times H^{\frac{3}{2}}\times H^{\frac{3}{2}}}\big)\al^{-\frac{1}{2}},
\end{aligned} 
\end{equation}
 where $D$ is a constant depending on $\big\|\vec u_0\big\|_{\HH^{3}}$ and $\big\|(v_0, w_0, |\alpha\nabla|^{-1}w_1)\big\|_{H^{\frac{5}{2}} \times H^{\frac{3}{2}}\times H^{\frac{3}{2}} }$.
\end{lem}
\begin{proof}
 Let $(\gamma,\kappa_1, \kappa_2)=(\frac{1}{2},\frac{1}{2},-\frac{1}{2})$ in Proposition \ref{prop:estPsiPhi}. From Lemma \ref{lem:local-theory}, we have 
\begin{equation}
	\label{1/2}
\begin{aligned}
&\quad\|\big(\Phi_\al,\Psi_\al\big)\|_{\mS^{\frac{1}{2}}\times\mW^{\frac{1}{2}}(I^\al)}\notag\\
&\leq 
D_1\big(\|\vec{u}_0\|_{\HH^3},\|(w_0,|\al\na|^{-1}w_1)\|_{ H^{\frac{3}{2}}\times H^{\frac{3}{2}}}\big)\al^{-\frac{1}{2}}\\
&\quad
+D_2\big(\|(v_0,w_0,|\al\na|^{-1}w_1)\|_{H^\frac{5}{2}\times H^{\frac{3}{2}}\times H^{\frac{3}{2}}}\big)\al^{-1}  
	\\
&\quad + D_3 \big(\|(v_0,w_0,|\al\na|^{-1}w_1)\|_{H^\frac{3}{2}\times H^{\frac{3}{2}}\times H^{\frac{3}{2}}}\big) \|(\Phi_{\al},\Psi_\al)\|_{\mS^{\frac{1}{2}}\times \mW^{\frac{1}{2}} (I^{\al})}\\
&\qquad\qquad\qquad\qquad\qquad\qquad\qquad\qquad
\cdot\Big[T^{\frac{1}{2}}+
T^{\frac{1}{2}}\big\|\big(\Phi_{\al},\Psi_\al\big)\big\|_{\mS^{\frac{1}{2}}\times \mW^{\frac{1}{2}}(I^{\al})}
+\|\Phi_{\al}\|^2_{\mS^{\frac{1}{2}}(I^\al)}\Big].
\end{aligned}
\end{equation}
Denote 
\begin{align*}
c_1\coloneqq & D_1\big(\|\vec{u}_0\|_{\HH^3},\|(w_0,|\al\na|^{-1}w_1)\|_{ H^{\frac{3}{2}}\times H^{\frac{3}{2}}}\big)\al^{-\frac{1}{2}},
\\
c_2\coloneqq &D_2\big(\|(v_0,w_0,|\al\na|^{-1}w_1)\|_{H^\frac{5}{2}\times H^{\frac{3}{2}}\times H^{\frac{3}{2}}} \big)\al^{-1} ,
\\
a_0\coloneqq  &2(c_1+c_2).
\end{align*}
Assume that $$\big\|\big(\Phi_{\al},\Psi_\al\big)\big\|_{\mS^{\frac{1}{2}}\times \mW^{\frac{1}{2}}(I^{\al})}\leq a_0.$$
Note that $\al_0\gg1$, then for any $\al>\al_0$,
\EQn{
\label{1a_0upper}
a_0\le  2D_1 \al_0^{-\frac12} + 2D_2 \al_0^{-1}\le   2(D_1+D_2)\al_0^{-\frac12}.
}
We may take $\al_0$ suitably such that $a_0\ll1$. Moreover, we can choose $T$ suitably small and large $\al_0$ such that
\EQn{
\label{a_0upper}
D_3T^{\frac{1}{2}}+D_3 T^{\frac{1}{2}}a_0+D_3a_0^2\le & 2D_3T^{\frac{1}{4}}+D_3a_0\\ 
\le & 2D_3T^{\frac{1}{4}}+2D_3 \mbrkb{ (D_1+D_2)\al_0^{-\frac12}}\\
\leq& \frac{1}{4}.
}
Therefore, by \eqref{a_0upper}, we have
\EQ{
		&\quad\|\big(\Phi_\al,\Psi_\al\big)\|_{\mS^{\frac{1}{2}}\times\mW^{\frac{1}{2}}(I^\al)}\notag\\
		&\leq 
		D_1\al^{-\frac{1}{2}} 
		+D_2 \al^{-1}  + \mbrkb{D_3 T^{\frac{1}{2}} + D_3T^{\frac12}a_0+D_3a_0^2} \|(\Phi_{\al},\Psi_\al)\|_{\mS^{\frac{1}{2}}\times \mW^{\frac{1}{2}} (I^{\al})},
}
which implies
\begin{align*}
\|\big(\Phi_\al,\Psi_\al\big)\|_{\mS^{\frac{1}{2}}\times\mW^{\frac{1}{2}}(I^\al)}
&\leq\frac{4}{3} c_1 + \frac43 c_2 < a_0.
\end{align*}
Then \eqref{est:12H3} follows by scaling $\|(r,q)\|_{\mS^{\frac{1}{2}}\times\mW^{\frac{1}{2}}(I)} = \|\big(\Phi_\al,\Psi_\al\big)\|_{\mS^{\frac12}\times\mW^{\frac12}(I^\al)}$.
\end{proof}

Next, we extend the above $\dot H^{\frac12}$-level estimate to $\|(r,q)\|_{\mS^{\gamma}\times\mW^{\gamma}(I)}$ for  $\gamma=0$.
\begin{prop}
\label{prop: est:gamma}
	Let Assumption \ref{assu:main} hold. Given  $\kappa\in\{0,1\}$, suppose that $\vec{u}_0\in\HH^{ 3+\kappa}$,
then there exists $\al_0$ suitably large such that when $\al>\al_0, $
\begin{align*}
\|q\bigr\|_{ \mW^{0}(I)}
&\leq  C
 \big(
\al^{-2}
 +  \al^{-1-\kappa}
\big) ,
\\
\|(r,q)\|_{\mS^{0}\times \mW^{0}(I)}&	\leq C
\big[\al^{-2}+\al^{-1-\kappa}+\al^{-1}\|(w_0,|\al\na|^{-1}w_1)\|_{H^{1}\times H^{1}}\big],
\end{align*}
where $C = C\big( \|\vec{u}_0\|_{\HH^{ 3+\kappa}},\ \|(v_0, w_0, |\alpha\nabla|^{-1}w_1)\|_{H^{2}\times H^{1}\times H^{1}} \big).$
\end{prop}
\begin{proof}		
Setting $(\gamma,\kappa_1,\kappa_2)=(0,0,\kappa)$ in Proposition \ref{prop:estPsiPhi}, using Lemma \ref{lem:local-theory}, we have
 \begin{align*}
&\quad
\bigl\|\big(\Phi_\al,\Psi_\al\big)\bigr\|_{\mS^{0}\times \mW^{0}(I^\al)}\\
&\leq
C_1\big(\|\vec{u}_0\|_{\HH^{ 3 +\kappa}},\|(w_0,|\al\na|^{-1}w_1)\|_{H^{1}\times H^{1}}\big)\big[\al^{-\frac{3}{2}}+\al^{-\frac{1}{2}-\kappa}\big]
\\
&\quad+ C_2\big(\|(v_0,w_0,|\al\na|^{-1}w_1)\|_{H^{2}\times H^{1}\times H^{1}}\big)\|(w_0,|\al\na|^{-1}w_1)\|_{ H^{1}\times H^{1}}\alpha^{  -\frac{1}{2}}\\
&\quad+C_3\big(\|(v_0,w_0,|\al\na|^{-1}w_1)\|_{H^{1}\times H^{1}\times H^{1}}\big)\|(\Phi_{\al},\Psi_\al)\|_{\mS^{0}\times \mW^{0} (I^{\al})}
\\
&\qquad\qquad\qquad\qquad\qquad\qquad\qquad\qquad
\cdot\Big(\|\Phi_{\al}\|^2_{\mS^{\frac{1}{2}}(I^\al)}
+T^{\frac{1}{2}}\big\|\big(\Phi_{\al},\Psi_\al\big)\big\|_{\mS^{\frac{1}{2}}\times \mW^{\frac{1}{2}}(I^{\al})}+T^{\frac{1}{4}}\Big).
 \end{align*}
From Lemma \ref{lem:estPhiPsi-12}, we note that  $\al_0\gg1,$ then for any $\al>\al_0$,
\begin{align*}
&\quad\|\Phi_{\al}\|^2_{\mS^{\frac{1}{2}}(I^\al)}+T^{\frac{1}{2}}\big\|\big(\Phi_{\al},\Psi_\al\big)\big\|_{\mS^{\frac{1}{2}}\times \mW^{\frac{1}{2}}(I^{\al})}+T^{\frac{1}{4}}
\\
&\leq D\big(\|\vec{u}_0\|_{\HH^3},\|(v_0,w_0,|\al\na|^{-1}w_1)\|_{H^{\frac{5}{2}}\times H^{\frac{3}{2}}\times H^{\frac{3}{2}}}\big)\al_0^{-\frac{1}{2}}(1+T^{\frac{1}{2}})+T^{\frac{1}{4}},
\end{align*}
which means we can choose suitably $T$ and $\al_0$ such that
\begin{align}
\label{bootgamma}
&\quad C_3\big[\|\Phi_{\al}\|^2_{\mS^{\frac{1}{2}}(I^\al)}+T^{\frac{1}{2}}\big\|\big(\Phi_{\al},\Psi_\al\big)\big\|_{\mS^{\frac{1}{2}}\times \mW^{\frac{1}{2}}(I^{\al})}+T^{\frac{1}{4}}\big]\notag\\
&\leq C_3\big[D\al_0^{-\frac{1}{2}}(1+T^{\frac{1}{2}})+T^{\frac{1}{4}}\big]\leq \frac{1}{4}.
\end{align}
By \eqref{bootgamma}, we obtain 
\begin{align*}
&\quad
\bigl\|\big(\Phi_\al,\Psi_\al\big)\bigr\|_{\mS^{0}\times \mW^{0}(I^\al)}\\
&\leq
\frac{4}{3}C_1\big(\|\vec{u}_0\|_{\HH^{3 +\kappa}},\|(w_0,|\al\na|^{-1}w_1)\|_{H^{1}\times H^{1}}\big)\big[\al^{-\frac{3}{2}}+\al^{-\frac{1}{2}-\kappa}\big]
\\
&\quad+\frac{4}{3} C_2\big(\|(v_0,w_0,|\al\na|^{-1}w_1)\|_{H^{2}\times H^{1}\times H^{1}}\big)\|(w_0,|\al\na|^{-1}w_1)\|_{ H^{1}\times H^{1}}\alpha^{-\frac{1}{2} }.
\end{align*}
Scaling back, we have 
 \begin{align}
 \label{estgamma0,1}
 \bigl\|(r,q)\bigr\|_{\mS^{0}\times \mW^{0}(I)} 
 &\leq 
C_1\big(\|\vec{u}_0\|_{\HH^{ 3 +\kappa}},\|(w_0,|\al\na|^{-1}w_1)\|_{H^{1}\times H^{1}}\big) \big[
 \al^{-2}+\al^{-1-\kappa}\big]
 \notag\\
 &\quad+ C_2\big(\|(v_0,w_0,|\al\na|^{-1}w_1)\|_{H^{2}\times H^{1}\times H^{1}}\big)\|(w_0,|\al\na|^{-1}w_1)\|_{ H^{1}\times H^{1}}\alpha^{-1}.
  \end{align}
We consolidate all constants into a single one, denoted by $C$, which depends on $\|\vec{u}_0\|_{\HH^{3+\kappa}}$ and $\|(v_0,w_0,|\al\na|^{-1}w_1)\|_{H^{2} \times H^{1}\times H^{1}}$. 
Then scaling back, we have
\begin{align*}
	\|(r,q)\|_{\mS^{0}\times \mW^{0}(I)}
	\leq& 
C\big[\al^{-2}+\al^{-1-\kappa}+\al^{-1}\|(w_0,|\al\na|^{-1}w_1)\|_{H^{1}\times H^{1}}\big].
\end{align*}
Finally, scaling back the estimates of $\Psi_\al$ in Lemma \ref{lem:nonlinear-estimates-psi} with $(\gamma,\kappa_1,\kappa_2)=(0,0,\kappa)$ and using Lemma \ref{lem:local-theory}, we have
\begin{align*}
\|q\|_{\mW^{0}(I)} 
&\leq
C\big(
\al^{-2}+\al^{-1-\kappa}\big),
\end{align*}
thus we complete the proof.
\end{proof}

\begin{proof}[Conclude the proof of the main theorem]
With the above proposition, the proof of Theorem \ref{mainthm} is straightforward. 
Recall the initial conditions 
$$ v_0 = u_0,  \quad  w_0=n_0+|u_0|^2, \quad  w_1=n_1+\partial_t|u|^2(0),$$
in Assumption \ref{assu:main}.
For the wave component, applying the first estimate in Proposition \ref{prop: est:gamma} with \(\kappa=0\) and \(\kappa=1\), respectively, yields  \eqref{thm:qconvergencerate1} and  \eqref{thm:qconvergencerate2}. For the Schr\"odinger component, applying the second estimate in Proposition \ref{prop: est:gamma} with \(\kappa=0\) yields \eqref{thm:rconvergencerate1}.
 Finally, if $n_0+|u_0|^2=0$ and $|\na|^{-1}n_1\in H^{ 1}$, then
 \begin{align*}
\|(w_0,|\al\na|^{-1}(n_1+\partial_t|u|^2(0)))\|_{H^{1}\times H^{1}}
&\les \al^{-1}\big(\||\na|^{-1}n_1\|_{H^{1}}+\|\na u_0 \bar{u}_0\|_{H^{1}}\big)\\
&\les \al^{-1}(\||\na|^{-1}n_1\|_{H^{1}}+\|u_0\|^2_{H^{2}}),
 \end{align*}
 where we use the unscaled form of\eqref{ptu2} in the last step. Applying the second estimate in
 Proposition \ref{prop: est:gamma}
with $\kappa=1$, we can obtain \eqref{thm:rconvergencerate2} in  Theorem \ref{mainthm}. 
\end{proof}

		\vspace{1cm}
			\section{Appendix}
			\label{sec:appendix}
			\subsection{Proof of Lemma \ref{lem:bi-W-N}}
		\begin{proof}
			Given a generic $F(t,x)\in L^2_{t,x}(\R\times \R^3),$ let 
			\begin{align*}
				\mathcal{I}\coloneqq \int\int_{\R\times \R^3} F(t,x) [e^{\pm it|\na|}f_M](x)[e^{-\frac{1}{2}it\De}g_N](x)\dx\dt.
			\end{align*}
		Using Plancherel's Theorem and the sifting property of Dirac function $\delta(x),$ we have
		\begin{align*}
			\mathcal{I}&= \int_{\R^3\times\R^3}\int_{\R\times \R} \hat{F}(\tau,\xi)e^{it\tau} e^{\pm it|\xi-\eta|}\hat{f}_M(\xi-\eta)e^{\frac{1}{2}it|\eta|^2}\hat{g}_N(\eta)\dta \dt\dxi\deta\\
			&= \int_{\R^3\times\R^3}\int_{\R} \hat{F}(\tau,\xi)\int_{\R} e^{it(\pm|\xi-\eta|+\frac{1}{2}|\eta|^2+\tau)}\dt\dta\hat{f}_M(\xi-\eta)\hat{g}_N(\eta)\dxi\deta
			\\
			&= \int_{\R^3\times\R^3}\int_{\R} \hat{F}(\tau,\xi)\delta_{\mp|\xi-\eta|-\frac{1}{2}|\eta|^2}(\tau)\dta\hat{f}_M(\xi-\eta)\hat{g}_N(\eta)\dxi\deta
			\\
			&=\int_{\R^3\times\R^3} \hat{F}(\mp|\xi|-\frac{1}{2}|\eta|^2,\xi+\eta)\hat{f}_M(\xi)\hat{g}_N(\eta)\dxi\deta.
		\end{align*}
		Changing variables according to 
		$$\xi=(\xi_1,\xi_2,\xi_3)=(\xi_1,\xi'), \quad \nu=\xi+\eta\in \R^3,\mbox{ and } a=\mp|\xi|-\frac{1}{2}|\eta|^2\in \R.$$
		The corresponding Jacobian denoted by $J$ satisfies
		\begin{align*}
			J^{-1}=\frac{\p(a,\xi',\nu)}{\p(\xi,\eta)}=
				\begin{pmatrix}
				\mp\frac{\xi_1}{|\xi|} & \mp\frac{\xi_2}{|\xi|} &\mp\frac{\xi_3}{|\xi|}&-\eta_1&-\eta_2&-\eta_3
				\\
				0&1&0&0&0&0  \\
				0&0&1&0&0&0  \\
					1&0&0&1&0&0  \\
						0&1&0&0&1&0  \\
						0&0&1&0&0&1  \\
			\end{pmatrix},
		\end{align*} 
		where the determinant of the Jacobian matrix satisfies $|J|^{-1}=\mp\frac{\xi_1}{|\xi|}+\eta_1.$

		 We now justify the lower bound for the Jacobian. Let
		$	A:=\mp\frac{\xi}{|\xi|}+\eta.$
		Since \(|\eta|\sim N\) and \(N\nsim1\), we have
		\[
		|A|
		\ge \bigl||\eta|-1\bigr|
		\gtrsim |1-N|.
		\]
		We decompose the frequency support into the three regions where
		\(|A_j|\) is maximal, \(j=1,2,3\). On each such region,
		\[
		|A_j|\gtrsim |A|\gtrsim |1-N|.
		\]
		It is enough to consider the region where \(j=1\), since the other two
		regions are treated in the same way after a permutation of the coordinate
		axes. Therefore, on this region,
		\[
		\left|\mp\frac{\xi_1}{|\xi|}+\eta_1\right|
		\gtrsim |1-N|,
		\]
		and hence
		\[
		|J|\lesssim |1-N|^{-1}.
		\]
		Then using H\"older's inequality, we have
		\begin{align}
			\label{appen-est:I}
			|\mathcal{I}|&=\left|\int_{\mathbb{R}^2\times\mathbb{R}^3\times  \mathbb{R}} \hat{F}(a,\nu)\hat{f}_M(\xi(a,\xi',\nu))\hat{g}_N(\eta(a,\xi',\nu))J \,da\,d\nu\,d\xi'\right|
			\notag\\
			&\lesssim \|F(a,\nu)\|_{L^{2}_{a,\nu}} \int_{\mathbb{R}^2} \left(\int_{\mathbb{R}\times \mathbb{R}^3} |\hat{f}_M(\xi(a,\xi',\nu))\hat{g}_N(\eta(a,\xi',\nu))J|^2 \,da\,d\nu\right)^{\frac{1}{2}} \,d\xi' \notag\\
			&\lesssim \|F(t,x)\|_{L^{2}_{t,x}} \left( \int_{\mathbb{R}\times \mathbb{R}^3\times \mathbb{R}^2} |\hat{f}_M(\xi(a,\xi',\nu))\hat{g}_N(\eta(a,\xi',\nu))J|^2 \,da\,d\nu\,d\xi' \right)^{\frac{1}{2}}
		 \left( \int_{\{\xi':|\xi'|\leq M\}} \,d\xi' \right)^{\frac{1}{2}} \notag\\
			&\lesssim \|F(t,x)\|_{L^{2}_{t,x}} \left( \int_{\mathbb{R}^3\times \mathbb{R}^3} |\hat{f}_M(\xi)\hat{g}_N(\eta)|^2|J| \,d\xi\,d\eta \right)^{\frac{1}{2}} 
			\left( \int_{\{\xi':|\xi'|\leq M\}} \,d\xi' \right)^{\frac{1}{2}} \notag\\
			&\lesssim M \|F(t,x)\|_{L^{2}_{t,x}} \|J\|_{L^{\infty}}^{\frac{1}{2}} \|f_M\|_{L^{2}_{x}} \|g_N\|_{L^{2}_{x}}.
		\end{align}
		By duality, we have
		\begin{align*}
			\|[e^{\pm it|\na|}f_M][e^{-\frac{1}{2}it\De}g_N]\|_{L^{2}_{t,x}}\les M  |1-N|^{-\frac{1}{2}} \|f_M\|_{L^{2}_{x}} \|g_N\|_{L^{2}_{x}}.
		\end{align*}
	The same argument, using the change of variables $(\xi,\eta)$ to  $(a,\eta',
		\nu)$, yields the bound with $N$ in place of $M$. Combining the two bounds gives the factor $\min\{M,N\}$, we thereby complete the proof.
		\end{proof}


\begin{thebibliography}{100}
			
			\bibitem{Added1988}
			H. Added and S. Added, 
			\newblock{\em Equations of Langmuir turbulence and nonlinear Schrödinger equation: smoothness and approximation}. 
			J. Funct. Anal., 79(1):183--210, 1988.
			
			
			\bibitem{Bai2025rough}
			R. Bai, Y. Lian, and Y. Wu, 
			\newblock{\em Regularization for the Schr\"odinger equation with rough potential: high-dimensional case}. 
		arXiv preprint arXiv:2510.25555, 2025.
			
			
			
			
			
			
			
   \bibitem{BoLi-KatoPonce}
   J. Bourgain and D. Li, 
   \newblock{\em On an endpoint Kato-Ponce inequality}. 
   Differential Integral Equations,
   27(11-12):1037--1072, 2014.
			
			
			
\bibitem{Bourgain1996lwp}
J. Bourgain and J. Colliander,
\newblock{\em On wellposedness of the Zakharov system}.
Int. Math. Res. Not., 1996(11):515--546, 1996.



\bibitem{upvpCandy} 
T. Candy, \newblock{\em Multi-scale bilinear restriction estimates for general phases}. Math. Ann. 375, 777--843,2019.

\bibitem{Candy2018Ann.PDE}
T. Candy, S. Herr,
\newblock{\em On the division problem for the wave maps equation}.
Ann. PDE, 4(2):Paper No. 17, 61 pp., 2018.	

\bibitem{BookNLS2003}
T. Cazenave, 
\newblock{\em Semilinear Schrödinger equations}. 
volume 10 of Courant Lecture Notes in Mathematics. New York University, Courant Institute of Mathematical Sciences, New York; American Mathematical Society, Providence, RI, 2003.
			
			\bibitem{ChenWu2021}
			Z. Chen and S. Wu, 
			\newblock{\em Local well-posedness for the Zakharov system in dimension $d = 2,3$}. 
			Commun. Pure Appl. Anal., 20(12):4307--4319, 2021.
			
			\bibitem{Coifman-Meyer}
			R. R. Coifman and Yves Meyer, 
			\newblock{\em Nonlinear Harmonic Analysis, Operator Theory and P.D.E.}. 
			Annals of Mathematics Studies, 1986.
			
			\bibitem{Gibbons_Thornhill_Wardrop_Haar_1977}
			J. Gibbons, S. G. Thornhill, M. J. Wardrop, and D. Ter Haar, 
			\newblock{\em  On the theory of Langmuir solitons}. 
			Journal of Plasma Physics, 17(2):153--170, 1977.
			
			\bibitem{Tsutsumi1997}
			J. Ginibre, Y. Tsutsumi, and G. Velo, 
			\newblock{\em On the Cauchy problem for the Zakharov system}. 
			J. Funct. Anal., 151(2):384--436, 1997.
			
			\bibitem{Gra-14}
			L. Grafakos, 
			\newblock{\em Modern fourier analysis}. 
			Graduate Texts in Mathematics, 3rd edn. Springer New York, 2014.

\bibitem{Hadac-Herr-Koch-Poincare}
M. Hadac, S. Herr, H. Koch,
\newblock{\em Well-posedness and scattering for the KP-II equation in a critical space}.
Ann. Inst. H. Poincar\'e Anal. Non Lin\'eaire, 26(3):917--941, 2009.
	
			\bibitem{Kato-Ponce}
			T. Kato and G. Ponce, 
			\newblock{\em Commutator estimates and the Euler and Navier-Stokes equations}. 
			Comm. Pure Appl. Math., 41(7):891--907, 1988.
			
			\bibitem{Koch-Tataru2005-CPAM}
			H. Koch, D. Tataru,
			\newblock{\em Dispersive estimates for principally normal pseudodifferential operators}.
			Commun. Pure Appl. Math., 58(2):217--284, 2005.
		
			\bibitem{Koch-Tataru2005-Duke}
			H. Koch, D. Tataru,
			\newblock{\em Conserved energies for the cubic nonlinear Schr{\"o}dinger equation in one dimension}.
			Duke Math. J., 167(17):3207--3313, 2018.
				
				
		\bibitem{Koch-Tataru-Visan:Book}
				H. Koch, D. Tataru, M. Visan,
				\newblock{\em Dispersive Equations and Nonlinear Waves: Generalized Korteweg-De Vries, Nonlinear Schr{\"o}dinger, Wave and Schr{\"o}dinger Maps}.
				Oberwolfach Seminars, Vol. 45, Birkh{\"a}user/Springer, Basel, 2014.
		
			\bibitem{Tao-stri}
			M. Keel and T. Tao, 
			\newblock{\em Endpoint Strichartz estimates}. 
			Amer. J. Math., 120(5):955--980, 1998.
			
			\bibitem{Kening1995}
			C. Kenig, G. Ponce, and L. Vega, 
			\newblock{\em On the Zakharov and Zakharov-Schulman systems}. 
			Journal of
			Functional Analysis, 127(1):204--234, 1995.
			


						
						
					



			
\bibitem{Li-KatoPonce}
D. Li, 
\newblock{\em On Kato-Ponce and fractional Leibniz}. 
Rev. Mat. Iberoam., 35(1):23--100, 2019.
	
	
	\bibitem{Lu2025}
				Y. Lu and Z. Zheng, 
				\newblock{\em Optimal convergence rates of the Klein-Gordon-Zakharov system in the non-relativistic limit}. 
				J. Differential Equations, 431:Paper No. 113195, 52, 2025.
				
						
\bibitem{machihara2001nonrelativistic}
S. Machihara, 
\newblock{\em The nonrelativistic limit of the nonlinear Klein-Gordon equation}. 
Funkcialaj Ekvacioj Serio Internacia, 44(2):243--252, 2001.
			
\bibitem{machihara2002nonrelativistic}
S. Machihara, K. Nakanishi, and T. Ozawa, 
\newblock{\em Nonrelativistic limit in the energy space for nonlinear Klein-Gordon equations}. 
Mathematische Annalen, 322(3):603--621, 2002.
			
\bibitem{masmoudi2002nonlinear}
N. Masmoudi and K. Nakanishi, 
\newblock{\em From nonlinear Klein-Gordon equation to a system of coupled nonlinear Schr\"odinger equations}. 
Mathematische Annalen, 324:359--389, 2002.
			
	\bibitem{masmoudi2003nonrelativistic}
	N. Masmoudi and K. Nakanishi, 
	\newblock{\em Nonrelativistic limit from Maxwell-Klein-Gordon and Maxwell-Dirac to Poisson-Schr\"odinger}. 
	International Mathematics Research Notices, 2003(13):697--734, 2003.
			
	\bibitem{Masmoudi2005}
	N. Masmoudi and K. Nakanishi, 
	\newblock{\em From the Klein-Gordon-Zakharov system to the nonlinear
	Schr\"odinger equation}. 
	J. Hyperbolic Differ. Equ., 2(4):975--1008, 2005.
			
			\bibitem{Masmoudi2008}
			N. Masmoudi and K. Nakanishi, 
			\newblock{\em Energy convergence for singular limits of Zakharov type systems}. 
			Invent. Math., 172(3):535--583, 2008.
			
			\bibitem{Masmoudi2010}
			N. Masmoudi and K. Nakanishi, 
			\newblock{\em From the Klein-Gordon-Zakharov system to a singular nonlinear
				Schr\"odinger system}. 
			Ann. Inst. H. Poincaré C Anal. Non Linéaire, 27(4):1073--1096, 2010.
			
			\bibitem{Ozawa1992}
			T. Ozawa and Y. Tsutsumi, 
			\newblock{\em The nonlinear Schr\"odinger limit and the initial layer of the Zakharov
				equations}. 
			Differential Integral Equations, 5(4):721--745, 1992.
			
			\bibitem{Sanwal2022}
			A. Sanwal, 
			\newblock{\em Local well-posedness for the Zakharov system in dimension $d \le 3$}. 
			Discrete Contin.
			Dyn. Syst., 42(3):1067--1103, 2022.
			
			
			
			
			\bibitem{Weinstein1986}
			S. Schochet and M. Weinstein, 
			\newblock{\em The nonlinear Schrödinger limit of the Zakharov equations governing Langmuir turbulence}. 
			Comm. Math. Phys., 106(4):569--580, 1986.
			
			\bibitem{zakharov1972}
			V. Zakharov, 
			\newblock{\em  Collapse of Langmuir waves}. 
			Sov. Phys. JETP, 35(5):908--914, 1972.
			
			
		\end{thebibliography}
	\end{document}